\documentclass[a4j,10pt,showkeys]{article}

\usepackage{amsmath,amsthm,amssymb}
\usepackage{geometry}
\usepackage{hyperref}
\usepackage[capitalize]{cleveref}
\usepackage{cancel}
\usepackage{mathtools}
\usepackage{bm}
\usepackage{color}
\usepackage{mathrsfs}
\usepackage[dvipsnames]{xcolor}
\hypersetup{
    colorlinks=true,
    citecolor=MidnightBlue,
    linkcolor=Red,
    urlcolor=OliveGreen,
}

\newtheorem{theo}{Theorem}[section]
\newtheorem{lemm}[theo]{Lemma}

\newtheorem{cor}[theo]{Corollary}
\theoremstyle{definition}
\newtheorem{defi}[theo]{Definition}
\newtheorem{rem}[theo]{Remark}
\newtheorem{assum}{Assumption}

\newtheorem*{SVC}{Problem (SVC)}

\newcommand{\bA}{\mathbb{A}}

\newcommand{\bE}{\mathbb{E}}
\newcommand{\bF}{\mathbb{F}}

\newcommand{\bN}{\mathbb{N}}

\newcommand{\bP}{\mathbb{P}}

\newcommand{\bR}{\mathbb{R}}
\newcommand{\bS}{\mathbb{S}}
\newcommand{\bT}{\mathbb{T}}

\newcommand{\bX}{\mathbb{X}}

\newcommand{\cD}{\mathcal{D}}

\newcommand{\cF}{\mathcal{F}}

\newcommand{\cI}{\mathcal{I}}

\newcommand{\cL}{\mathcal{L}}
\newcommand{\cM}{\mathcal{M}}
\newcommand{\cN}{\mathcal{N}}

\newcommand{\cP}{\mathcal{P}}

\newcommand{\cS}{\mathcal{S}}
\newcommand{\cT}{\mathcal{T}}
\newcommand{\cU}{\mathcal{U}}

\newcommand{\sL}{\mathscr{L}}

\newcommand{\sR}{\mathscr{R}}

\newcommand{\ep}{\varepsilon}
\newcommand{\diff}{\mathrm{d}}
\newcommand{\op}{\mathrm{op}}
\newcommand{\sym}{\mathrm{sym}}

\newcommand{\1}{\mbox{\rm{1}}\hspace{-0.25em}\mbox{\rm{l}}}
\newcommand{\lint}{\mathalpha{\ltimes}}
\newcommand{\rint}{\mathalpha{\rtimes}}

\providecommand{\keywords}[1]{\textbf{Keywords:} #1}
\makeatletter

\@addtoreset{equation}{section}
\makeatother
\def\widebar{\accentset{{\cc@style\underline{\mskip10mu}}}}
\numberwithin{equation}{section}
\allowdisplaybreaks

\title{Linear-quadratic stochastic Volterra controls II: Optimal strategies and Riccati--Volterra equations}

\author{
Yushi Hamaguchi\footnote{Corresponding Author. Graduate School of Engineering Science, Department of Systems Innovation, Osaka University. 1-3, Machikaneyama, Toyonaka, Osaka, Japan (Email: \href{mailto:hmgch2950@gmail.com}{hmgch2950@gmail.com}). This author was supported by JSPS KAKENHI Grant Number 22K13958.}
~~ and ~~
Tianxiao Wang\footnote{School of Mathematics, Sichuan University. Chengdu, P.\ R.\ China (Email: \href{mailto:wtxiao2014@scu.edu.cn}{wtxiao2014@scu.edu.cn}). This author was supported by National Natural Science Foundation of China (No.\ 11971332 and 11931011) and the Science Development Project of Sichuan University under grant 2020SCUNL201.}
}

\begin{document}
\maketitle


\begin{abstract}
In this paper, we study linear-quadratic control problems for stochastic Volterra integral equations with singular and non-convolution-type coefficients. The weighting matrices in the cost functional are not assumed to be non-negative definite. From a new viewpoint, we formulate a framework of causal feedback strategies. The existence and the uniqueness of a causal feedback optimal strategy are characterized by means of the corresponding Riccati--Volterra equation.
\end{abstract}


\keywords
Linear-quadratic control; stochastic Volterra integral equation; Riccati--Volterra equation.


\textbf{2020 Mathematics Subject Classification}: 60H20; 45A05; 93E20; 93B52.




\section{Introduction}

Linear-quadratic (LQ) control problems are special classes of optimal control problems described by a linear state dynamics and a quadratic cost functional. In the continuous time setting, the state dynamics is assumed to be governed by a controlled differential/integral equation. In this paper, we consider the following controlled linear stochastic Volterra integral equation (SVIE):
\begin{equation}\label{eq_state}
\begin{split}
	&X(t)=x(t)+\int^t_{t_0}\{A(t,s)X(s)+B(t,s)u(s)+b(t,s)\}\,\diff s+\int^t_{t_0}\{C(t,s)X(s)+D(t,s)u(s)+\sigma(t,s)\}\,\diff W(s),\\
	&\hspace{8cm}t\in(t_0,T),
\end{split}
\end{equation}
where $u$ is a control process, $x$ is a given deterministic function called the free term (which is also called the forcing term), $W$ is a Brownian motion, $A,B,C$ and $D$ are matrix-valued deterministic coefficients, and $b$ and $\sigma$ are vector-valued stochastic inhomogeneous terms. The cost functional is given by the following quadratic functional:
\begin{equation}\label{eq_cost}
	J(t_0,x;u):=\bE\left[\int^T_{t_0}\left\{\left\langle\begin{pmatrix}Q(t)&S(t)^\top\\S(t)&R(t)\end{pmatrix}\begin{pmatrix}X(t)\\u(t)\end{pmatrix},\begin{pmatrix}X(t)\\u(t)\end{pmatrix}\right\rangle+2\left\langle\begin{pmatrix}q(t)\\\rho(t)\end{pmatrix},\begin{pmatrix}X(t)\\u(t)\end{pmatrix}\right\rangle\right\}\,\diff t\right],
\end{equation}
where $Q,S$ and $R$ are matrix-valued deterministic functions, and $q$ and $\rho$ are vector-valued adapted processes. The LQ control problem for an SVIE, which we call an \emph{LQ stochastic Volterra control problem}, is to minimize the quadratic cost functional $J(t_0,x;u)$ over all control process $u$ subject to the state dynamics \eqref{eq_state}.

The controlled SVIE \eqref{eq_state} is a Volterra-type extension of a controlled linear stochastic differential equation (SDE):
\begin{equation}\label{eq_state_SDE}
	\begin{dcases}
	\,\diff X(s)=\{A(s)X(s)+B(s)u(s)+b(s)\}\,\diff s+\{C(s)X(s)+D(s)u(s)+\sigma(s)\}\,\diff W(s),\ s\in[t_0,T],\\
	X(t_0)=x,
	\end{dcases}
\end{equation}
with $x$ being a constant. LQ control problems for SDEs were first studied by Wonham \cite{Wo68} in 1968 and followed by many researchers; see \cite[Chapter 6]{YoZh99} and \cite{SuYo20} for systematic studies and the recent developments of LQ control theory for SDEs. In this context, there are at least two different frameworks, namely, the \emph{open-loop framework} and the \emph{closed-loop framework}. On the one hand, in the open-loop framework, the problem is to find, for each fixed input condition $(t_0,x)$, a control process $\hat{u}$ such that
\begin{equation*}
	J(t_0,x;\hat{u})\leq J(t_0,x;u)
\end{equation*}
for any other control processes $u$. Such a control process is called an open-loop optimal control. The open-loop optimal control is characterized by a coupled system of an SDE and a backward SDE (BSDE) (see \cite[Section 2.3]{SuYo20}). On the other hand, in the closed-loop framework, the problem is to find an optimal ``strategy'' which a controller uses to select a control action based on his/her state. More precisely, in the LQ control problem for SDE \eqref{eq_state_SDE} with the cost functional \eqref{eq_cost}, consider a matrix-valued deterministic function $\Xi$ and a stochastic inhomogeneous term $v$ which are independent of the choice of input conditions $(t_0,x)$. The pair $(\Xi,v)$ is called a closed-loop strategy (which is also called a state-feedback strategy). Then for each input condition $(t_0,x)$, consider the following closed-loop system of the controlled SDE \eqref{eq_state_SDE}:
\begin{equation}\label{eq_closed-loop_system_SDE}
	\begin{dcases}
	\,\diff X^{t_0,x}(s)=\{A(s)X^{t_0,x}(s)+B(s)u^{t_0,x}(s)+b(s)\}\,\diff s\\
	\hspace{3cm}+\{C(s)X^{t_0,x}(s)+D(s)u^{t_0,x}(s)+\sigma(s)\}\,\diff W(s),\ s\in[t_0,T],\\
	X^{t_0,x}(t_0)=x,\\
	u^{t_0,x}(s)=\Xi(s)X^{t_0,x}(s)+v(s),\ s\in[t_0,T],
	\end{dcases}
\end{equation}
or equivalently the SDE
\begin{equation*}
\begin{dcases}
	\,\diff X^{t_0,x}(s)=\{(A(s)+B(s)\Xi(s))X^{t_0,x}(s)+B(s)v(s)+b(s)\}\,\diff s\\
	\hspace{3cm}+\{(C(s)+D(s)\Xi(s))X^{t_0,x}(s)+D(s)v(s)+\sigma(s)\}\,\diff W(s),\ s\in[t_0,T],\\
	X^{t_0,x}(t_0)=x.
	\end{dcases}
\end{equation*}
We note that the above system is an equation for the state process $X=X^{t_0,x}$, and the control process $u=u^{t_0,x}$ is obtained as the outcome of the strategy $(\Xi,v)$ by inserting the solution $X^{t_0,x}$ into the expression $u^{t_0,x}=\Xi X^{t_0,x}+v$. In order to clarify the dependency of the outcome $u^{t_0,x}$ on the closed-loop strategy $(\Xi,v)$ and the input condition $(t_0,x)$, we write $u^{t_0,x}=(\Xi,v)[t_0,x]$. The problem in the closed-loop framework is to find a closed-loop strategy $(\hat{\Xi},\hat{v})$ such that
\begin{equation*}
	J(t_0,x;(\hat{\Xi},\hat{v})[t_0,x])\leq J(t_0,x;(\Xi,v)[t_0,x])
\end{equation*}
for any other closed-loop strategy $(\Xi,v)$ and any input condition $(t_0,x)$. In this case, the pair $(\hat{\Xi},\hat{v})$ is called a closed-loop optimal strategy. The closed-loop optimality is closely related to the solvability of a Riccati (differential) equation and a BSDE (see \cite[Section 2.4]{SuYo20}). It is worth to mention that if $(\hat{\Xi},\hat{v})$ is a closed-loop optimal strategy, then the outcome $\hat{u}^{t_0,x}=(\hat{\Xi},\hat{v})[t_0,x]$ is an open-loop optimal control for every input condition $(t_0,x)$. Therefore, each closed-loop optimal strategy can be seen as a state-feedback representation of an open-loop optimal control.

Optimal control problems of (non-linear) SVIEs were first studied by Yong \cite{Yo08}. By means of the maximum principle, he characterized the open-loop optimal control by the so-called Type-II backward stochastic Volterra integral equation (Type-II BSVIE) which is a Volterra-type extension of a BSDE. Since then, several researchers have tried to solve optimal control problems for SVIEs in the open-loop framework; see \cite{AgOk15,ChYo07,Ha21+,Ha21++,ShWaYo15,ShWeXi20,WaT18,WaT20,WaTZh17}. On the other hand, in the special case of SVIEs with completely monotone and convolution-type kernels, several kinds of feedback representations of the optimal controls were investigated by \cite{AbMiPh21,BoCoMa12,CoMa15}. Specifically, Abi Jaber, Miller and Pham \cite{AbMiPh21} studied LQ stochastic Volterra control problems with completely monotone and convolution-type kernels. Based on an infinite-dimensional approach, they obtained a kind of a linear feedback representation of the optimal control; see also \cite{AbMiPh21+} for the study on the associated integral operator Riccati equation. We emphasize that the approaches of \cite{AbMiPh21,AbMiPh21+,BoCoMa12,CoMa15} heavily rely on the special structure of the completely monotone and convolution-type kernels, and they cannot be applied to the non-convolution-type SVIE \eqref{eq_state}.

The purpose of this paper is to formulate and investigate the closed-loop framework of LQ stochastic Volterra control problems with general (singular and non-convolution-type) coefficients. In this framework, a difficulty comes from the definition of the ``strategy''. Indeed, as discussed by Pritchard and You \cite{PrYo96} in the deterministic LQ Volterra control problems, the class of state-feedback strategies of the form $u^{t_0,x}(t)=\Xi(t)X^{t_0,x}(t)+v(t)$ is not sufficient to capture the Volterra structure of the state dynamics (see also \cite{HaLiYo21} for the study of deterministic LQ Volterra control problems). In our previous paper \cite{HaWa1}, inspired by the so-called causal projection approach of \cite{HaLiYo21,PrYo96}, we introduced the notion of \emph{causal feedback strategies} for the linear controlled SVIE \eqref{eq_state}. This is a feedback strategy involving not only the state process $X(t)$, but also the \emph{forward state process} defined by
\begin{equation*}
	\Theta(s,t)=x(s)+\int^t_{t_0}\{A(s,r)X(r)+B(s,r)u(r)+b(s,r)\}\,\diff r+\int^t_{t_0}\{C(s,r)X(r)+D(s,r)u(r)+\sigma(s,r)\}\,\diff W(r)
\end{equation*}
for $(s,t)\in\triangle_2(t_0,T):=\{(s,t)\,|\,t_0<t<s<T\}$. The forward state $\Theta(s,t)$ can be seen as the causal projection of the original controlled SVIE \eqref{eq_state} which is determined by information of $X$ and $u$ up to the current time $t$. A causal feedback strategy consists of a triplet $(\Xi,\Gamma,v)$ of matrix-valued deterministic functions $\Xi$ and $\Gamma$ and a stochastic inhomogeneous term $v$ which leads to the following closed-loop system: for each input condition $(t_0,x)$,
\begin{equation}\label{eq_closed-loop_system}
	\begin{dcases}
	X^{t_0,x}(t)=x(t)+\int^t_{t_0}\{A(t,s)X^{t_0,x}(s)+B(t,s)u^{t_0,x}(s)+b(t,s)\}\,\diff s\\
	\hspace{3cm}+\int^t_{t_0}\{C(t,s)X^{t_0,x}(s)+D(t,s)u^{t_0,x}(s)+\sigma(t,s)\}\,\diff W(s),\ t\in(t_0,T),\\
	\Theta^{t_0,x}(s,t)=x(s)+\int^t_{t_0}\{A(s,r)X^{t_0,x}(r)+B(s,r)u^{t_0,x}(r)+b(s,r)\}\,\diff r\\
	\hspace{3cm}+\int^t_{t_0}\{C(s,r)X^{t_0,x}(r)+D(s,r)u^{t_0,x}(r)+\sigma(s,r)\}\,\diff W(r),\ (s,t)\in\triangle_2(t_0,T),\\
	u^{t_0,x}(t)=\Xi(t)X^{t_0,x}(t)+\int^T_t\Gamma(s,t)\Theta^{t_0,x}(s,t)\,\diff s+v(t),\ t\in(t_0,T).
	\end{dcases}
\end{equation}
We say that a pair $(X^{t_0,x},\Theta^{t_0,x})$ satisfying the above system a \emph{causal feedback solution} of the controlled SVIE \eqref{eq_state} at $(t_0,x)$ corresponding to the causal feedback strategy $(\Xi,\Gamma,v)$.
This framework is different from that of \cite{AbMiPh21,AbMiPh21+,BoCoMa12,CoMa15} and more reasonable in view of the (generalized) flow property and the time-consistency of the state dynamics. For more detailed theory on the causal feedback strategies and the associated causal feedback solutions, see our previous paper \cite{HaWa1}. In order to clarify the dependency of the outcome $u^{t_0,x}$ on the causal feedback strategy $(\Xi,\Gamma,v)$ and the input condition $(t_0,x)$, we write $u^{t_0,x}=(\Xi,\Gamma,v)[t_0,x]$. Our problem is to find a causal feedback strategy $(\hat{\Xi},\hat{\Gamma},\hat{v})$ such that
\begin{equation*}
	J(t_0,x;(\hat{\Xi},\hat{\Gamma},\hat{v})[t_0,x])\leq J(t_0,x;(\Xi,\Gamma,v)[t_0,x])
\end{equation*}
for any other causal feedback strategy $(\Xi,\Gamma,v)$ and any input condition $(t_0,x)$. In this case, we call the triplet $(\hat{\Xi},\hat{\Gamma},\hat{v})$ a \emph{causal feedback optimal strategy}.

The main contributions of this paper are the following two points:
\begin{itemize}
\item[(i)]
We show that the existence of a causal feedback optimal strategy is equivalent to the ``regular solvability'' of a \emph{Riccati--Volterra equation} \eqref{eq_Riccati--Volterra} together with an additional condition for the solution to a \emph{Type-II extended BSVIE} (Type-II EBSVIE) \eqref{eq_Type-II_EBSVIE+}. The causal feedback optimal strategy and the associated value functional are expressed by the solutions of these equations. See \cref{theo_optimality}.
\item[(ii)]
We show that the existence of a ``strongly regular solution'' of the Riccati--Volterra equation \eqref{eq_Riccati--Volterra} is equivalent to the uniform convexity of the cost functional. These two equivalent conditions imply the existence and the uniqueness of the causal feedback optimal strategy. See \cref{theo_strongly_regular_solvability}.
\end{itemize}
Furthermore, we found the following interesting fact:
\begin{itemize}
\item[\underline{Fact.}]
If the control does not enter the drift part, that is, if $B=0$, then the causal feedback optimal strategy $(\hat{\Xi},\hat{\Gamma},\hat{v})$ is of a state-feedback form in the sense that $\hat{\Gamma}=0$. Moreover, if in addition the inhomogeneous terms $b,\sigma,q$ and $\rho$ are zeros, then it is of a Markovian state-feedback form in the sense that $\hat{\Gamma}=0$ and $\hat{v}=0$. See \cref{rem_B=0}.
\end{itemize}
This is a surprising consequence since, even in the homogeneous case, the state process is highly non-Markovian and being non-semimartingale due to the Volterra structure. Very recently, a similar fact was also found in an independent work of Wang, Yong and Zhou \cite{WaYoZh22} by a different method. In \cite{WaYoZh22}, they considered an LQ stochastic Volterra control problem (involving a terminal cost) in the open-loop framework, where the coefficients $A,B,C$ and $D$ are non-convolution-type but assumed to be regular (i.e.\ bounded and differentiable), the inhomogeneous terms $b,\sigma,q$ and $\rho$ are zeros, $S=0$, and the weighting matrices $Q$ and $R$ are assumed to be non-negative and strictly positive definite, respectively. By a dynamic programming method and a decoupling technique, they derived a causal feedback represention of the open-loop optimal control by means of a path-dependent Riccati equation which is different from our Riccati--Volterra equation \eqref{eq_Riccati--Volterra}.

Besides the fact that the coefficients of the controlled SVIE \eqref{eq_state} are non-convolution-type and singular, our cost functional \eqref{eq_cost} is also quit general compared to \cite{AbMiPh21,AbMiPh21+,WaYoZh22} since we do not a priori impose any non-negativity conditions on the matrix-valued functions $Q$ and $R$. In particular, under standard non-negativity assumptions for $Q$ and $R$ which are similar to \cite{AbMiPh21,AbMiPh21+,WaYoZh22}, we see that the Riccati--Volterra equation is strongly regularly solvable, which implies that there exists a unique causal feedback optimal strategy (see \cref{cor_standard_optimal}). Our results (i) and (ii) mentioned above are extensions of the known results \cite{SuLiYo16,SuYo14} of LQ control problems for SDEs (see also \cite[Section 2.4]{SuYo20}) to our LQ stochastic Volterra control problems. Type-II EBSVIEs were introduced and investigated in our previous paper \cite{HaWa1}. They are extensions of a class of Type-II BSVIEs introduced by Yong \cite{Yo08} to the framework of causal feedback solutions of controlled SVIEs. The Riccati--Volterra equation \eqref{eq_Riccati--Volterra} is a coupled system of Riccati-type Volterra integro-differential equations which appears for the first time in the literature. This is closely related to \emph{Lyapunov--Volterra equations} which were also introduced in our previous paper \cite{HaWa1}.

The rest of this paper is organized as follows: In \cref{section_SVC}, we formulate the LQ stochastic Volterra control problems in the framework of causal feedback strategies. In \cref{section_preliminaries}, we recall the results of our previous work \cite{HaWa1}. Specifically, we introduce Type-II EBSVIEs and Lyapunov--Volterra equations which play fundamental roles in the present paper. In \cref{section_representation_cost}, we give a useful representation of the cost functional. In \cref{section_optimality}, we introduce the Riccati--Volterra equation and prove the first main result (\cref{theo_optimality}). In \cref{section_strongly_regular_solvability}, we investigate the (strongly regular) solvability of the Riccati--Volterra equation and prove the second main result (\cref{theo_strongly_regular_solvability}). Some auxiliary lemmas are proved in \hyperref[appendix]{Appendix}.


\subsection*{Notation}

$(\Omega,\cF,\bP)$ is a complete probability space, and $W$ is a one-dimensional Brownian motion. $\bF=(\cF_t)_{t\geq0}$ denotes the $\bP$-augmented filtration generated by $W$. $\bE[\cdot]$ denotes the expectation. Throughout this paper, $\bE[\cdot]^{1/2}$ denotes the square root of the expectation $\bE[\cdot]$, not the expectation of the square root. For each $0\leq t_0<T<\infty$, we define
\begin{align*}
	&\triangle_2(t_0,T):=\{(t,s)\in(t_0,T)^2\,|\,T>t>s>t_0\},\ \text{(a triangle region)}\\
	&\square_3(t_0,T):=\{(s_1,s_2,t)\in(t_0,T)^3\,|\,t<s_1\wedge s_2\}.\ \text{(a square pyramid region)}
\end{align*}
For each matrix $M\in\bR^{d_1\times d_2}$ with $d_1,d_2\in\bN$, $|M|$ denotes the Frobenius norm, $M^\top\in\bR^{d_2\times d_1}$ denotes the transpose, $M^\dagger\in\bR^{d_2\times d_1}$ denotes the Moore--Penrose pseudoinverse, and $\sR(M)$ denotes the range. For each $d\in\bN$, $\bS^d$ denotes the set of $(d\times d)$-symmetric matrices. We define $\bR^d:=\bR^{d \times 1}$, that is, each element of $\bR^d$ is understood as a column vector. We denote by $\langle\cdot,\cdot\rangle$ the usual inner product on a Euclidean space. $I_d$ denotes the $(d\times d)$-identity matrix. For each set $\Lambda$, $\1_\Lambda$ denotes the indicator function.

For each $0\leq t_0<T<\infty$ and $d_1,d_2\in\bN$, we define some spaces of stochastic (and deterministic) processes as follows:
\begin{itemize}
\item
$(L^2_\bF(t_0,T;\bR^{d_1\times d_2}),\|\cdot\|_{L^2_\bF(t_0,T)})$ is the Hilbert space of $\bR^{d_1\times d_2}$-valued, square-integrable and $\bF$-progressively measurable processes on $(t_0,T)$.
\item
$(L^{2,1}_\bF(\triangle_2(t_0,T);\bR^{d_1\times d_2}),\|\cdot\|_{L^{2,1}_\bF(\triangle_2(t_0,T))})$ is the Banach space of $\bR^{d_1\times d_2}$-valued and measurable processes $\xi$ on $\triangle_2(t_0,T)$ such that $\xi(t,\cdot)$ is $\bF$-progressively measurable on $(t_0,t)$ for each $t\in(t_0,T)$ and that $\|\xi\|_{L^{2,1}_\bF(\triangle_2(t_0,T))}<\infty$, where
\begin{equation*}
	\|\xi\|_{L^{2,1}_\bF(\triangle_2(t_0,T))}:=\bE\Bigl[\int^T_{t_0}\Bigl(\int^t_{t_0}|\xi(t,s)|\,\diff s\Bigr)^2\,\diff t\Bigr]^{1/2}.
\end{equation*}
\item
$(L^2_\bF(\triangle_2(t_0,T);\bR^{d_1\times d_2}),\|\cdot\|_{L^2_\bF(\triangle_2(t_0,T))})$ is the Hilbert space of $\xi\in L^{2,1}_\bF(\triangle_2(t_0,T);\bR^{d_1\times d_2})$ such that $\|\xi\|_{L^2_\bF(\triangle_2(t_0,T))}<\infty$, where
\begin{equation*}
	\|\xi\|_{L^2_\bF(\triangle_2(t_0,T))}:=\bE\Bigl[\int^T_{t_0}\!\!\int^t_{t_0}|\xi(t,s)|^2\,\diff s\!\,\diff t\Bigr]^{1/2}.
\end{equation*}
\item
$(L^2_{\bF,\mathrm{c}}(\triangle_2(t_0,T);\bR^{d_1\times d_2}),\|\cdot\|_{L^2_{\bF,\mathrm{c}}(\triangle_2(t_0,T))})$ is the Banach space of $\xi\in L^2_\bF(\triangle_2(t_0,T);\bR^{d_1\times d_2})$ such that $s\mapsto\xi(t,s)$ is uniformly continuous on $(t_0,t)$ with the limits $\xi(t,t):=\lim_{s\uparrow t}\xi(t,s)$ and $\xi(t,t_0):=\lim_{s\downarrow t_0}\xi(t,s)$ exist for a.e.\ $t\in(t_0,T)$ a.s.\ and that $\|\xi\|_{L^2_{\bF,\mathrm{c}}(\triangle_2(t_0,T))}<\infty$, where
\begin{equation*}
	\|\xi\|_{L^2_{\bF,\mathrm{c}}(\triangle_2(t_0,T))}:=\bE\Bigl[\int^T_{t_0}\sup_{s\in[t_0,t]}|\xi(t,s)|^2\,\diff t\Bigr]^{1/2}.
\end{equation*}
\item
$(L^\infty(t_0,T;\bR^{d_1\times d_2}),\|\cdot\|_{L^\infty(t_0,T)})$ is the Banach space of  $\bR^{d_1\times d_2}$-valued essentially bounded measurable functions on $(t_0,T)$.
\item
For $\bT=(t_0,T),(t_0,T)^2,\triangle_2(t_0,T)$ and $\square_3(t_0,T)$, $(L^2(\bT;\bR^{d_1\times d_2}),\|\cdot\|_{L^2(\bT)})$ is the Hilbert space of $\bR^{d_1\times d_2}$-valued and square-integrable deterministic functions on $\bT$.
\item
$(L^{2,1}(\triangle_2(t_0,T);\bR^{d_1\times d_2}),\|\cdot\|_{L^{2,1}(\triangle_2(t_0,T))})$ is the Banach space of $\bR^{d_1\times d_2}$-valued deterministic functions $f$ on $\triangle_2(t_0,T)$ such that $\|f\|_{L^{2,1}(\triangle_2(t_0,T))}<\infty$, where
\begin{equation*}
	\|f\|_{L^{2,1}(\triangle_2(t_0,T))}:=\Bigl(\int^T_{t_0}\Bigl(\int^t_{t_0}|f(t,s)|\,\diff s\Bigr)^2\,\diff t\Bigr)^{1/2}.
\end{equation*}
\item
$(L^{2,2,1}(\square_3(t_0,T);\bR^{d_1\times d_2}),\|\cdot\|_{L^{2,2,1}(\square_3(t_0,T))})$ is the Banach space of $\bR^{d_1\times d_2}$-valued deterministic functions $f$ on $\square_3(t_0,T)$ such that $\|f\|_{L^{2,2,1}(\square_3(t_0,T))}<\infty$, where
\begin{equation*}
	\|f\|_{L^{2,2,1}(\square_3(t_0,T))}:=\Bigl(\int^T_{t_0}\!\!\int^T_{t_0}\Bigl(\int^{s_1\wedge s_2}_{t_0}|f(s_1,s_2,t)|\,\diff t\Bigr)^2\,\diff s_1\!\,\diff s_2\Bigr)^{1/2}<\infty.
\end{equation*}
\item
$L^{2,2,1}_\sym(\square_3(t_0,T);\bR^{d\times d})$ is the set of $f\in L^{2,2,1}(\square_3(t_0,T);\bR^{d\times d})$ such that $f(s_1,s_2,t)=f(s_2,s_1,t)^\top$ for a.e.\ $(s_1,s_2,t)\in\square_3(0,T)$. It is easy to see that $L^{2,2,1}_\sym(\square_3(t_0,T);\bR^{d\times d})$ is a closed subspace of $L^{2,2,1}(\square_3(t_0,T);\bR^{d_1\times d_2})$.
\item
$\sL^2(\triangle_2(t_0,T);\bR^{d_1\times d_2})$ is the set of $f\in L^2(\triangle_2(t_0,T);\bR^{d_1\times d_2})$ satisfying the following two conditions:
\begin{itemize}
\item
it holds that
\begin{equation*}
	\|f\|_{\sL^2(\triangle_2(t_0,T))}:=\underset{t\in(t_0,T)}{\mathrm{ess\,sup}}\Bigl(\int^T_t|f(s,t)|^2\,\diff s\Bigr)^{1/2}<\infty;
\end{equation*}
\item
for any $\ep>0$, there exists a finite partition $\{U_i\}^m_{i=0}$ of $(t_0,T)$ with $t_0=U_0<U_1<\cdots<U_m=T$ such that
\begin{equation*}
	\underset{t\in(U_i,U_{i+1})}{\mathrm{ess\,sup}}\Bigl(\int^{U_{i+1}}_t|f(s,t)|^2\,\diff s\Bigr)^{1/2}<\ep
\end{equation*}
for each $i\in\{0,1,\dots,m-1\}$.
\end{itemize}
It is easy to see that $(\sL^2(\triangle_2(t_0,T);\bR^{d_1\times d_2}),\|\cdot\|_{\sL^2(\triangle_2(t_0,T))})$ is a Banach space.
\end{itemize}

Throughout this paper, $d\in\bN$ represents the dimension of state processes, and $\ell\in\bN$ represents the dimension of control processes. We fix a finite terminal time $T\in(0,\infty)$.


\section{LQ stochastic Volterra control problems}\label{section_SVC}

We define the set of input conditions by $\cI:=\{(t_0,x)\,|\,t_0\in[0,T),\,x\in L^2(t_0,T;\bR^d)\}$ and control processes by $\cU(t_0,T):=L^2_\bF(t_0,T;\bR^\ell)$. For each input condition $(t_0,x)\in \cI$ and control $u\in\cU(t_0,T)$, consider the controlled linear SVIE \eqref{eq_state} and the quadratic cost functional \eqref{eq_cost}. The following is the standing assumption of this paper.


\begin{assum}\label{assum_coefficients}
\begin{itemize}
\item
The coefficients: $A\in L^2(\triangle_2(0,T);\bR^{d\times d})$, $B\in L^2(\triangle_2(0,T);\bR^{d\times\ell})$, $C\in\sL^2(\triangle_2(0,T);\bR^{d\times d})$, $D\in\sL^2(\triangle_2(0,T);\bR^{d\times\ell})$, $Q\in L^\infty(0,T;\bS^d)$, $R\in L^\infty(0,T;\bS^\ell)$, $S\in L^\infty(0,T;\bR^{\ell\times d})$.
\item
The inhomogeneous terms: $b\in L^{2,1}_\bF(\triangle_2(0,T);\bR^d)$, $\sigma\in L^2_\bF(\triangle_2(0,T);\bR^d)$, $q\in L^2_\bF(0,T;\bR^d)$, $\rho\in L^2_\bF(0,T;\bR^\ell)$.
\end{itemize}
\end{assum}


\begin{rem}
In the standing assumption, the coefficients and the inhomogeneous terms of the controlled SVIE \eqref{eq_state} are singular and of non-convolution-types. For example, $A(t,s)$ is allowed to diverge as $s\uparrow t$, and the same is true for $B,C,D,b$ and $\sigma$. Our framework is more general than \cite{ChYo07} (where the coefficients are of non-convolution-types, but $B,C$ and $D$ are essentially regular, and the inhomogeneous terms $b$ and $\sigma$ do not appear) and \cite{AbMiPh21,AbMiPh21+} (where the coefficients are singular, but they are of convolution-types with completely monotone kernels, and the inhomogeneous terms are deterministic). It is also worth to mention that the assumptions for the coefficients $C$ and $D$ being in $\sL^2(\triangle_2(t_0,T);\bR^{d_1\times d_2})$ fit into the framework of the so-called $\star$-Volterra kernels introduced in \cite{Ha21+++}. Furthermore, we do not impose any non-negativity conditions on the weighting matrices $Q$ and $R$ at this time.
\end{rem}

The LQ stochastic Volterra control problem is stated as follows.


\begin{SVC}
For each $(t_0,x)\in\cI$, find a control process $\hat{u}\in\cU(t_0,T)$ satisfying
\begin{equation}\label{eq_SVC}
	J(t_0,x;\hat{u})=\inf_{u\in\cU(t_0,T)}J(t_0,x;u)=:V(t_0,x).
\end{equation} 
\end{SVC}


\begin{defi}
For each $(t_0,x)\in\cI$, a control process $\hat{u}\in\cU(t_0,T)$ satisfying \eqref{eq_SVC} is called an \emph{open-loop optimal control} at $(t_0,x)$. We call the map $V$ the \emph{value functional} of Problem (SVC).
\end{defi}

In this paper, we are interested in the closed-loop framework of Problem (SVC). More precisely, we consider the following causal feedback strategies which were introduced in our previous paper \cite{HaWa1}.


\begin{defi}
Each triplet $(\Xi,\Gamma,v)\in\cS(0,T):=L^\infty(0,T;\bR^{\ell\times d})\times L^2(\triangle_2(0,T);\bR^{\ell\times d})\times\cU(0,T)$ is called a \emph{causal feedback strategy}. For each $(\Xi,\Gamma,v)\in\cS(0,T)$ and $(t_0,x)\in\cI$, we say that a triplet $(X^{t_0,x},\Theta^{t_0,x},u^{t_0,x})\in L^2_\bF(t_0,T;\bR^d)\times L^2_{\bF,\mathrm{c}}(\triangle_2(t_0,T);\bR^d)\times\cU(t_0,T)$ is a \emph{causal feedback solution} of controlled SVIE \eqref{eq_state} at $(t_0,x)$ corresponding to $(\Xi,\Gamma,v)$ if it satisfies the closed-loop system \eqref{eq_closed-loop_system}. We sometimes call $(X^{t_0,x},\Theta^{t_0,x})$ the causal feedback solution for simplicity. The control process $u^{t_0,x}\in\cU(t_0,T)$ is called the \emph{outcome} of the causal feedback strategy $(\Xi,\Gamma,v)$ at $(t_0,x)$, and we write $(\Xi,\Gamma,v)[t_0,x](t):=u^{t_0,x}(t)$.
\end{defi}


\begin{theo}\label{theo_SVIE}
For each causal feedback strategy $(\Xi,\Gamma,v)\in\cS(0,T)$ and each input condition $(t_0,x)\in\cI$, the controlled SVIE \eqref{eq_state} has a unique causal feedback solution $(X^{t_0,x},\Theta^{t_0,x},u^{t_0,x})\in L^2_\bF(t_0,T;\bR^d)\times L^2_{\bF,\mathrm{c}}(\triangle_2(t_0,T);\bR^d)\times\cU(t_0,T)$. Furthermore, there exists a constant $K>0$ depending only on $A,B,C,D,\Xi,\Gamma$ such that
\begin{equation}\label{eq_closed-loop_estimate}
\begin{split}
	 &\|X^{t_0,x}\|_{L^2_\bF(t_0,T)}+\|\Theta^{t_0,x}\|_{L^2_{\bF,\mathrm{c}}(\triangle_2(t_0,T))}+\|u^{t_0,x}\|_{L^2_\bF(t_0,T)}\\
	&\leq K(\|x\|_{L^2(t_0,T)}+\|b\|_{L^{2,1}_\bF(\triangle_2(t_0,T))}+\|\sigma\|_{L^2_\bF(\triangle_2(t_0,T))}+\|v\|_{L^2_\bF(t_0,T)}).
\end{split}
\end{equation}
\end{theo}


\begin{proof}
See \cite[Theorem 2.4]{HaWa1}.
\end{proof}


\begin{rem}
We emphasize that the causal feedback strategy $(\Xi,\Gamma,v)$ is chosen to be independent of the input condition $(t_0,x)$, while the causal feedback solution $(X^{t_0,x},\Theta^{t_0,x},u^{t_0,x})$ depends on $(t_0,x)$. It is worth to mention that the causal feedback solution satisfies the (generalized) flow property with respect to the input condition $(t_0,x)$ in a suitable sense. For more detailed discussions, see our previous paper \cite{HaWa1}.
\end{rem}

The purpose of this paper is to investigate the causal feedback optimal strategy defined as follows.


\begin{defi}
A causal feedback strategy $(\hat{\Xi},\hat{\Gamma},\hat{v})\in\cS(0,T)$ is called a \emph{causal feedback optimal strategy} of Problem (SVC) if
\begin{equation*}
	J(t_0,x;(\hat{\Xi},\hat{\Gamma},\hat{v})[t_0,x])\leq J(t_0,x;(\Xi,\Gamma,v)[t_0,x])
\end{equation*}
for any input condition $(t_0,x)\in\cI$ and any causal feedback strategy $(\Xi,\Gamma,v)\in\cS(0,T)$.
\end{defi}

The following lemma provides equivalent formulations of the causal feedback optimal strategy.


\begin{lemm}\label{lemm_causal feedback_equivalence}
For each $(\hat{\Xi},\hat{\Gamma},\hat{v})\in\cS(0,T)$, the following are equivalent:
\begin{itemize}
\item[(i)]
$(\hat{\Xi},\hat{\Gamma},\hat{v})$ is a causal feedback optimal strategy of Problem (SVC);
\item[(ii)]
it holds that
\begin{equation*}
	J(t_0,x;(\hat{\Xi},\hat{\Gamma},\hat{v})[t_0,x])\leq J(t_0,x;(\Xi,0,v)[t_0,x])
\end{equation*}
for any $(\Xi,v)\in L^\infty(0,T;\bR^{\ell\times d})\times\cU(0,T)$ and any input condition $(t_0,x)\in\cI$;
\item[(iii)]
it holds that
\begin{equation*}
	J(t_0,x;(\hat{\Xi},\hat{\Gamma},\hat{v})[t_0,x])\leq J(t_0,x;(0,\Gamma,v)[t_0,x])
\end{equation*}
for any $(\Gamma,v)\in L^2(\triangle_2(0,T);\bR^{\ell\times d})\times\cU(0,T)$ and any input condition $(t_0,x)\in\cI$;
\item[(iv)]
it holds that
\begin{equation*}
	J(t_0,x;(\hat{\Xi},\hat{\Gamma},\hat{v})[t_0,x])\leq J(t_0,x;(\hat{\Xi},\hat{\Gamma},v)[t_0,x])
\end{equation*}
for any $v\in\cU(0,T)$ and any input condition $(t_0,x)\in\cI$;
\item[(v)]
it holds that
\begin{equation*}
	J(t_0,x;(\hat{\Xi},\hat{\Gamma},\hat{v})[t_0,x])\leq J(t_0,x;u)
\end{equation*}
for any $u\in\cU(t_0,T)$ and any input condition $(t_0,x)\in\cI$.
\end{itemize}
\end{lemm}


\begin{proof}
From the definition of the causal feedback optimality, (i) implies (ii),(iii),(iv) and (v). The implications (ii)\,$\Rightarrow$\,(v), (iii)\,$\Rightarrow$\,(v) and (v)\,$\Rightarrow$\,(i) are also trivial. We only need to show the implication (iv)\,$\Rightarrow$\,(v).

Assume that (iv) holds. Let $(t_0,x)\in\cI$ and $u\in\cU(t_0,T)$ be arbitrary, and denote by $(X,\Theta)$ the corresponding state pair. Define $v(t):=0$ for $t\in(0,t_0]$ and $v(t):=u(t)-\hat{\Xi}(t)X(t)-\int^T_t\hat{\Gamma}(s,t)\Theta(s,t)\,\diff s$ for $t\in(t_0,T)$. Clearly, $v$ is in $\cU(0,T)$ and satisfies $(\hat{\Xi},\hat{\Gamma},v)[t_0,x]=u$. By (iv), we have
\begin{equation*}
	J(t_0,x;(\hat{\Xi},\hat{\Gamma},\hat{v})[t_0,x])\leq J(t_0,x;(\hat{\Xi},\hat{\Gamma},v)[t_0,x])=J(t_0,x;u).
\end{equation*}
Thus, (v) holds. This completes the proof.
\end{proof}


\begin{rem}
From the above lemma, if $(\hat{\Xi},\hat{\Gamma},\hat{v})\in\cS(0,T)$ is a causal feedback optimal strategy of Problem (SVC), then for any input condition $(t_0,x)\in\cI$, the outcome $(\hat{\Xi},\hat{\Gamma},\hat{v})[t_0,x]\in\cU(t_0,T)$ is an open-loop optimal control of Problem (SVC) at $(t_0,x)$. Therefore, each causal feedback optimal strategy can be seen as a causal feedback representation of an open-loop optimal control. Note that, even if a \emph{state-feedback strategy} $(\hat{\Xi},0,\hat{v})\in\cS(0,T)$ (in which the feedback $\hat{\Gamma}$ of the forward state process is absent) is optimal among all state-feedback strategies in the sense that
\begin{equation*}
	J(t_0,x;(\hat{\Xi},0,\hat{v})[t_0,x])\leq J(t_0,x;(\Xi,0,v)[t_0,x])
\end{equation*}
for any $(\Xi,0,v)\in\cS(0,T)$ and any $(t_0,x)\in\cI$, it must be optimal among all causal feedback strategies in the sense that
\begin{equation*}
	J(t_0,x;(\hat{\Xi},0,\hat{v})[t_0,x])\leq J(t_0,x;(\Xi,\Gamma,v)[t_0,x])
\end{equation*}
for any $(\Xi,\Gamma,v)\in\cS(0,T)$ and any $(t_0,x)\in\cI$.
\end{rem}

We will also consider the homogeneous version of Problem (SVC), where the inhomogeneous terms $b,\sigma,q$ and $\rho$ are absent. In this case, the controlled SVIE \eqref{eq_state} and the cost functional \eqref{eq_cost} become
\begin{equation}\label{eq_state_0}
	X_0(t)=x(t)+\int^t_{t_0}\{A(t,s)X_0(s)+B(t,s)u(s)\}\,\diff s+\int^t_{t_0}\{C(t,s)X_0(s)+D(t,s)u(s)\}\,\diff W(s),\ t\in(t_0,T),
\end{equation}
and
\begin{equation}\label{eq_cost_0}
	J^0(t_0,x;u):=\bE\left[\int^T_{t_0}\left\langle\begin{pmatrix}Q(t)&S(t)^\top\\S(t)&R(t)\end{pmatrix}\begin{pmatrix}X_0(t)\\u(t)\end{pmatrix},\begin{pmatrix}X_0(t)\\u(t)\end{pmatrix}\right\rangle\,\diff t\right],
\end{equation}
respectively. We write the homogeneous problem by Problem (SVC)$^0$ and the corresponding value functional by $V^0(t_0,x)$. For each causal feedback strategy $(\Xi,\Gamma,v)\in\cS(0,T)$ and input condition $(t_0,x)\in\cI$, the corresponding causal feedback solution $(X^{t_0,x}_0,\Theta^{t_0,x}_0,u^{t_0,x}_0)\in L^2_\bF(t_0,T;\bR^d)\times L^2_{\bF,\mathrm{c}}(\triangle_2(t_0,T);\bR^d)\times\cU(t_0,T)$ of the homogeneous controlled SVIE \eqref{eq_state_0} satisfies
\begin{equation*}
	\begin{dcases}
	X^{t_0,x}_0(t)=x(t)+\int^t_{t_0}\{A(t,s)X^{t_0,x}_0(s)+B(t,s)u^{t_0,x}_0(s)\}\,\diff s\\
	\hspace{3cm}+\int^t_{t_0}\{C(t,s)X^{t_0,x}_0(s)+D(t,s)u^{t_0,x}_0(s)\}\,\diff W(s),\ t\in(t_0,T),\\
	\Theta^{t_0,x}_0(s,t)=x(s)+\int^t_{t_0}\{A(s,r)X^{t_0,x}_0(r)+B(s,r)u^{t_0,x}_0(r)\}\,\diff r\\
	\hspace{3cm}+\int^t_{t_0}\{C(s,r)X^{t_0,x}_0(r)+D(s,r)u^{t_0,x}_0(r)\}\,\diff W(r),\ (s,t)\in\triangle_2(t_0,T),\\\displaystyle
	u^{t_0,x}_0(t)=\Xi(t)X^{t_0,x}_0(t)+\int^T_t\Gamma(s,t)\Theta^{t_0,x}_0(s,t)\,\diff s+v(t),\ t\in(t_0,T).
	\end{dcases}
\end{equation*}
We denote the outcome by $(\Xi,\Gamma,v)^0[t_0,x]:=u^{t_0,x}_0\in\cU(t_0,T)$.


\section{Preliminaries}\label{section_preliminaries}

In this section, we summarize the results of our previous work \cite{HaWa1}. Specifically, we introduce \emph{Type-II extended backward stochastic Volterra integral equations} (Type-II EBSVIEs) and \emph{Lyapunov--Volterra equations} which play fundamental roles in the study of Problem (SVC). For more detailed discussions and proofs, see \cite{HaWa1}.


\subsection{Type-II EBSVIEs and duality principle}

Let $(\Xi,\Gamma)\in L^\infty(0,T;\bR^{\ell\times d})\times L^2(\triangle_2(0,T);\bR^{\ell\times d})$. For each $\chi\in L^{2,1}_\bF(\triangle_2(0,T);\bR^d)$ and $\psi\in L^2_\bF(0,T;\bR^d)$, we consider the following \emph{Type-II EBSVIE}:
\begin{equation}\label{eq_non-local_BSDE}
	\begin{dcases}
	\,\diff\eta(t,s)=-\Bigl\{\chi(t,s)+\Gamma(t,s)^\top\int^T_sB(r,s)^\top\eta(r,s)\,\diff r+\Gamma(t,s)^\top\int^T_sD(r,s)^\top\zeta(r,s)\,\diff r\Bigr\}\,\diff s\\
	\hspace{3cm}+\zeta(t,s)\,\diff W(s),\ (t,s)\in\triangle_2(0,T),\\
	\eta(t,t)=\psi(t)+\int^T_t(A+B\triangleright\Xi)(r,t)^\top\eta(r,t)\,\diff r+\int^T_t(C+D\triangleright\Xi)(r,t)^\top\zeta(r,t)\,\diff r,\ t\in(0,T),
	\end{dcases}
\end{equation}
with $A+B\triangleright\Xi\in L^2(\triangle_2(0,T);\bR^{d\times d})$ and $C+D\triangleright\Xi\in\sL^2(\triangle_2(0,T);\bR^{d\times d})$ defined by
\begin{equation*}
	(A+B\triangleright\Xi)(t,s):=A(t,s)+B(t,s)\Xi(s),\ (C+D\triangleright\Xi)(t,s):=C(t,s)+D(t,s)\Xi(s)
\end{equation*}
for $(t,s)\in\triangle_2(0,T)$.


\begin{defi}
We say that a pair $(\eta,\zeta)\in L^2_{\bF,\mathrm{c}}(\triangle_2(0,T);\bR^d)\times L^2_\bF(\triangle_2(0,T);\bR^d)$ is an \emph{adapted solution} to the Type-II EBSVIE \eqref{eq_non-local_BSDE} if it satisfies
\begin{equation*}
	\begin{dcases}
	\eta(t,\theta)=\eta(t,t)+\int^t_\theta\Bigl\{\chi(t,s)+\Gamma(t,s)^\top\int^T_s B(r,s)^\top\eta(r,s)\,\diff r+\Gamma(t,s)^\top\int^T_s D(r,s)^\top\zeta(r,s)\,\diff r\Bigr\}\,\diff s\\
	\hspace{4cm}-\int^t_\theta\zeta(t,s)\,\diff W(s),\\
	\eta(t,t)=\psi(t)+\int^T_t(A+B\triangleright\Xi)(r,t)^\top\eta(r,t)\,\diff r+\int^T_t(C+D\triangleright\Xi)(r,t)^\top\zeta(r,t)\,\diff r,
	\end{dcases}
\end{equation*}
for a.e.\ $t\in(0,T)$ and any $\theta\in[0,t]$, a.s.
\end{defi}


\begin{theo}\label{theo_Type-II_EBSVIE}
Let $(\Xi,\Gamma)\in L^\infty(0,T;\bR^{\ell\times d})\times L^2(\triangle_2(0,T);\bR^{\ell\times d})$ be fixed. For any $\chi\in L^{2,1}_\bF(\triangle_2(0,T);\bR^d)$ and $\psi\in L^2_\bF(0,T;\bR^d)$, there exists a unique adapted solution $(\eta,\zeta)\in L^2_{\bF,\mathrm{c}}(\triangle_2(0,T);\bR^d)\times L^2_\bF(\triangle_2(0,T);\bR^d)$ to the Type-II EBSVIE \eqref{eq_non-local_BSDE}. Furthermore, for any $v\in\cU(0,T)$ and $(t_0,x)\in\cI$, the following \emph{duality principle} holds:
\begin{equation}\label{eq_duality}
\begin{split}
	&\bE\Bigl[\int^T_{t_0}\Bigl\{\langle\psi(t),X^{t_0,x}(t)\rangle+\int^T_t\langle\chi(s,t),\Theta^{t_0,x}(s,t)\rangle\,\diff s\Bigr\}\,\diff t\Bigr]\\
	&=\int^T_{t_0}\langle\bE[\eta(t,t_0)],x(t)\rangle\,\diff t+\bE\Bigl[\int^T_{t_0}\Bigl\{\int^T_t\langle\eta(s,t),b(s,t)\rangle\,\diff s+\int^T_t\langle\zeta(s,t),\sigma(s,t)\rangle\,\diff s\Bigr\}\,\diff t\Bigr]\\
	&\hspace{0.5cm}+\bE\Bigl[\int^T_{t_0}\Bigl\langle\int^T_t\{B(s,t)^\top\eta(s,t)+D(s,t)^\top\zeta(s,t)\}\,\diff s,v(t)\Bigr\rangle\,\diff t\Bigr],
\end{split}
\end{equation}
where $(X^{t_0,x},\Theta^{t_0,x})\in L^2_\bF(t_0,T;\bR^d)\times L^2_{\bF,\mathrm{c}}(\triangle_2(t_0,T);\bR^d)$ is the causal feedback solution to the SVIE \eqref{eq_state} at $(t_0,x)$ corresponding to the causal feedback strategy $(\Xi,\Gamma,v)\in\cS(0,T)$.
\end{theo}


\subsection{Lyapunov--Volterra equations and quadratic functionals}


\begin{defi}
We denote by $\Pi(0,T)$ the set of pairs $P=(P^{(1)},P^{(2)})$ with $P^{(1)}:(0,T)\to\bR^{d\times d}$ and $P^{(2)}:\square_3(t_0,T)\to\bR^{d\times d}$ such that
\begin{itemize}
\item
$P^{(1)}\in L^\infty(0,T;\bS^d)$;
\item
for a.e.\ $(s_1,s_2)\in(0,T)^2$, $t\mapsto P^{(2)}(s_1,s_2,t)$ is absolutely continuous on $(0,s_1\wedge s_2)$;
\item
the function $P^{(2)}(s_1,s_2,s_1\wedge s_2):=\lim_{t\uparrow s_1\wedge s_2}P(s_1,s_2,t)$, $(s_1,s_2)\in(0,T)^2$, is in $L^2((0,T)^2;\bR^{d\times d})$;
\item
the function $\dot{P}^{(2)}(s_1,s_2,t):=\frac{\partial P^{(2)}}{\partial t}(s_1,s_2,t)$, $(s_1,s_2,t)\in\square_3(0,T)$, is in $L^{2,2,1}(\square_3(0,T);\bR^{d\times d})$;
\item
for a.e.\ $(s_1,s_2,t)\in\square_3(0,T)$, it holds that $P^{(2)}(s_1,s_2,t)=P^{(2)}(s_2,s_1,t)^\top$.
\end{itemize}
\end{defi}


\begin{rem}\label{rem_Pi_Banach}
Observe that $\Pi(0,T)$ is a Banach space with the norm
\begin{align*}
	&\|P\|_{\Pi(0,T)}:=\underset{t\in(0,T)}{\mathrm{ess\,sup}}|P^{(1)}(t)|+\Bigl(\int^T_0\!\!\int^T_0|P^{(2)}(s_1,s_2,s_1\wedge s_2)|^2\,\diff s_1\!\,\diff s_2\Bigr)^{1/2}\\
	&\hspace{4cm}+\Bigl(\int^T_0\!\!\int^T_0\Bigl(\int^{s_1\wedge s_2}_0|\dot{P}^{(2)}(s_1,s_2,t)|\,\diff t\Bigr)^2\,\diff s_1\!\,\diff s_2\Bigr)^{1/2}
\end{align*}
for $P=(P^{(1)},P^{(2)})\in\Pi(0,T)$. Furthermore, the following holds:
\begin{equation*}
	\underset{t\in(0,T)}{\mathrm{ess\,sup}}|P^{(1)}(t)|+\Bigl(\int^T_0\!\!\int^T_0\sup_{t\in[0,s_1\wedge s_2]}|P^{(2)}(s_1,s_2,t)|^2\,\diff s_1\!\,\diff s_2\Bigr)^{1/2}\leq\|P\|_{\Pi(0,T)}<\infty.
\end{equation*}
\end{rem}


\begin{lemm}\label{lemm_self-adjoint}
Let $P=(P^{(1)},P^{(2)})\in\Pi(0,T)$. Then for each $t_0\in[0,T)$, the map $\cP^{t_0}:L^2(t_0,T;\bR^d)\to L^2(t_0,T;\bR^d)$ defined by
\begin{equation*}
	(\cP^{t_0}x)(t):=P^{(1)}(t)x(t)+\int^T_{t_0}P^{(2)}(t,r,t_0)x(r)\,\diff r,\ t\in(t_0,T),
\end{equation*}
for $x\in L^2(t_0,T;\bR^d)$, is a self-adjoint bounded linear operator on the Hilbert space $L^2(t_0,T;\bR^d)$.
\end{lemm}


\begin{lemm}\label{lemm_Pi_0}
Let $P=(P^{(1)},P^{(2)})\in\Pi(0,T)$. Assume that
\begin{equation*}
	\int^T_{t_0}\langle P^{(1)}(t)x(t),x(t)\rangle\,\diff t+\int^T_{t_0}\!\!\int^T_{t_0}\langle P^{(2)}(s_1,s_2,t_0)x(s_2),x(s_1)\rangle\,\diff s_1\!\,\diff s_2=0
\end{equation*}
for any $(t_0,x)\in\cI$. Then $P=0$ in the sense that $P^{(1)}(t)=0$ for a.e.\ $t\in(0,T)$ and $P^{(2)}(s_1,s_2,t)=0$ for any $t\in[0,s_1\wedge s_2]$ for a.e.\ $(s_1,s_2)\in(0,T)^2$.
\end{lemm}


\begin{defi}\label{defi_lint_rint}
Let $P=(P^{(1)},P^{(2)})\in\Pi(0,T)$. For each $M_1:\triangle_2(0,T)\to\bR^{d_1\times d}$ and $M_2:\triangle_2(0,T)\to\bR^{d\times d_2}$ with $d_1,d_2\in\bN$, we define
\begin{align*}
	&(M_1\lint P)(s,t):=M_1(s,t)P^{(1)}(s)+\int^T_tM_1(r,t)P^{(2)}(r,s,t)\,\diff r,\ (s,t)\in\triangle_2(0,T),\\
	&(P\rint M_2)(s,t):=P^{(1)}(s)M_2(s,t)+\int^T_tP^{(2)}(s,r,t)M_2(r,t)\,\diff r,\ (s,t)\in\triangle_2(0,T),
\end{align*}
and
\begin{equation*}
	(M_1\lint P\rint M_2)(t):=\!\int^T_tM_1(s,t)P^{(1)}(s)M_2(s,t)\,\diff s+\!\int^T_t\!\!\int^T_t\!M_1(s_1,t)P^{(2)}(s_1,s_2,t)M_2(s_2,t)\,\diff s_1\!\,\diff s_2,\ t\in(0,T).
\end{equation*}
\end{defi}


\begin{lemm}\label{lemm_lint-rint_integrability}
Let $P=(P^{(1)},P^{(2)})\in\Pi(0,T)$. Fix $d_1,d_2\in\bN$.
\begin{itemize}
\item[(i)]
For each $M\in \cL(\triangle_2(0,T);\bR^{d\times d_2})$ with $\cL$ being one of $L^{2,1}_\bF$, $L^2_\bF$, $L^{2,1}$, $L^2$ and $\sL^2$, it holds that $P\rint M\in \cL(\triangle_2(0,T);\bR^{d\times d_2})$ and
\begin{equation*}
	\|P\rint M\|_{\cL(\triangle_2(0,T))}\leq\|P\|_{\Pi(0,T)}\|M\|_{\cL(\triangle_2(0,T))}.
\end{equation*}
Furthermore, $(P\rint M)^\top=M^\top\lint P$.
\item[(ii)]
For each $M_1\in \sL^2(\triangle_2(0,T);\bR^{d_1\times d})$ and $M_2\in \sL^2(\triangle_2(0,T);\bR^{d\times d_2})$, it holds that $M_1\lint P\rint M_2\in L^\infty(0,T;\bR^{d_1\times d_2})$ and
\begin{equation*}
	\|M_1\lint P\rint M_2\|_{L^\infty(0,T)}\leq\|M_1\|_{\sL^2(\triangle_2(0,T))}\|P\|_{\Pi(0,T)}\|M_2\|_{\sL^2(\triangle_2(0,T))}.
\end{equation*}
Furthermore, $(M_1\lint P\rint M_2)^\top=M^\top_2\lint P\rint M^\top_1$. In particular, $M^\top_2\lint P\rint M_2\in L^\infty(0,T;\bS^{d_2})$.
\item[(iii)]
For each $M\in \sL^2(\triangle_2(0,T);\bR^{d_1\times d})$ and $\xi\in L^2_\bF(\triangle_2(0,T);\bR^{d\times d_2})$, it holds that
\begin{equation*}
	M\lint P\rint\xi\in L^2_\bF(0,T;\bR^{d_1\times d_2}),\ \xi^\top\lint P\rint\xi\in L^1_\bF(0,T;\bS^{d_2}).
\end{equation*}
\end{itemize}
\end{lemm}


\begin{lemm}\label{lemm_Ito}
Let $P=(P^{(1)},P^{(2)})\in\Pi(0,T)$, $(t_0,x)\in\cI$ and $u\in\cU(t_0,T)$. Then
\begin{align*}
	&\bE\Bigl[\int^T_{t_0}\Bigl\{\langle P^{(1)}(t)X(t),X(t)\rangle+2\int^T_t\langle P^{(2)}(s,t,t)X(t),\Theta(s,t)\rangle\,\diff s\Bigr\}\,\diff t\Bigr]\\
	&=\int^T_{t_0}\langle P^{(1)}(t)x(t),x(t)\rangle\,\diff t+\int^T_{t_0}\!\!\int^T_{t_0}\langle P^{(2)}(s_1,s_2,t_0)x(s_2),x(s_1)\rangle\,\diff s_1\!\,\diff s_2+\bE\Bigl[\int^T_{t_0}(\sigma^\top\lint P\rint\sigma)(t)\,\diff t\Bigr]\\
	&\hspace{0.5cm}+\bE\Bigl[\int^T_{t_0}\Bigl\{\langle(D^\top\lint P\rint D)(t)u(t),u(t)\rangle\\
	&\hspace{2cm}+2\Bigl\langle(D^\top\lint P\rint C)(t)X(t)+\int^T_t(B^\top\lint P)(s,t)\Theta(s,t)\,\diff s+(D^\top\lint P\rint\sigma)(t),u(t)\Bigr\rangle\\
	&\hspace{2cm}+\langle(C^\top\lint P\rint C)(t)X(t),X(t)\rangle+2\int^T_t\langle(P\rint A)(s,t)X(t),\Theta(s,t)\rangle\,\diff s\\
	&\hspace{2cm}+\int^T_t\!\!\int^T_t\langle\dot{P}^{(2)}(s_1,s_2,t)\Theta(s_2,t),\Theta(s_1,t)\rangle\,\diff s_1\!\,\diff s_2\\
	&\hspace{2cm}+2\langle(C^\top\lint P\rint\sigma)(t),X(t)\rangle+2\int^T_t\langle(P\rint b)(s,t),\Theta(s,t)\rangle\,\diff s\Bigr\}\,\diff t\Bigr],
\end{align*}
where $X$ and $\Theta$ are the state process and the forward state process, respectively, corresponding to the input condition $(t_0,x)$ and the control $u$.
\end{lemm}

For each $(\Xi,\Gamma)\in L^\infty(0,T;\bR^{\ell\times d})\times L^2(\triangle_2(0,T);\bR^{\ell\times d})$ and $P=(P^{(1)},P^{(2)})\in\Pi(0,T)$, define
\begin{equation}\label{eq_Lyapunov--Volterra_coefficients}
\begin{split}
	&F^{(1)}[\Xi;P](t):=(C^\top\lint P\rint C)(t)+\Xi(t)^\top(D^\top\lint P\rint C)(t)+(C^\top\lint P\rint D)(t)\Xi(t)+\Xi(t)^\top(D^\top\lint P\rint D)(t)\Xi(t),\\
	&\hspace{6cm}t\in(0,T),\\
	&F^{(2)}[\Xi,\Gamma;P](s,t):=(P\rint A)(s,t)+(P\rint B)(s,t)\Xi(t)+\Gamma(s,t)^\top(D^\top\lint P\rint C)(t)+\Gamma(s,t)^\top(D^\top\lint P\rint D)(t)\Xi(t),\\
	&\hspace{6cm}(s,t)\in\triangle_2(0,T),\\
	&F^{(3)}[\Gamma;P](s_1,s_2,t):=\Gamma(s_1,t)^\top(B^\top\lint P)(s_2,t)+(P\rint B)(s_1,t)\Gamma(s_2,t)+\Gamma(s_1,t)^\top(D^\top\lint P\rint D)(t)\Gamma(s_2,t),\\
	&\hspace{6cm}(s_1,s_2,t)\in\square_3(0,T).
\end{split}
\end{equation}
Observe that the map
\begin{equation*}
	P\mapsto(F^{(1)}[\Xi;P],F^{(2)}[\Xi,\Gamma;P],F^{(3)}[\Gamma;P]).
\end{equation*}
is a bounded linear operator from $\Pi(0,T)$ to $L^\infty(0,T;\bS^d)\times L^2(\triangle_2(0,T);\bR^{d\times d})\times L^{2,2,1}_\sym(\square_3(0,T);\bR^{d\times d})$.

For each $(Q^{(1)},Q^{(2)},Q^{(3)})\in L^\infty(0,T;\bS^d)\times L^2(\triangle_2(0,T);\bR^{d\times d})\times L^{2,2,1}_\sym(\square_3(0,T);\bR^{d\times d})$, we introduce the following \emph{Lyapunov--Volterra equation}:
\begin{equation}\label{eq_Lyapunov--Volterra}
	\begin{dcases}
	P^{(1)}(t)=F^{(1)}[\Xi;P](t)+Q^{(1)}(t),\ t\in(0,T),\\
	P^{(2)}(s,t,t)=P^{(2)}(t,s,t)^\top=F^{(2)}[\Xi,\Gamma;P](s,t)+Q^{(2)}(s,t),\ (s,t)\in\triangle_2(0,T),\\
	\dot{P}^{(2)}(s_1,s_2,t)+F^{(3)}[\Gamma;P](s_1,s_2,t)+Q^{(3)}(s_1,s_2,t)=0,\ (s_1,s_2,t)\in\square_3(0,T).
	\end{dcases}
\end{equation}
We say that $P=(P^{(1)},P^{(2)})\in\Pi(0,T)$ satisfying the above equalities a solution to the Lyapunov--Volterra equation \eqref{eq_Lyapunov--Volterra}.


\begin{theo}\label{theo_Lyapunov--Volterra}
Let $(\Xi,\Gamma)\in L^\infty(0,T;\bR^{\ell\times d})\times L^2(\triangle_2(0,T);\bR^{\ell\times d})$. For each $(Q^{(1)},Q^{(2)},Q^{(3)})\in L^\infty(0,T;\bS^d)\times L^2(\triangle_2(0,T);\bR^{d\times d})\times L^{2,2,1}_\sym(\square_3(0,T);\bR^{d\times d})$, There exists a unique solution $P=(P^{(1)},P^{(2)})\in\Pi(0,T)$ to the Lyapunov--Volterra equation \eqref{eq_Lyapunov--Volterra}. Furthermore, for any $(t_0,x)\in\cI$, the following holds:
\begin{equation}\label{eq_quadratic_representation}
\begin{split}
	&\int^T_{t_0}\langle P^{(1)}(t)x(t),x(t)\rangle\,\diff t+\int^T_{t_0}\!\!\int^T_{t_0}\langle P^{(2)}(s_1,s_2,t_0)x(s_2),x(s_1)\rangle\,\diff s_1\!\,\diff s_2\\
	&=\bE\Bigl[\int^T_{t_0}\Bigl\{\langle Q^{(1)}(t)X^{t_0,x}_0(t),X^{t_0,x}_0(t)\rangle+2\int^T_t\langle Q^{(2)}(s,t)X^{t_0,x}_0(t),\Theta^{t_0,x}_0(s,t)\rangle\,\diff s\\
	&\hspace{4cm}+\int^T_t\!\!\int^T_t\langle Q^{(3)}(s_1,s_2,t)\Theta^{t_0,x}_0(s_2,t),\Theta^{t_0,x}_0(s_1,t)\rangle\,\diff s_1\!\,\diff s_2\Bigr\}\,\diff t\Bigr],
\end{split}
\end{equation}
where $(X^{t_0,x}_0,\Theta^{t_0,x}_0)\in L^2_\bF(t_0,T;\bR^d)\times L^2_{\bF,\mathrm{c}}(\triangle_2(t_0,T);\bR^d)$ is the causal feedback solution to the homogeneous controlled SVIE \eqref{eq_state_0} at $(t_0,x)$ corresponding to $(\Xi,\Gamma,0)\in\cS(0,T)$.
\end{theo}


\section{Representation of the cost functional}\label{section_representation_cost}

Based on our previous results \cite{HaWa1} summarized in \cref{section_preliminaries}, we provide a useful representation formula for the cost functional.

For each $(\Xi,\Gamma)\in L^\infty(0,T;\bR^{\ell\times d})\times L^2(\triangle_2(0,T);\bR^{\ell\times d})$, define
\begin{equation}\label{eq_Lyapunov--Volterra_inhomogeneous}
\begin{split}
	&Q^{(1)}[\Xi](t):=Q(t)+\Xi(t)^\top S(t)+S(t)^\top\Xi(t)+\Xi(t)^\top R(t)\Xi(t),\ t\in(0,T),\\
	&Q^{(2)}[\Xi,\Gamma](s,t):=\Gamma(s,t)^\top S(t)+\Gamma(s,t)^\top R(t)\Xi(t),\ (s,t)\in\triangle_2(0,T),\\
	&Q^{(3)}[\Gamma](s_1,s_2,t):=\Gamma(s_1,t)^\top R(t)\Gamma(s_2,t),\ (s_1,s_2,t)\in\square_3(0,T).
\end{split}
\end{equation}
It is easy to see that
\begin{equation*}
	(Q^{(1)}[\Xi],Q^{(2)}[\Xi,\Gamma],Q^{(3)}[\Gamma])\in L^\infty(0,T;\bS^d)\times L^2(\triangle_2(0,T);\bR^{d\times d})\times L^{2,2,1}_\sym(\square_3(0,T);\bR^{d\times d}).
\end{equation*}


\begin{theo}\label{theo_representation_cost}
Let $(\Xi,\Gamma,v)\in\cS(0,T)$ be arbitrary. Let $P=(P^{(1)},P^{(2)})\in\Pi(0,T)$ be the solution to the Lyapunov--Volterra equation \begin{equation}\label{eq_Lyapunov--Volterra_(Xi,Gamma)}
	\begin{dcases}
	P^{(1)}(t)=F^{(1)}[\Xi;P](t)+Q^{(1)}[\Xi](t),\ t\in(0,T),\\
	P^{(2)}(s,t,t)=P^{(2)}(t,s,t)^\top=F^{(2)}[\Xi,\Gamma;P](s,t)+Q^{(2)}[\Xi,\Gamma](s,t),\ (s,t)\in\triangle_2(0,T),\\
	\dot{P}^{(2)}(s_1,s_2,t)+F^{(3)}[\Gamma;P](s_1,s_2,t)+Q^{(3)}[\Gamma](s_1,s_2,t)=0,\ (s_1,s_2,t)\in\square_3(0,T),
	\end{dcases}
\end{equation}
and let $(\eta,\zeta)\in L^2_{\bF,\mathrm{c}}(\triangle_2(0,T);\bR^d)\times L^2_\bF(\triangle_2(0,T);\bR^d)$ be the adapted solution to the Type-II EBSVIE
\begin{equation}\label{eq_Type-II_EBSVIE_(Xi,Gamma,v)}
	\begin{dcases}
	\mathrm{d}\eta(t,s)=-\Bigl\{(P\rint b)(t,s)+\Gamma(t,s)^\top(\rho(s)+(D^\top\lint P\rint\sigma)(s))\\
	\hspace{2.5cm}+((P\rint B)(t,s)+\Gamma(t,s)^\top(R(s)+(D^\top\lint P\rint D)(s)))v(s)\\
	\hspace{2.5cm}+\Gamma(t,s)^\top\int^T_sB(r,s)^\top\eta(r,s)\,\diff r+\Gamma(t,s)^\top\int^T_sD(r,s)^\top\zeta(r,s)\,\diff r\Bigr\}\,\diff s\\
	\hspace{2cm}+\zeta(t,s)\,\diff W(s),\ (t,s)\in\triangle_2(0,T),\\
	\eta(t,t)=q(t)+(C^\top\lint P\rint\sigma)(t)+\Xi(t)^\top(\rho(t)+(D^\top\lint P\rint\sigma)(t))\\
	\hspace{2cm}+(S(t)^\top+(C^\top\lint P\rint D)(t)+\Xi(t)^\top(R(t)+(D^\top\lint P\rint D)(t)))v(t)\\
	\hspace{2cm}+\int^T_t(A+B\triangleright\Xi)(r,t)^\top\eta(r,t)\,\diff r+\int^T_t(C+D\triangleright\Xi)(r,t)^\top\zeta(r,t)\,\diff r,\ t\in(0,T).
	\end{dcases}
\end{equation}
Define $\kappa\in L^2_\bF(0,T;\bR^\ell)$ by
\begin{equation*}
	\kappa(t):=\rho(t)\mathalpha{+}(D^\top\lint P\rint\sigma)(t)\mathalpha{+}\int^T_tB(s,t)^\top\eta(s,t)\,\diff s\mathalpha{+}\int^T_tD(s,t)^\top\zeta(s,t)\,\diff s,\ t\in(0,T).
\end{equation*}
Then for any $(t_0,x)\in\cI$, $\tilde{v}\in\cU(0,T)$ and $\mu\in\bR$, it holds that
\begin{equation}\label{eq_representation_cost}
\begin{split}
	&J(t_0,x;(\Xi,\Gamma,v+\mu \tilde{v})[t_0,x])\\
	&=\int^T_{t_0}\langle P^{(1)}(t)x(t),x(t)\rangle\,\diff t+\int^T_{t_0}\!\!\int^T_{t_0}\langle P^{(2)}(s_1,s_2,t_0)x(s_2),x(s_1)\rangle\,\diff s_1\!\,\diff s_2+2\int^T_{t_0}\langle\bE[\eta(t,t_0)],x(t)\rangle\,\diff t\\
	&\hspace{0.5cm}+\bE\Bigl[\int^T_{t_0}\Bigl\{(\sigma^\top\lint P\rint\sigma)(t)+2\int^T_t\langle\eta(s,t),b(s,t)\rangle\,\diff s+2\int^T_t\langle\zeta(s,t),\sigma(s,t)\rangle\,\diff s\\
	&\hspace{2cm}+\langle(R(t)+(D^\top\lint P\rint D)(t))v(t),v(t)\rangle+2\langle\kappa(t),v(t)\rangle\Bigr\}\,\diff t\Bigr]\\
	&\hspace{0.5cm}+\mu^2\bE\Bigl[\int^T_{t_0}\Bigl\{\langle(R(t)+(D^\top\lint P\rint D)(t))\tilde{v}(t),\tilde{v}(t)\rangle\\
	&\hspace{2cm}+2\Bigl\langle(S(t)+(D^\top\lint P\rint C)(t)+(R(t)+(D^\top\lint P\rint D)(t))\Xi(t))\tilde{X}^{t_0,0}_0(t)\\
	&\hspace{3cm}+\int^T_t((B^\top\rint P)(s,t)+(R(t)+(D^\top\lint P\rint D)(t))\Gamma(s,t))\tilde{\Theta}^{t_0,0}_0(s,t)\,\diff s,\tilde{v}(t)\Bigr\rangle\Bigr\}\,\diff t\Bigr]\\
	&\hspace{0.5cm}+2\mu\bE\Bigl[\int^T_{t_0}\Bigl\langle(S(t)+(D^\top\lint P\rint C)(t)+(R(t)+(D^\top\lint P\rint D)(t))\Xi(t))X^{t_0,x}(t)\\
	&\hspace{3cm}+\int^T_t((B^\top\rint P)(s,t)+(R(t)+(D^\top\lint P\rint D)(t))\Gamma(s,t))\Theta^{t_0,x}(s,t)\,\diff s\\
	&\hspace{3cm}+\kappa(t)+(R(t)+(D^\top\lint P\rint D)(t))v(t),\tilde{v}(t)\Bigr\rangle\,\diff t\Bigr],
\end{split}
\end{equation}
where $(X^{t_0,x},\Theta^{t_0,x})\in L^2_\bF(t_0,T;\bR^d)\times L^2_{\bF,\mathrm{c}}(\triangle_2(t_0,T);\bR^d)$ is the causal feedback solution to the controlled SVIE \eqref{eq_state} at $(t_0,x)$ corresponding to $(\Xi,\Gamma,v)$, and $(\tilde{X}^{t_0,0}_0,\tilde{\Theta}^{t_0,0}_0)\in L^2_\bF(t_0,T;\bR^d)\times L^2_{\bF,\mathrm{c}}(\triangle_2(t_0,T);\bR^d)$ is the causal feedback solution to the homogeneous controlled SVIE \eqref{eq_state_0} at $(t_0,0)$ corresponding to $(\Xi,\Gamma,\tilde{v})$.
\end{theo}


\begin{rem}
Noting \cref{lemm_lint-rint_integrability}, the well-posedness of the Lyapunov--Volterra equation \eqref{eq_Lyapunov--Volterra_(Xi,Gamma)} and the Type-II EBSVIE \eqref{eq_Type-II_EBSVIE_(Xi,Gamma,v)} follow from \cref{theo_Lyapunov--Volterra} and \cref{theo_Type-II_EBSVIE}, respectively.
\end{rem}


\begin{proof}[Proof of \cref{theo_representation_cost}]
First, we let $P=(P^{(1)},P^{(2)})\in\Pi(0,T)$ be arbitrary. Observe that, for any $(t_0,x)\in\cI$ and $u\in\cU(t_0,T)$,
\begin{align*}
	&J(t_0,x;u)\\
	&=\bE\Bigl[\int^T_{t_0}\{\langle Q(t)X(t),X(t)\rangle+\langle R(t)u(t),u(t)\rangle+2\langle S(t)X(t),u(t)\rangle+2\langle q(t),X(t)\rangle+2\langle\rho(t),u(t)\rangle\}\,\diff t\Bigr]\\
	&=\bE\Bigl[\int^T_{t_0}\Bigl\{\langle P^{(1)}(t)X(t),X(t)\rangle+2\int^T_t\langle P^{(2)}(s,t,t)X(t),\Theta(s,t)\rangle\,\diff s\Bigr\}\,\diff t\Bigr]\\
	&\hspace{0.5cm}+\bE\Bigl[\int^T_{t_0}\Bigl\{\langle R(t)u(t),u(t)\rangle+2\langle S(t)X(t)+\rho(t),u(t)\rangle\\
	&\hspace{2cm}+\langle(-P^{(1)}(t)+Q(t))X(t),X(t)\rangle-2\int^T_t\langle P^{(2)}(s,t,t)X(t),\Theta(s,t)\rangle\,\diff s+2\langle q(t),X(t)\rangle\Bigr\}\,\diff t\Bigr],
\end{align*}
where $(X,\Theta)$ is the corresponding state pair. By \cref{lemm_Ito}, we have
\begin{equation}\label{eq_J}
\begin{split}
	&J(t_0,x;u)\\
	&=\int^T_{t_0}\langle P^{(1)}(t)x(t),x(t)\rangle\,\diff t+\int^T_{t_0}\!\!\int^T_{t_0}\langle P^{(2)}(s_1,s_2,t_0)x(s_2),x(s_1)\rangle\,\diff s_1\!\,\diff s_2+\bE\Bigl[\int^T_{t_0}(\sigma^\top\lint P\rint\sigma)(t)\,\diff t\Bigr]\\
	&\hspace{0.5cm}+\bE\Bigl[\int^T_{t_0}\Bigl\{\langle(R(t)+(D^\top\lint P\rint D)(t))u(t),u(t)\rangle\\
	&\hspace{1.5cm}+2\Bigl\langle(S(t)+(D^\top\lint P\rint C)(t))X(t)+\int^T_t(B^\top\lint P)(s,t)\Theta(s,t)\,\diff s+\rho(t)+(D^\top\lint P\rint\sigma)(t),u(t)\Bigr\rangle\\
	&\hspace{1.5cm}+\langle(-P^{(1)}(t)+Q(t)+(C^\top\lint P\rint C)(t))X(t),X(t)\rangle\\
	&\hspace{1.5cm}+2\int^T_t\langle(-P^{(2)}(s,t,t)+(P\rint A)(s,t))X(t),\Theta(s,t)\rangle\,\diff s\\
	&\hspace{1.5cm}+\int^T_t\!\!\int^T_t\langle\dot{P}^{(2)}(s_1,s_2,t)\Theta(s_2,t),\Theta(s_1,t)\rangle\,\diff s_1\!\,\diff s_2\\
	&\hspace{1.5cm}+2\langle q(t)+(C^\top\lint P\rint\sigma)(t),X(t)\rangle+2\int^T_t\langle(P\rint b)(s,t),\Theta(s,t)\rangle\,\diff s\Bigr\}\,\diff t\Bigr].
\end{split}
\end{equation}

For each $(\Xi,\Gamma,v)\in\cS(0,T)$, let $P=(P^{(1)},P^{(2)})\in\Pi(0,T)$ be the solution to the Lyapunov--Volterra equation \eqref{eq_Lyapunov--Volterra_(Xi,Gamma)} and $(\eta,\zeta)\in L^2_{\bF,\mathrm{c}}(\triangle_2(0,T);\bR^d)\times L^2_\bF(\triangle_2(0,T);\bR^d)$ be the adapted-solution to the Type-II EBSVIE \eqref{eq_Type-II_EBSVIE_(Xi,Gamma,v)}. Let $(t_0,x)\in\cI$, $\tilde{v}\in\cU(0,T)$ and $\mu\in\bR$ be given, and denote by $(\tilde{X}^{t_0,x}_1,\tilde{\Theta}^{t_0,x}_1,\tilde{u}^{t_0,x}_1)\in L^2_\bF(t_0,T;\bR^d)\times L^2_{\bF,\mathrm{c}}(\triangle_2(t_0,T);\bR^d)\times\cU(t_0,T)$ the causal feedback solution to the controlled SVIE \eqref{eq_state} at $(t_0,x)\in\cI$ corresponding to the causal feedback strategy $(\Xi,\Gamma,v+\mu \tilde{v})\in\cS(0,T)$. Considering $u(t)=(\Xi,\Gamma,v+\mu\tilde{v})[t_0,x](t)=\Xi(t)\tilde{X}^{t_0,x}_1(t)+\int^T_t\Gamma(s,t)\tilde{\Theta}^{t_0,x}_1(s,t)\,\diff s+v(t)+\mu \tilde{v}(t)$ for $t\in(t_0,T)$, by \eqref{eq_J}, we have
\begin{align*}
	&J(t_0,x;(\Xi,\Gamma,v+\mu \tilde{v})[t_0,x])\\
	&=\int^T_{t_0}\langle P^{(1)}(t)x(t),x(t)\rangle\,\diff t+\int^T_{t_0}\!\!\int^T_{t_0}\langle P^{(2)}(s_1,s_2,t_0)x(s_2),x(s_1)\rangle\,\diff s_1\!\,\diff s_2+\bE\Bigl[\int^T_{t_0}(\sigma^\top\lint P\rint\sigma)(t)\,\diff t\Bigr]\\
	&\hspace{0.5cm}+\bE\Bigl[\int^T_{t_0}\Bigl\{\Bigl\langle(R(t)+(D^\top\lint P\rint D)(t))\Bigl(\Xi(t)\tilde{X}^{t_0,x}_1(t)+\int^T_t\Gamma(s,t)\tilde{\Theta}^{t_0,x}_1(s,t)\,\diff s+v(t)+\mu \tilde{v}(t)\Bigr),\\
	&\hspace{4cm}\Xi(t)\tilde{X}^{t_0,x}_1(t)+\int^T_t\Gamma(s,t)\tilde{\Theta}^{t_0,x}_1(s,t)\,\diff s+v(t)+\mu \tilde{v}(t)\Bigr\rangle\\
	&\hspace{1.5cm}+2\Bigl\langle(S(t)+(D^\top\lint P\rint C)(t))\tilde{X}^{t_0,x}_1(t)+\int^T_t(B^\top\lint P)(s,t)\tilde{\Theta}^{t_0,x}_1(s,t)\,\diff s+\rho(t)+(D^\top\lint P\rint\sigma)(t),\\
	&\hspace{4cm}\Xi(t)\tilde{X}^{t_0,x}_1(t)+\int^T_t\Gamma(s,t)\tilde{\Theta}^{t_0,x}_1(s,t)\,\diff s+v(t)+\mu \tilde{v}(t)\Bigr\rangle\\
	&\hspace{1.5cm}+\langle(-P^{(1)}(t)+Q(t)+(C^\top\lint P\rint C)(t))\tilde{X}^{t_0,x}_1(t),\tilde{X}^{t_0,x}_1(t)\rangle\\
	&\hspace{1.5cm}+2\int^T_t\langle(-P^{(2)}(s,t,t)+(P\rint A)(s,t))\tilde{X}^{t_0,x}_1(t),\tilde{\Theta}^{t_0,x}_1(s,t)\rangle\,\diff s\\
	&\hspace{1.5cm}+\int^T_t\!\!\int^T_t\langle\dot{P}^{(2)}(s_1,s_2,t)\tilde{\Theta}^{t_0,x}_1(s_2,t),\tilde{\Theta}^{t_0,x}_1(s_1,t)\rangle\,\diff s_1\!\,\diff s_2\\
	&\hspace{1.5cm}+2\langle q(t)+(C^\top\lint P\rint\sigma)(t),\tilde{X}^{t_0,x}_1(t)\rangle+2\int^T_t\langle(P\rint b)(s,t),\tilde{\Theta}^{t_0,x}_1(s,t)\rangle\,\diff s\Bigr\}\,\diff t\Bigr]\\
	&=\int^T_{t_0}\langle P^{(1)}(t)x(t),x(t)\rangle\,\diff t+\int^T_{t_0}\!\!\int^T_{t_0}\langle P^{(2)}(s_1,s_2,t_0)x(s_2),x(s_1)\rangle\,\diff s_1\!\,\diff s_2\\
	&\hspace{0.5cm}+\bE\Bigl[\int^T_{t_0}\Bigl\{(\sigma^\top\lint P\rint\sigma)(t)+\langle(R(t)+(D^\top\lint P\rint D)(t))v(t),v(t)\rangle+2\langle\rho(t)+(D^\top\lint P\rint\sigma)(t),v(t)\rangle\Bigr\}\,\diff t\Bigr]\\
	&\hspace{0.5cm}+\mu^2\bE\Bigl[\int^T_{t_0}\langle(R(t)+(D^\top\lint P\rint D)(t))\tilde{v}(t),\tilde{v}(t)\rangle\,\diff t\Bigr]\\
	&\hspace{0.5cm}+2\mu\bE\Bigl[\int^T_{t_0}\Bigl\langle(S(t)+(D^\top\lint P\rint C)(t)+(R(t)+(D^\top\lint P\rint D)(t))\Xi(t))\tilde{X}^{t_0,x}_1(t)\\
	&\hspace{3cm}+\int^T_t((B^\top\rint P)(s,t)+(R(t)+(D^\top\lint P\rint D)(t))\Gamma(s,t))\tilde{\Theta}^{t_0,x}_1(s,t)\,\diff s\\
	&\hspace{3cm}+\rho(t)+(D^\top\lint P\rint\sigma)(t)+(R(t)+(D^\top\lint P\rint D)(t))v(t),\tilde{v}(t)\Bigr\rangle\,\diff t\Bigr]\\
	&\hspace{0.5cm}+I_1+2I_2,
\end{align*}
where
\begin{align*}
	I_1&:=\bE\Bigl[\int^T_{t_0}\Bigl\{\langle(-P^{(1)}(t)+F^{(1)}[\Xi;P](t)+Q^{(1)}[\Xi](t))\tilde{X}^{t_0,x}_1(t),\tilde{X}^{t_0,x}_1(t)\rangle\\
	&\hspace{0.5cm}+2\int^T_t\langle(-P^{(2)}(s,t,t)+F^{(2)}[\Xi,\Gamma;P](s,t)+Q^{(2)}[\Xi,\Gamma](s,t))\tilde{X}^{t_0,x}_1(t),\tilde{\Theta}^{t_0,x}_1(s,t)\rangle\,\diff s\\
	&\hspace{0.5cm}+\int^T_t\!\!\int^T_t\langle(\dot{P}^{(2)}(s_1,s_2,t)+F^{(3)}[\Gamma;P](s_1,s_2,t)+Q^{(3)}[\Gamma](s_1,s_2,t))\tilde{\Theta}^{t_0,x}_1(s_2,t),\tilde{\Theta}^{t_0,x}_1(s_1,t)\rangle\,\diff s_1\!\,\diff s_2\Bigr\}\,\diff t\Bigr]
\end{align*}
and
\begin{align*}
	I_2&:=\bE\Bigl[\int^T_{t_0}\Bigl\{\langle q(t)+(C^\top\lint P\rint\sigma)(t)+\Xi(t)^\top(\rho(t)+(D^\top\lint P\rint\sigma)(t))\\
	&\hspace{3cm}+(S(t)^\top+(C^\top\lint P\rint D)(t)+\Xi(t)^\top(R(t)+(D^\top\lint P\rint D)(t)))v(t),\tilde{X}^{t_0,x}_1(t)\rangle\\
	&\hspace{2cm}+\int^T_t\langle(P\rint b)(s,t)+\Gamma(s,t)^\top(\rho(t)+(D^\top\lint P\rint\sigma)(t))\\
	&\hspace{3cm}+((P\rint B)(s,t)+\Gamma(s,t)^\top(R(t)+(D^\top\lint P\rint D)(t)))v(t),\tilde{\Theta}^{t_0,x}_1(s,t)\rangle\,\diff s\Bigr\}\,\diff t\Bigr].
\end{align*}
Since $P=(P^{(1)},P^{(2)})\in\Pi(0,T)$ solves the Lyapunov--Volterra equation \eqref{eq_Lyapunov--Volterra_(Xi,Gamma)}, we have $I_1=0$. Furthermore, since $(\eta,\zeta)\in L^2_{\bF,\mathrm{c}}(\triangle_2(0,T);\bR^d)\times L^2_\bF(\triangle_2(0,T);\bR^d)$ is the adapted solution to the Type-II EBSVIE \eqref{eq_Type-II_EBSVIE_(Xi,Gamma,v)}, the duality principle (\cref{theo_Type-II_EBSVIE}) yields that
\begin{align*}
	I_2&=\int^T_{t_0}\langle\bE[\eta(t,t_0)],x(t)\rangle\,\diff t+\bE\Bigl[\int^T_{t_0}\Bigl\{\int^T_t\langle\eta(s,t),b(s,t)\rangle\,\diff s+\int^T_t\langle\zeta(s,t),\sigma(s,t)\rangle\,\diff s\Bigr\}\,\diff t\Bigr]\\
	&\hspace{0.5cm}+\bE\Bigl[\int^T_{t_0}\Bigl\langle\int^T_tB(s,t)^\top\eta(s,t)\,\diff s+\int^T_tD(s,t)^\top\zeta(s,t)\,\diff s,v(t)\Bigr\rangle\,\diff t\Bigr]\\
	&\hspace{0.5cm}+\mu\bE\Bigl[\int^T_{t_0}\Bigl\langle\int^T_tB(s,t)^\top\eta(s,t)\,\diff s+\int^T_tD(s,t)^\top\zeta(s,t)\,\diff s,\tilde{v}(t)\Bigr\rangle\,\diff t\Bigr].
\end{align*}
Thus, we have
\begin{align*}
	&J(t_0,x;(\Xi,\Gamma,v+\mu \tilde{v})[t_0,x])\\
	&=\int^T_{t_0}\langle P^{(1)}(t)x(t),x(t)\rangle\,\diff t+\int^T_{t_0}\!\!\int^T_{t_0}\langle P^{(2)}(s_1,s_2,t_0)x(s_2),x(s_1)\rangle\,\diff s_1\!\,\diff s_2+2\int^T_{t_0}\langle\bE[\eta(t,t_0)],x(t)\rangle\,\diff t\\
	&\hspace{0.5cm}+\bE\Bigl[\int^T_{t_0}\Bigl\{(\sigma^\top\lint P\rint\sigma)(t)+2\int^T_t\langle\eta(s,t),b(s,t)\rangle\,\diff s+2\int^T_t\langle\zeta(s,t),\sigma(s,t)\rangle\,\diff s\\
	&\hspace{2cm}+\langle(R(t)+(D^\top\lint P\rint D)(t))v(t),v(t)\rangle+2\langle\kappa(t),v(t)\rangle\Bigr\}\,\diff t\Bigr]\\
	&\hspace{0.5cm}+\mu^2\bE\Bigl[\int^T_{t_0}\langle(R(t)+(D^\top\lint P\rint D)(t))\tilde{v}(t),\tilde{v}(t)\rangle\,\diff t\Bigr]\\
	&\hspace{0.5cm}+2\mu\bE\Bigl[\int^T_{t_0}\Bigl\langle(S(t)+(D^\top\lint P\rint C)(t)+(R(t)+(D^\top\lint P\rint D)(t))\Xi(t))\tilde{X}^{t_0,x}_1(t)\\
	&\hspace{3cm}+\int^T_t((B^\top\rint P)(s,t)+(R(t)+(D^\top\lint P\rint D)(t))\Gamma(s,t))\tilde{\Theta}^{t_0,x}_1(s,t)\,\diff s\\
	&\hspace{3cm}+\kappa(t)+(R(t)+(D^\top\lint P\rint D)(t))v(t),\tilde{v}(t)\Bigr\rangle\,\diff t\Bigr],
\end{align*}
where
\begin{equation*}
	\kappa(t):=\rho(t)\mathalpha{+}(D^\top\lint P\rint\sigma)(t)\mathalpha{+}\int^T_tB(s,t)^\top\eta(s,t)\,\diff s\mathalpha{+}\int^T_tD(s,t)^\top\zeta(s,t)\,\diff s,\ t\in(0,T).
\end{equation*}
By the uniqueness of the causal feedback solution to the controlled SVIE, we see that $\tilde{X}^{t_0,x}_1=X^{t_0,x}+\mu \tilde{X}^{t_0,0}_0$ and $\tilde{\Theta}^{t_0,x}_1=\Theta^{t_0,x}+\mu\tilde{\Theta}^{t_0,0}_0$, where $(X^{t_0,x},\Theta^{t_0,x})$ is the causal feedback solution to the controlled SVIE \eqref{eq_state_0} at $(t_0,x)\in\cI$ corresponding to $(\Xi,\Gamma,v)\in\cS(0,T)$, and $(\tilde{X}^{t_0,0}_0,\tilde{\Theta}^{t_0,0}_0)$ is the causal feedback solution to the homogeneous controlled SVIE \eqref{eq_state_0} at $(t_0,0)\in\cI$ corresponding to $(\Xi,\Gamma,\tilde{v})\in\cS(0,T)$. Inserting these formulae into the above expression, we get \eqref{eq_representation_cost}. This completes the proof.
\end{proof}


\begin{rem}\label{rem_expression_cost_functional}
\cref{theo_representation_cost} provides us useful expressions of the cost functionals. With the notations in the above theorem, for any $(t_0,x)\in\cI$ and $\tilde{v}\in\cU(t_0,T)$, the following hold:
\begin{align*}
	&J(t_0,x;(\Xi,\Gamma,v)[t_0,x])\\
	&=\int^T_{t_0}\langle P^{(1)}(t)x(t),x(t)\rangle\,\diff t+\int^T_{t_0}\!\!\int^T_{t_0}\langle P^{(2)}(s_1,s_2,t_0)x(s_2),x(s_1)\rangle\,\diff s_1\!\,\diff s_2+2\int^T_{t_0}\langle\bE[\eta(t,t_0)],x(t)\rangle\,\diff t\\
	&\hspace{0.5cm}+\bE\Bigl[\int^T_{t_0}\Bigl\{(\sigma^\top\lint P\rint\sigma)(t)+2\int^T_t\langle\eta(s,t),b(s,t)\rangle\,\diff s+2\int^T_t\langle\zeta(s,t),\sigma(s,t)\rangle\,\diff s\\
	&\hspace{2cm}+\langle(R(t)+(D^\top\lint P\rint D)(t))v(t),v(t)\rangle+2\langle\kappa(t),v(t)\rangle\Bigr\}\,\diff t\Bigr],\\
	&J^0(t_0,0;(\Xi,\Gamma,\tilde{v})^0[t_0,0])\\
	&=\bE\Bigl[\int^T_{t_0}\Bigl\{\langle(R(t)+(D^\top\lint P\rint D)(t))\tilde{v}(t),\tilde{v}(t)\rangle\\
	&\hspace{2cm}+2\Bigl\langle(S(t)+(D^\top\lint P\rint C)(t)+(R(t)+(D^\top\lint P\rint D)(t))\Xi(t))\tilde{X}^{t_0,0}_0(t)\\
	&\hspace{3cm}+\int^T_t((B^\top\rint P)(s,t)+(R(t)+(D^\top\lint P\rint D)(t))\Gamma(s,t))\tilde{\Theta}^{t_0,0}_0(s,t)\,\diff s,\tilde{v}(t)\Bigr\rangle\Bigr\}\,\diff t\Bigr],\\
	&\cD_v J(t_0,x;(\Xi,\Gamma,v)[t_0,x])(t)\\
	&=2\Bigl\{(S(t)+(D^\top\lint P\rint C)(t)+(R(t)+(D^\top\lint P\rint D)(t))\Xi(t))X^{t_0,x}(t)\\
	&\hspace{1cm}+\int^T_t((B^\top\rint P)(s,t)+(R(t)+(D^\top\lint P\rint D)(t))\Gamma(s,t))\Theta^{t_0,x}(s,t)\,\diff s\\
	&\hspace{1cm}+\kappa(t)+(R(t)+(D^\top\lint P\rint D)(t))v(t)\Bigr\},\ t\in(t_0,T),
\end{align*}
where $\cD_v J(t_0,x;(\Xi,\Gamma,v)[t_0,x])\in\cU(t_0,T)$ denotes the Fr\'{e}chet derivative of the functional $\cU(t_0,T)\ni v\mapsto J(t_0,x;(\Xi,\Gamma,v)[t_0,x])\in\bR$ evaluated at $v\in\cU(t_0,T)$. Furthermore, we obtain the following formula:
\begin{equation}\label{eq_cost_expansion}
\begin{split}
	&J(t_0,x;(\Xi,\Gamma,v+\mu \tilde{v})[t_0,x])\\
	&=J(t_0,x;(\Xi,\Gamma,v)[t_0,x])+\mu^2J^0(t_0,0;(\Xi,\Gamma,\tilde{v})^0[t_0,0])+2\mu\langle\cD_v J(t_0,x;(\Xi,\Gamma,v)[t_0,x]),\tilde{v}\rangle_{\cU(t_0,T)},\ \forall\,\mu\in\bR,
\end{split}
\end{equation}
where $\langle\cdot,\cdot\rangle_{\cU(t_0,T)}$ denotes the inner product in the Hilbert space $\cU(t_0,T)=L^2_\bF(t_0,T;\bR^\ell)$.
\end{rem}


\begin{rem}
By the definitions \eqref{eq_Lyapunov--Volterra_coefficients} and \eqref{eq_Lyapunov--Volterra_inhomogeneous} of the coefficient $(F^{(1)}[\Xi;P],F^{(2)}[\Xi,\Gamma;P],F^{(3)}[\Gamma;P])$ and the inhomogeneous term $(Q^{(1)}[\Xi],Q^{(2)}[\Xi,\Gamma],Q^{(3)}[\Gamma])$, respectively, the Lyapunov--Volterra equation \eqref{eq_Lyapunov--Volterra_(Xi,Gamma)} is written as
\begin{equation*}
	\begin{dcases}
	P^{(1)}(t)=Q(t)+(C^\top\lint P\rint C)(t)+\Xi(t)^\top(S(t)+(D^\top\lint P\rint C)(t))+(S(t)^\top+(C^\top\lint P\rint D)(t))\Xi(t)\\
	\hspace{2cm}+\Xi(t)^\top(R(t)+(D^\top\lint P\rint D)(t))\Xi(t),\ t\in(0,T),\\
	P^{(2)}(s,t,t)=P^{(2)}(t,s,t)^\top=(P\rint A)(s,t)+(P\rint B)(s,t)\Xi(t)+\Gamma(s,t)^\top(S(t)+(D^\top\lint P\rint C)(t))\\
	\hspace{5cm}+\Gamma(s,t)^\top(R(t)+(D^\top\lint P\rint D)(t))\Xi(t),\ (s,t)\in\triangle_2(0,T),\\
	\dot{P}^{(2)}(s_1,s_2,t)+\Gamma(s_1,t)^\top(B^\top\lint P)(s_2,t)+(P\rint B)(s_1,t)\Gamma(s_2,t)\\
	\hspace{1cm}+\Gamma(s_1,t)^\top(R(t)+(D^\top\lint P\rint D)(t))\Gamma(s_2,t)=0,\ (s_1,s_2,t)\in\square_3(0,T).
	\end{dcases}
\end{equation*}
Noting \cref{defi_lint_rint}, the above can be also written in the integral form.
\end{rem}


\section{Optimal strategies and Riccati--Volterra equations}\label{section_optimality}

In this section, we characterize causal feedback optimal strategies of Problem (SVC) by means of a Riccati-type equation. We introduce the following equation (depending only on the coefficients $A,B,C,D,Q,R$ and $S$):
\begin{equation}\label{eq_Riccati--Volterra}
	\begin{dcases}
	P^{(1)}(t)=Q(t)+(C^\top\lint P\rint C)(t)\\
	\hspace{1.5cm}-(S(t)^\top+(C^\top\lint P\rint D)(t))(R(t)+(D^\top\lint P\rint D)(t))^\dagger(S(t)+(D^\top\lint P\rint C)(t)),\ t\in(0,T),\\
	P^{(2)}(s,t,t)=P^{(2)}(t,s,t)^\top\\
	=(P\rint A)(s,t)-(P\rint B)(s,t)(R(t)+(D^\top\lint P\rint D)(t))^\dagger(S(t)+(D^\top\lint P\rint C)(t)),\ (s,t)\in\triangle_2(0,T),\\
	\dot{P}^{(2)}(s_1,s_2,t)=(P\rint B)(s_1,t)(R(t)+(D^\top\lint P\rint D)(t))^\dagger(B^\top\lint P)(s_2,t),\ (s_1,s_2,t)\in\square_3(0,T),
	\end{dcases}
\end{equation}
where, for each matrix $M$, $M^\dagger$ denotes the Moore--Penrose pseudoinverse (see \cite[Appendix A]{SuYo20}). Noting \cref{defi_lint_rint}, the above can be written in the integral form. This is a coupled system of Riccati-type (backward) Volterra integro-differential equations for the pair $P=(P^{(1)},P^{(2)})$ of matrix-valued deterministic functions, and we call it a \emph{Riccati--Volterra equation}. By a solution to the above Riccati--Volterra equation, we mean a pair $P=(P^{(1)},P^{(2)})\in\Pi(0,T)$ satisfying \eqref{eq_Riccati--Volterra}. Similarly to the study on LQ control problems for SDEs \cite{SuYo20}, we introduce the notions of the regular and the strongly regular solutions to the Riccati--Volterra equation \eqref{eq_Riccati--Volterra}.


\begin{defi}
Let $P=(P^{(1)},P^{(2)})\in\Pi(0,T)$ be a solution to the Riccati--Volterra equation \eqref{eq_Riccati--Volterra}. We say that it is \emph{regular} if
\begin{itemize}
\item[(i)]
$R(t)+(D^\top\lint P\rint D)(t)\geq0$ for a.e.\ $t\in(0,T)$;
\item[(ii)]
$\sR(S(t)+(D^\top\lint P\rint C)(t))\subset\sR(R(t)+(D^\top\lint P\rint D)(t))$ for a.e.\ $t\in(0,T)$, and the function
\begin{equation}\label{eq_check_Xi}
	\check{\Xi}(t):=-(R(t)+(D^\top\lint P\rint D)(t))^\dagger(S(t)+(D^\top\lint P\rint C)(t)),\ t\in(0,T),
\end{equation}
is in $L^\infty(0,T;\bR^{\ell\times d})$;
\item[(iii)]
$\sR((B^\top\lint P)(s,t))\subset\sR(R(t)+(D^\top\lint P\rint D)(t))$ for a.e.\ $(s,t)\in\triangle_2(0,T)$, and the function
\begin{equation}\label{eq_check_Gamma}
	\check{\Gamma}(s,t):=-(R(t)+(D^\top\lint P\rint D)(t))^\dagger(B^\top\lint P)(s,t),\ (s,t)\in\triangle_2(0,T),
\end{equation}
is in $L^2(\triangle_2(0,T);\bR^{\ell\times d})$.
\end{itemize}
Furthermore, we say that the solution $P$ is \emph{strongly regular} if there exists a constant $\lambda>0$ such that
\begin{equation*}
	R(t)+(D^\top\lint P\rint D)(t)\geq\lambda I_\ell
\end{equation*}
for a.e.\ $t\in(0,T)$.
\end{defi}


\begin{rem}
Noting that $S+D^\top\lint P\rint C\in L^\infty(0,T;\bR^{\ell\times d})$ and $B^\top\lint P\in L^2(\triangle_2(0,T);\bR^{\ell\times d})$ for any $P\in\Pi(0,T)$, we see that a strongly regular solution is a regular solution.
\end{rem}

The following is the main theorem of this section.


\begin{theo}\label{theo_optimality}
Problem (SVC) has a causal feedback optimal strategy if and only if the following two conditions hold:
\begin{itemize}
\item[(i)]
The Riccati--Volterra equation \eqref{eq_Riccati--Volterra} admits a regular solution $P=(P^{(1)},P^{(2)})\in\Pi(0,T)$.
\item[(ii)]
Define $(\check{\Xi},\check{\Gamma})\in L^\infty(0,T;\bR^{\ell\times d})\times L^2(\triangle_2(0,T);\bR^{\ell\times d})$ by \eqref{eq_check_Xi} and \eqref{eq_check_Gamma}, and let $(\eta,\zeta)\in L^2_{\bF,\mathrm{c}}(\triangle_2(0,T);\bR^d)\times L^2_\bF(\triangle_2(0,T);\bR^d)$ be the adapted solution to the Type-II EBSVIE
\begin{equation}\label{eq_Type-II_EBSVIE+}
	\begin{dcases}
	\mathrm{d}\eta(t,s)=-\Bigl\{(P\rint b)(t,s)+\check{\Gamma}(t,s)^\top(\rho(s)+(D^\top\lint P\rint\sigma)(s))\\
	\hspace{3cm}+\check{\Gamma}(t,s)^\top\int^T_sB(r,s)^\top\eta(r,s)\,\diff r+\check{\Gamma}(t,s)^\top\int^T_sD(r,s)^\top\zeta(r,s)\,\diff r\Bigr\}\,\diff s\\
	\hspace{2cm}+\zeta(t,s)\,\diff W(s),\ (t,s)\in\triangle_2(0,T),\\
	\eta(t,t)=q(t)+(C^\top\lint P\rint\sigma)(t)+\check{\Xi}(t)^\top(\rho(t)+(D^\top\lint P\rint\sigma)(t))\\
	\hspace{2cm}+\int^T_t(A+B\triangleright\check{\Xi})(r,t)^\top\eta(r,t)\,\diff r+\int^T_t(C+D\triangleright\check{\Xi})(r,t)^\top\zeta(r,t)\,\diff r,\ t\in(0,T).
	\end{dcases}	
\end{equation}
Then the process $\kappa\in L^2_\bF(0,T;\bR^\ell)$ defined by
\begin{equation*}
	\kappa(t):=\rho(t)\mathalpha{+}(D^\top\lint P\rint\sigma)(t)\mathalpha{+}\int^T_tB(s,t)^\top\eta(s,t)\,\diff s\mathalpha{+}\int^T_tD(s,t)^\top\zeta(s,t)\,\diff s,\ t\in(0,T),
\end{equation*}
satisfies
\begin{equation*}
	\kappa(t)\in\sR(R(t)\mathalpha{+}(D^\top\lint P\rint D)(t))
\end{equation*}
for a.e.\ $t\in(0,T)$, a.s., and the process $\check{v}$ defined by
\begin{equation}\label{eq_check_v}
	\check{v}(t):=-(R(t)\mathalpha{+}(D^\top\lint P\rint D)(t))^\dagger\kappa(t),\ t\in(0,T),
\end{equation}
is in $\cU(0,T)=L^2_\bF(0,T;\bR^\ell)$.
\end{itemize}
In this case, any causal feedback optimal strategy $(\hat{\Xi},\hat{\Gamma},\hat{v})\in\cS(0,T)$ of Problem (SVC) admits the following representation:
\begin{align*}
	&\hat{\Xi}(t)=\check{\Xi}(t)+(I_\ell\mathalpha{-}(R(t)\mathalpha{+}(D^\top\lint P\rint D)(t))^\dagger(R(t)\mathalpha{+}(D^\top\lint P\rint D)(t)))\tilde{\Xi}(t),\ t\in(0,T),\\
	&\hat{\Gamma}(s,t)=\check{\Gamma}(s,t)+(I_\ell\mathalpha{-}(R(t)\mathalpha{+}(D^\top\lint P\rint D)(t))^\dagger(R(t)\mathalpha{+}(D^\top\lint P\rint D)(t)))\tilde{\Gamma}(s,t),\ (s,t)\in\triangle_2(0,T),\\
	&\hat{v}(t)=\check{v}(t)+(I_\ell\mathalpha{-}(R(t)\mathalpha{+}(D^\top\lint P\rint D)(t))^\dagger(R(t)\mathalpha{+}(D^\top\lint P\rint D)(t)))\tilde{v}(t),\ t\in(0,T),
\end{align*}
with $(\tilde{\Xi},\tilde{\Gamma},\tilde{v})\in\cS(0,T)$ being arbitrary. Furthermore, the value functional is given by
\begin{align*}
	V(t_0,x)&=\int^T_{t_0}\langle P^{(1)}(t)x(t),x(t)\rangle\,\diff t+\int^T_{t_0}\!\!\int^T_{t_0}\langle P^{(2)}(t_1,t_2,t_0)x(t_2),x(t_1)\rangle\,\diff t_1\!\,\diff t_2+2\int^T_{t_0}\langle\bE[\eta(t,t_0)],x(t)\rangle\,\diff t\\
	&\hspace{0.5cm}+\bE\Bigl[\int^T_{t_0}\Bigl\{(\sigma^\top\lint P\rint\sigma)(t)+2\int^T_t\langle\eta(s,t),b(s,t)\rangle\,\diff s+2\int^T_t\langle\zeta(s,t),\sigma(s,t)\rangle\,\diff s\\
	&\hspace{2cm}-\langle(R(t)\mathalpha{+}(D^\top\lint P\rint D)(t))^\dagger\kappa(t),\kappa(t)\rangle\Bigr\}\,\diff t\Bigr]
\end{align*}
for each $(t_0,x)\in\cI$.
\end{theo}

\begin{proof}
\underline{The sufficiency}: Assume that the conditions (i) and (ii) in the theorem hold, and define $(\check{\Xi},\check{\Gamma},\check{v})\in\cS(0,T)$ by \eqref{eq_check_Xi}, \eqref{eq_check_Gamma} and \eqref{eq_check_v}. Furthermore, define $(\hat{\Xi},\hat{\Gamma},\hat{v})\in\cS(0,T)$ by
\begin{align*}
	&\hat{\Xi}(t)=\check{\Xi}(t)+(I_\ell\mathalpha{-}(R(t)\mathalpha{+}(D^\top\lint P\rint D)(t))^\dagger(R(t)\mathalpha{+}(D^\top\lint P\rint D)(t)))\tilde{\Xi}(t),\ t\in(0,T),\\
	&\hat{\Gamma}(s,t)=\check{\Gamma}(s,t)+(I_\ell\mathalpha{-}(R(t)\mathalpha{+}(D^\top\lint P\rint D)(t))^\dagger(R(t)\mathalpha{+}(D^\top\lint P\rint D)(t)))\tilde{\Gamma}(s,t),\ (s,t)\in\triangle_2(0,T),\\
	&\hat{v}(t)=\check{v}(t)+(I_\ell\mathalpha{-}(R(t)\mathalpha{+}(D^\top\lint P\rint D)(t))^\dagger(R(t)\mathalpha{+}(D^\top\lint P\rint D)(t)))\tilde{v}(t),\ t\in(0,T),
\end{align*}
with $(\tilde{\Xi},\tilde{\Gamma},\tilde{v})\in\cS(0,T)$ being arbitrary. We show that $(\hat{\Xi},\hat{\Gamma},\hat{v})$ is a causal feedback optimal strategy of Problem (SVC). By the constructions and a fundamental calculus of the Moore--Penrose pseudoinverse (see \cite[Proposition A.15]{SuYo20} and its proof), we have
\begin{align*}
	&S(t)+(D^\top\lint P\rint C)(t)+(R(t)+(D^\top\lint P\rint D)(t))\hat{\Xi}(t)=0\ \text{for a.e.\ $t\in(0,T)$},\\
	&(B^\top\lint P)(s,t)+(R(t)+(D^\top\lint P\rint D)(t))\hat{\Gamma}(s,t)=0\ \text{for a.e.\ $(s,t)\in\triangle_2(0,T)$},\\
	&\kappa(t)+(R(t)+(D^\top\lint P\rint D)(t))\hat{v}(t)=0\ \text{for a.e.\ $t\in(0,T)$, a.s.}
\end{align*}
Inserting the above formulae into the Riccati--Volterra equation \eqref{eq_Riccati--Volterra} and the Type-II EBSVIE \eqref{eq_Type-II_EBSVIE+}, we see that
\begin{itemize}
\item
$P=(P^{(1)},P^{(2)})\in\Pi(0,T)$ solves the Lyapunov--Volterra equation \eqref{eq_Lyapunov--Volterra_(Xi,Gamma)} with $(\Xi,\Gamma)=(\hat{\Xi},\hat{\Gamma})$;
\item
$(\eta,\zeta)\in L^2_{\bF,\mathrm{c}}(\triangle_2(0,T);\bR^d)\times L^2_\bF(\triangle_2(0,T);\bR^d)$ solves the Type-II EBSVIE \eqref{eq_Type-II_EBSVIE_(Xi,Gamma,v)} with $(\Xi,\Gamma,v)=(\hat{\Xi},\hat{\Gamma},\hat{v})$.
\end{itemize}
Furthermore, we have
\begin{equation*}
	\langle(R(t)+(D^\top\lint P\rint D)(t))\hat{v}(t),\hat{v}(t)\rangle+2\langle\kappa(t),\hat{v}(t)\rangle=-\langle(R(t)+(D^\top\lint P\rint D)(t))^\dagger\kappa(t),\kappa(t)\rangle
\end{equation*}
for a.e.\ $t\in(0,T)$, a.s. Thus, by \cref{theo_representation_cost}, for any $(t_0,x)\in\cI$, $\tilde{v}\in\cU(0,T)$ and $\mu\in\bR$, we have
\begin{align*}
	&J(t_0,x;(\hat{\Xi},\hat{\Gamma},\hat{v}+\mu \tilde{v})[t_0,x])\\
	&=\int^T_{t_0}\langle P^{(1)}(t)x(t),x(t)\rangle\,\diff t+\int^T_{t_0}\!\!\int^T_{t_0}\langle P^{(2)}(t_1,t_2,t_0)x(t_2),x(t_1)\rangle\,\diff t_1\!\,\diff t_2+2\int^T_{t_0}\langle\bE[\eta(t,t_0)],x(t)\rangle\,\diff t\\
	&\hspace{0.5cm}+\bE\Bigl[\int^T_{t_0}\Bigl\{(\sigma^\top\lint P\rint\sigma)(t)+2\int^T_t\langle\eta(s,t),b(s,t)\rangle\,\diff s+2\int^T_t\langle\zeta(s,t),\sigma(s,t)\rangle\,\diff s\\
	&\hspace{2cm}-\langle(R(t)\mathalpha{+}(D^\top\lint P\rint D)(t))^\dagger\kappa(t),\kappa(t)\rangle\Bigr\}\,\diff t\Bigr]\\
	&\hspace{0.5cm}+\mu^2\bE\Bigl[\int^T_{t_0}\langle(R(t)+(D^\top\lint P\rint D)(t))\tilde{v}(t),\tilde{v}(t)\rangle\,\diff t\Bigr].
\end{align*}
Since $R(t)+(D^\top\lint P\rint D)(t)\geq0$ for a.e.\ $t\in(0,T)$, we see that
\begin{equation*}
	J(t_0,x;(\hat{\Xi},\hat{\Gamma},\hat{v}+\mu \tilde{v})[t_0,x])\geq J(t_0,x;(\hat{\Xi},\hat{\Gamma},\hat{v})[t_0,x])
\end{equation*}
for any $(t_0,x)\in\cI$, $\tilde{v}\in\cU(0,T)$ and $\mu\in\bR$. This implies that
\begin{equation*}
	J(t_0,x;(\hat{\Xi},\hat{\Gamma},v)[t_0,x])\geq J(t_0,x;(\hat{\Xi},\hat{\Gamma},\hat{v})[t_0,x])
\end{equation*}
for any $(t_0,x)\in\cI$ and $v\in\cU(0,T)$. By \cref{lemm_causal feedback_equivalence}, we see that $(\hat{\Xi},\hat{\Gamma},\hat{v})$ is a causal feedback optimal strategy of Problem (SVC). Furthermore, again by \cref{lemm_causal feedback_equivalence}, we obtain the expression of the value functional:
\begin{align*}
	V(t_0,x)&=\inf_{u\in\cU(t_0,T)} J(t_0,x;u)\\
	&=J(t_0,x;(\hat{\Xi},\hat{\Gamma},\hat{v})[t_0,x])\\
	&=\int^T_{t_0}\langle P^{(1)}(t)x(t),x(t)\rangle\,\diff t+\int^T_{t_0}\!\!\int^T_{t_0}\langle P^{(2)}(t_1,t_2,t_0)x(t_2),x(t_1)\rangle\,\diff t_1\!\,\diff t_2+2\int^T_{t_0}\langle\bE[\eta(t,t_0)],x(t)\rangle\,\diff t\\
	&\hspace{0.5cm}+\bE\Bigl[\int^T_{t_0}\Bigl\{(\sigma^\top\lint P\rint\sigma)(t)+2\int^T_t\langle\eta(s,t),b(s,t)\rangle\,\diff s+2\int^T_t\langle\zeta(s,t),\sigma(s,t)\rangle\,\diff s\\
	&\hspace{2cm}-\langle(R(t)\mathalpha{+}(D^\top\lint P\rint D)(t))^\dagger\kappa(t),\kappa(t)\rangle\Bigr\}\,\diff t\Bigr].
\end{align*}

\underline{The necessity}: Assume that $(\hat{\Xi},\hat{\Gamma},\hat{v})\in\cS(0,T)$ is a causal feedback optimal strategy of Problem (SVC). Let $P=(P^{(1)},P^{(2)})\in\Pi(0,T)$ be the solution to the Lyapunov--Volterra equation \eqref{eq_Lyapunov--Volterra_(Xi,Gamma)} with $(\Xi,\Gamma)=(\hat{\Xi},\hat{\Gamma})$, and let $(\eta,\zeta)\in L^2_{\bF,\mathrm{c}}(\triangle_2(0,T);\bR^d)\times L^2_\bF(\triangle_2(0,T);\bR^d)$ be the adapted solution to the Type-II EBSVIE \eqref{eq_Type-II_EBSVIE_(Xi,Gamma,v)} with $(\Xi,\Gamma,v)=(\hat{\Xi},\hat{\Gamma},\hat{v})$. By the optimality of $(\hat{\Xi},\hat{\Gamma},\hat{v})$, we have
\begin{equation*}
	J(t_0,x;(\hat{\Xi},\hat{\Gamma},\hat{v}+\mu \tilde{v})[t_0,x])-J(t_0,x;(\hat{\Xi},\hat{\Gamma},\hat{v})[t_0,x])\geq0
\end{equation*}
for any $(t_0,x)\in\cI$, $\tilde{v}\in\cU(0,T)$ and $\mu\in\bR$. By \cref{theo_representation_cost}, we must have
\begin{equation}\label{eq_optimality_1}
\begin{split}
	&\bE\Bigl[\int^T_{0}\Bigl\{\langle (R(t)+(D^\top\lint P\rint D)(t))\tilde{v}(t),\tilde{v}(t)\rangle\\
	&\hspace{1cm}+2\Bigl\langle\bigl(S(t)+(D^\top\lint P\rint C)(t)+(R(t)+(D^\top\lint P\rint D)(t))\hat{\Xi}(t)\bigr)\tilde{X}^{0,0}_0(t)\\
	&\hspace{2cm}+\int^T_t\bigl((B^\top\lint P)(s,t)+(R(t)+(D^\top\lint P\rint D)(t))\hat{\Gamma}(s,t)\bigr)\tilde{\Theta}^{0,0}_0(s,t)\,\diff s,\tilde{v}(t)\Bigr\rangle\Bigr\}\,\diff t\Bigr]\geq0\\
	&\hspace{8cm}\forall\,\tilde{v}\in\cU(0,T),
\end{split}
\end{equation}
with $(\tilde{X}^{0,0}_0,\tilde{\Theta}^{0,0}_0)$ being the causal feedback solution to the homogeneous controlled SVIE \eqref{eq_state_0} with respect to the zero input condition $(t_0,x)=(0,0)\in\cI$ and the causal feedback strategy $(\hat{\Xi},\hat{\Gamma},\tilde{v})\in\cS(0,T)$, and
\begin{equation}\label{eq_optimality_2}
\begin{split}
	&\bigl(S(t)+(D^\top\lint P\rint C)(t)+(R(t)+(D^\top\lint P\rint D)(t))\hat{\Xi}(t)\bigr)\hat{X}^{t_0,x}(t)\\
	&\hspace{0.5cm}+\int^T_t\bigl((B^\top\lint P)(s,t)+(R(t)+(D^\top\lint P\rint D)(t))\hat{\Gamma}(s,t)\bigr)\hat{\Theta}^{t_0,x}(s,t)\,\diff s\\
	&\hspace{0.5cm}+\kappa(t)+(R(t)+(D^\top\lint P\rint D)(t))\hat{v}(t)=0\\
	&\hspace{4cm}\text{for a.e.\ $t\in(t_0,T)$, a.s., $\forall\,(t_0,x)\in\cI$},
\end{split}
\end{equation}
with $(\hat{X}^{t_0,x},\hat{\Theta}^{t_0,x})$ being the causal feedback solution to the controlled SVIE \eqref{eq_state} with respect to the input condition $(t_0,x)\in\cI$ and the causal feedback strategy $(\hat{\Xi},\hat{\Gamma},\hat{v})\in\cS(0,T)$. By subtracting the left-hand sides of \eqref{eq_optimality_2} corresponding to the free terms $x$ and $0$, the later from the former, and then taking the expectations, we see that
\begin{equation}\label{eq_optimality_2'}
\begin{split}
	&\bigl(S(t)+(D^\top\lint P\rint C)(t)+(R(t)+(D^\top\lint P\rint D)(t))\hat{\Xi}(t)\bigr)\bE[\hat{X}^{t_0,x}_0(t)]\\
	&\hspace{0.5cm}+\int^T_t\bigl((B^\top\lint P)(s,t)+(R(t)+(D^\top\lint P\rint D)(t))\hat{\Gamma}(s,t)\bigr)\bE[\hat{\Theta}^{t_0,x}_0(s,t)]\,\diff s=0\\
	&\hspace{4cm}\text{for a.e.\ $t\in(t_0,T)$, $\forall\,(t_0,x)\in\cI$},
\end{split}
\end{equation}
where $\hat{X}^{t_0,x}_0:=\hat{X}^{t_0,x}-\hat{X}^{t_0,0}$ and $\hat{\Theta}^{t_0,x}_0:=\hat{\Theta}^{t_0,x}-\hat{\Theta}^{t_0,0}$. Observe that $(\hat{X}^{t_0,x}_0,\hat{\Theta}^{t_0,x}_0)$ is the causal feedback solution to the homogeneous controlled SVIE \eqref{eq_state_0} with respect to the input condition $(t_0,x)\in\cI$ and the causal feedback strategy $(\hat{\Xi},\hat{\Gamma},0)\in\cS(0,T)$. Thus, applying \cref{lemm_MN=0} to \eqref{eq_optimality_2'}, we obtain
\begin{equation}\label{eq_optimality_Xi}
	S(t)+(D^\top\lint P\rint C)(t)+(R(t)+(D^\top\lint P\rint D)(t))\hat{\Xi}(t)=0\ \text{for a.e.\ $t\in(0,T)$}
\end{equation}
and
\begin{equation}\label{eq_optimality_Gamma}
	(B^\top\lint P)(s,t)+(R(t)+(D^\top\lint P\rint D)(t))\hat{\Gamma}(s,t)=0\ \text{for a.e.\ $(s,t)\in\triangle_2(0,T)$.}
\end{equation}
Also, by \eqref{eq_optimality_2},
\begin{equation}\label{eq_optimality_v}
	\kappa(t)+(R(t)+(D^\top\lint P\rint D)(t))\hat{v}(t)=0\ \text{for a.e.\ $t\in(0,T)$, a.s.}
\end{equation}
Slightly modifying \cite[Proposition A.15]{SuYo20}, we obtain from \eqref{eq_optimality_Xi}, \eqref{eq_optimality_Gamma} and \eqref{eq_optimality_v} the following three assertions.
\begin{itemize}
\item[(i)]
$\sR(S(t)+(D^\top\lint P\rint C)(t))\subset\sR(R(t)+(D^\top\lint P\rint D)(t))$ for a.e.\ $t\in(0,T)$, and the function $\check{\Xi}$ given by
\begin{equation*}
	\check{\Xi}(t):=-(R(t)+(D^\top\lint P\rint D)(t))^\dagger(S(t)+(D^\top\lint P\rint C)(t)),\ t\in(0,T),
\end{equation*}
is in $L^\infty(0,T;\bR^{\ell\times d})$. Furthermore, $\hat{\Xi}$ is of the form
\begin{equation*}
	\hat{\Xi}(t)=\check{\Xi}(t)+(I_\ell-(R(t)+(D^\top\lint P\rint D)(t))^\dagger(R(t)+(D^\top\lint P\rint D)(t)))\tilde{\Xi}(t),\ t\in(0,T),
\end{equation*}
for some $\tilde{\Xi}\in L^\infty(0,T;\bR^{\ell\times d})$.
\item[(ii)]
$\sR((B^\top\lint P)(s,t))\subset\sR(R(t)+(D^\top\lint P\rint D)(t))$ for a.e.\ $(s,t)\in\triangle_2(0,T)$, and the function $\check{\Gamma}$ given by
\begin{equation*}
	\check{\Gamma}(s,t):=-(R(t)+(D^\top\lint P\rint D)(t))^\dagger(B^\top\lint P)(s,t),\ (s,t)\in\triangle_2(0,T),
\end{equation*}
is in $L^2(\triangle_2(0,T);\bR^{\ell\times d})$. Furthermore, $\hat{\Gamma}$ is of the form
\begin{equation*}
	\hat{\Gamma}(s,t)=\check{\Gamma}(s,t)+(I_\ell-(R(t)+(D^\top\lint P\rint D)(t))^\dagger(R(t)+(D^\top\lint P\rint D)(t)))\tilde{\Gamma}(s,t),\ (s,t)\in\triangle_2(0,T),
\end{equation*}
for some $\tilde{\Gamma}\in L^2(\triangle_2(0,T);\bR^{\ell\times d})$.
\item[(iii)]
$\kappa(t)\in\sR(R(t)+(D^\top\lint P\rint D)(t))$ for a.e.\ $t\in(0,T)$, a.s., and the process $\check{v}$ given by
\begin{equation*}
	\check{v}(t):=-(R(t)+(D^\top\lint P\rint D)(t))^\dagger\kappa(t),\ t\in(0,T),
\end{equation*}
is in $\cU(0,T)=L^2_\bF(0,T;\bR^\ell)$. Furthermore, $\hat{v}$ is of the form
\begin{equation*}
	\hat{v}(t)=\check{v}(t)+(I_\ell-(R(t)+(D^\top\lint P\rint D)(t))^\dagger(R(t)+(D^\top\lint P\rint D)(t)))\tilde{v}(t),\ t\in(0,T),
\end{equation*}
for some $\tilde{v}\in\cU(0,T)$.
\end{itemize}
By inserting the above expressions of $(\hat{\Xi},\hat{\Gamma},\hat{v})\in\cS(0,T)$ into the equations \eqref{eq_Lyapunov--Volterra_(Xi,Gamma)} and \eqref{eq_Type-II_EBSVIE_(Xi,Gamma,v)}, we see that the pair $P=(P^{(1)},P^{(2)})\in\Pi(0,T)$ solves the Riccati--Volterra equation \eqref{eq_Riccati--Volterra} and that the pair $(\eta,\zeta)$ is the adapted solution to the Type-II EBSVIE \eqref{eq_Type-II_EBSVIE+}. Furthermore, By \eqref{eq_optimality_Xi} and \eqref{eq_optimality_Gamma}, the condition \eqref{eq_optimality_1} becomes
\begin{equation*}
	\bE\Bigl[\int^T_{0}\langle (R(t)+(D^\top\lint P\rint D)(t))\tilde{v}(t),\tilde{v}(t)\rangle\,\diff t\Bigr]\geq0,\ \forall\,\tilde{v}\in\cU(0,T).
\end{equation*}
This immediately implies that $R(t)+(D^\top\lint P\rint D)(t)\geq0$ for a.e.\ $t\in(0,T)$. Therefore, $P$ is a regular solution of the Riccati--Volterra equation \eqref{eq_Riccati--Volterra}. This completes the proof.
\end{proof}

If the Riccati--Volterra equation \eqref{eq_Riccati--Volterra} admits a strongly regular solution $P=(P^{(1)},P^{(2)})\in\Pi(0,T)$, then the term $R+D^\top\lint P\rint D$ is invertible with the bounded inverse matrix $(R+D^\top\lint P\rint D)^\dagger=(R+D^\top\lint P\rint D)^{-1}$, and the condition (ii) in \cref{theo_optimality} automatically holds. Thus, we have the following corollary.


\begin{cor}\label{cor_unique_optimal_strategy}
Suppose that the Riccati--Volterra equation \eqref{eq_Riccati--Volterra} admits a strongly regular solution $P=(P^{(1)},P^{(2)})\in\Pi(0,T)$. Then Problem (SVC) has a unique causal feedback optimal strategy $(\hat{\Xi},\hat{\Gamma},\hat{v})\in\cS(0,T)$ given by
\begin{equation}\label{eq_unique_optimal_strategy}
\begin{split}
	&\hat{\Xi}(t)=-(R(t)\mathalpha{+}(D^\top\lint P\rint D)(t))^{-1}(S(t)\mathalpha{+}(D^\top\lint P\rint C)(t)),\ t\in(0,T),\\
	&\hat{\Gamma}(s,t)=-(R(t)\mathalpha{+}(D^\top\lint P\rint D)(t))^{-1}(B^\top\lint P)(s,t),\ (s,t)\in\triangle_2(0,T),\\
	&\hat{v}(t)=-(R(t)\mathalpha{+}(D^\top\lint P\rint D)(t))^{-1}\Bigl(\rho(t)\mathalpha{+}(D^\top\lint P\rint\sigma)(t)\mathalpha{+}\int^T_tB(s,t)^\top\eta(s,t)\,\diff s\mathalpha{+}\int^T_tD(s,t)^\top\zeta(s,t)\,\diff s\Bigr),\\
	&\hspace{8cm}t\in(0,T),
\end{split}
\end{equation}
where $(\eta,\zeta)\in L^2_{\bF,\mathrm{c}}(\triangle_2(0,T);\bR^d)\times L^2_\bF(\triangle_2(0,T);\bR^d)$ is the unique adapted solution to the Type-II EBSVIE \eqref{eq_Type-II_EBSVIE+}.
\end{cor}

In the homogeneous Problem (SVC)$^0$, the unique adapted solution of the Type-II EBSVIE \eqref{eq_Type-II_EBSVIE+} is $(\eta,\zeta)=(0,0)$. Thus, we immediately obtain the following corollary.


\begin{cor}
The homogeneous Problem (SVC)$^0$ has a causal feedback optimal strategy if and only if the Riccati--Volterra equation \eqref{eq_Riccati--Volterra} admits a regular solution $P=(P^{(1)},P^{(2)})\in\Pi(0,T)$. In this case, the value functional is given by
\begin{equation*}
	V^0(t_0,x)=\int^T_{t_0}\langle P^{(1)}(t)x(t),x(t)\rangle\,\diff t+\int^T_{t_0}\!\!\int^T_{t_0}\langle P^{(2)}(t_1,t_2,t_0)x(t_2),x(t_1)\rangle\,\diff t_1\!\,\diff t_2
\end{equation*}
for each $(t_0,x)\in\cI$.
\end{cor}

From the above representation formula of the homogeneous value functional $V^0(t_0,x)$ and \cref{lemm_Pi_0}, we get the following uniqueness result of the regular solution to the Riccati--Volterra equation.


\begin{cor}
The Riccati--Volterra equation \eqref{eq_Riccati--Volterra} has at most one regular solution.
\end{cor}


\begin{rem}\label{rem_B=0}
Consider the case where the control does not enter the drift part, that is, $B=0$. In this case, the Riccati--Volterra equation \eqref{eq_Riccati--Volterra} becomes
\begin{equation*}
	\begin{dcases}
	P^{(1)}(t)=Q(t)+(C^\top\lint P\rint C)(t)\\
	\hspace{1.5cm}-(S(t)^\top+(C^\top\lint P\rint D)(t))(R(t)+(D^\top\lint P\rint D)(t))^\dagger(S(t)+(D^\top\lint P\rint C)(t)),\ t\in(0,T),\\
	P^{(2)}(s,t,t)=P^{(2)}(t,s,t)^\top=(P\rint A)(s,t),\ (s,t)\in\triangle_2(0,T),\\
	\dot{P}^{(2)}(s_1,s_2,t)=0,\ (s_1,s_2,t)\in\square_3(0,T).
	\end{dcases}
\end{equation*}
Thus, if there exists a (strongly) regular solution $P=(P^{(1)},P^{(2)})\in\Pi(0,T)$ to the above Riccati--Volterra equation, then $P^{(2)}(s_1,s_2,t)$ does not depend on the last parameter $t\in(0,s_1\wedge s_2)$. Furthermore, the function $\check{\Gamma}$ in \eqref{eq_check_Gamma} vanishes. Therefore, in this case, there exists a (unique) causal feedback optimal strategy $(\hat{\Xi},\hat{\Gamma},\hat{v})$ of Problem (SVC) with $\hat{\Gamma}=0$. In other words, the (unique) causal feedback optimal strategy is a \emph{state-feedback form} in the sense that it does not use the feedback of the forward state process $\Theta$. Furthermore, in the case of the homogeneous Problem (SVC)$^0$, the stochastic inhomogeneous term $\hat{v}$ can be zero. In this case, the (unique) causal feedback optimal strategy is a \emph{Markovian state-feedback form} in the sense that it is just a deterministic linear functional of the current state. This is a surprising consequence since, even in the homogeneous Problem (SVC)$^0$ with $B=0$, the state process is highly non-Markovian and being non-semimartingale due to the Volterra structure.
\end{rem}


\begin{rem}
Very recently, a similar fact as in \cref{rem_B=0} was also found in an independent work of Wang, Yong and Zhou \cite{WaYoZh22}, where an LQ stochastic Volterra control problem (involving a terminal cost) was studied in the open-loop framework. In \cite{WaYoZh22}, the coefficients $A,B,C$ and $D$ are non-convolution-type but assumed to be regular (i.e.\ bounded and differentiable), the inhomogeneous terms $b,\sigma,q$ and $\rho$ are zeros, $S=0$, and the weighting matrices $Q$ and $R$ are assumed to be non-negative and strictly positive definite, respectively. By a dynamic programming method and a decoupling technique, they derived a causal feedback represention of the open-loop optimal control by means of a path-dependent (operator-valued) Riccati equation, which is different from our Riccati--Volterra equation \eqref{eq_Riccati--Volterra}. Compared to \cite{WaYoZh22}, our problem is in the closed-loop framework, and the Riccati--Volterra equation \eqref{eq_Riccati--Volterra} is a system of integro-differential equations for the (finite-dimensional) kernels $P=(P^{(1)},P^{(2)})$ of a self-adjoint operator $\cP^{t_0}$ (see \cref{lemm_self-adjoint}).
\end{rem}


\section{Strongly regular solvability of the Riccati--Volterra equation}\label{section_strongly_regular_solvability}

As we have seen in the previous section, any causal feedback optimal strategies of Problem (SVC) are characterized by using the (unique) regular solution of the Riccati--Volterra equation \eqref{eq_Riccati--Volterra}. Also, the existence of the strongly regular solution, which is stronger than the regular solution, implies the uniqueness of the causal feedback optimal strategy.

In this section, we prove the equivalence between the strongly regular solvability of Riccati--Volterra equation \eqref{eq_Riccati--Volterra} and the uniform convexity of the cost functional. Furthermore, we provide a sufficient condition for the two equivalent properties.


\begin{defi}
Let $(H,\|\cdot\|_H)$ be a Hilbert space, and consider a functional $F:H\to\bR$. We say that $F$ is uniformly convex if there exists a constant $\lambda>0$ such that, for any $u_1,u_2\in H$ and $\mu\in[0,1]$, it holds that
\begin{equation*}
	F((1-\mu)u_1+\mu u_2)\leq(1-\mu)F(u_1)+\mu F(u_2)-\lambda(1-\mu)\mu\|u_1-u_2\|^2_H.
\end{equation*}
\end{defi}


\begin{lemm}\label{lemm_uniform_convex}
The following are equivalent:
\begin{itemize}
\item[(i)]
The functional $\cU(0,T)\ni u\mapsto J^0(0,0;u)$ is uniformly convex;
\item[(i)']
There exists a constant $\lambda>0$ such that $J^0(0,0;u)\geq\lambda\bE[\int^T_{0}|u(t)|^2\,\diff t]$ for any $u\in\cU(0,T)$;
\item[(ii)]
For some $(\Xi,\Gamma)\in L^\infty(0,T;\bR^{\ell\times d})\times L^2(\triangle_2(0,T);\bR^{\ell\times d})$, the functional $\cU(0,T)\ni v\mapsto J^0(0,0;(\Xi,\Gamma,v)^0[0,0])$ is uniformly convex;
\item[(ii)']
There exist $(\Xi,\Gamma)\in L^\infty(0,T;\bR^{\ell\times d})\times L^2(\triangle_2(0,T);\bR^{\ell\times d})$ and  $\lambda>0$ such that $J^0(0,0;(\Xi,\Gamma,v)^0[0,0])\geq\lambda\bE[\int^T_{0}|v(t)|^2\,\diff t]$ for any $v\in\cU(0,T)$;
\item[(iii)]
There exists a constant $\lambda>0$ such that $J^0(t_0,0;u)\geq\lambda\bE[\int^T_{t_0}|u(t)|^2\,\diff t]$ for any $t_0\in[0,T)$ and any $u\in\cU(t_0,T)$.
\end{itemize}
If one of the above conditions holds, then there exist $\lambda>0$ and $\alpha\in\bR$ such that, for any $(\Xi,\Gamma)\in L^\infty(0,T;\bR^{\ell\times d})\times L^2(\triangle_2(0,T);\bR^{\ell\times d})$, the solution $P=(P^{(1)},P^{(2)})\in\Pi(0,T)$ to the Lyapunov--Volterra equation \eqref{eq_Lyapunov--Volterra_(Xi,Gamma)} satisfies
\begin{equation}\label{eq_estimate_lambda}
	R(t)+(D^\top\lint P\rint D)(t)\geq\lambda I_\ell
\end{equation}
for a.e.\ $t\in(0,T)$, and
\begin{equation}\label{eq_estimate_alpha}
	\int^T_{t_0}\langle P^{(1)}(t)x(t),x(t)\rangle\,\diff t+\int^T_{t_0}\!\!\int^T_{t_0}\langle P^{(2)}(s_1,s_2,t_0)x(s_2),x(s_1)\rangle\,\diff s_1\!\,\diff s_2\geq\alpha\int^T_{t_0}|x(t)|^2\,\diff t
\end{equation}
for any $(t_0,x)\in\cI$.
\end{lemm}


\begin{proof}
Noting \eqref{eq_cost_expansion}, we can easily show the equivalences (i)\,$\Leftrightarrow$\,(i)' and (ii)\,$\Leftrightarrow$\,(ii)', and we omit the details. The implications (iii)\,$\Rightarrow$\,(i)'\,$\Rightarrow$\,(ii)' are trivial.

\underline{(ii)'\,$\Rightarrow$\,(i)'}: Suppose that (ii)' holds. By \cref{lemm_strategy_bijective}, the map $\cT:v\mapsto(\Xi,\Gamma,v)^0[0,0]$ is a bijective bounded linear operator on the Hilbert space $\cU(0,T)$. By the inverse mapping theorem, the inverse $\cT^{-1}$ is a bounded linear operator on $\cU(0,T)$, and thus the operator norm $\|\cT\|_\op$ of $\cT$ is positive. Thus, for any $u\in\cU(0,T)$, we have
\begin{equation*}
	J^0(0,0;u)=J^0(0,0;(\Xi,\Gamma,\cT^{-1}u)^0[0,0])\geq\lambda\bE\Bigl[\int^T_0|(\cT^{-1}u)(t)|^2\,\diff t\Bigr]\geq\lambda\|\cT\|^{-2}_\op\bE\Bigl[\int^T_0|u(t)|^2\,\diff t\Bigr].
\end{equation*}
Hence, (i)' holds.

\underline{(i)'\,$\Rightarrow$\,(iii)}: Suppose that (i)' holds. For any $t_0\in[0,T)$ and $u\in\cU(t_0,T)$, denote by $X^{t_0,u}_0$ the corresponding state process with the free term and the inhomogeneous terms being zero. Define the zero extension of $u$ by $\tilde{u}(t):=0$ for $t\in(0,t_0]$ and $\tilde{u}(t):=u(t)$ for $t\in(t_0,T)$. Then $\tilde{u}\in\cU(0,T)$. Furthermore, by the uniqueness of the solution to the SVIE, the corresponding state process satisfies $X^{0,\tilde{u}}_0(t)=0$ for $t\in(0,t_0]$ and $X^{0,\tilde{u}}_0(t)=X^{t_0,u}_0(t)$ for $t\in(t_0,T)$. Therefore, we get
\begin{equation*}
	J^0(t_0,0;u)=J^0(0,0;\tilde{u})\geq\lambda\bE\Bigl[\int^T_{0}|\tilde{u}(t)|^2\,\diff t\Bigr]=\lambda\bE\Bigl[\int^T_{t_0}|u(t)|^2\,\diff t\Bigr],
\end{equation*}
and thus (iii) holds. This completes the proof of the equivalence of (i),(i)',(ii),(ii)' and (iii).

\underline{Proof of the last assertion}: Assume that (i) holds with the constant $\lambda>0$. Let $(\Xi,\Gamma)\in L^\infty(0,T;\bR^{\ell\times d})\times L^2(\triangle_2(0,T);\bR^{\ell\times d})$ be given, and let $P=(P^{(1)},P^{(2)})\in\Pi(0,T)$ be the solution to the Lyapunov--Volterra equation \eqref{eq_Lyapunov--Volterra_(Xi,Gamma)}. By \cref{theo_representation_cost}, for any $\tilde{v}\in\cU(0,T)$, we have
\begin{align*}
	&J^0(0,0;(\Xi,\Gamma,\tilde{v})^0[0,0])\\
	&=\bE\Bigl[\int^T_0\Bigl\{\langle(R(t)+(D^\top\lint P\rint D)(t))\tilde{v}(t),\tilde{v}(t)\rangle\\
	&\hspace{2cm}+2\Bigl\langle(S(t)+(D^\top\lint P\rint C)(t)+(R(t)+(D^\top\lint P\rint D)(t))\Xi(t))\tilde{X}_0(t)\\
	&\hspace{3cm}+\int^T_t((B^\top\rint P)(s,t)+(R(t)+(D^\top\lint P\rint D)(t))\Gamma(s,t))\tilde{\Theta}_0(s,t)\,\diff s,\tilde{v}(t)\Bigr\rangle\Bigr\}\,\diff t\Bigr],
\end{align*}
where $(\tilde{X}_0,\tilde{\Theta}_0)=(\tilde{X}^{0,0}_0,\tilde{\Theta}^{0,0}_0)$ is the causal feedback solution to the homogeneous controlled SVIE \eqref{eq_state_0} with respect to the zero input condition $(t_0,x)=(0,0)\in\cI$ and the causal feedback strategy $(\Xi,\Gamma,\tilde{v})\in\cS(0,T)$. On the other hand, by the assumption, we have
\begin{align*}
	J^0(0,0;(\Xi,\Gamma,\tilde{v})^0[0,0])&\geq\lambda\bE\Bigl[\int^T_0|(\Xi,\Gamma,\tilde{v})^0[0,0](t)|^2\,\diff t\Bigr]\\
	&=\lambda\bE\Bigl[\int^T_0\Bigl|\Xi(t)\tilde{X}_0(t)+\int^T_t\Gamma(s,t)\tilde{\Theta}_0(s,t)\,\diff s+\tilde{v}(t)\Bigr|^2\,\diff t\Bigr]\\
	&\geq\lambda\bE\Bigl[\int^T_0\Bigl\{|\tilde{v}(t)|^2+2\Bigl\langle\Xi(t)\tilde{X}_0(t)+\int^T_t\Gamma(s,t)\tilde{\Theta}_0(s,t)\,\diff s,\tilde{v}(t)\Bigr\rangle\Bigr\}\,\diff t\Bigr].
\end{align*}
Thus, we have
\begin{equation*}
	\bE\Bigl[\int^T_0\Bigl\{\langle R_\lambda(t)\tilde{v}(t),\tilde{v}(t)\rangle+2\Bigl\langle\Xi_\lambda(t)\tilde{X}_0(t)+\int^T_t\Gamma_\lambda(s,t)\tilde{\Theta}_0(s,t)\,\diff s,\tilde{v}(t)\Bigr\rangle\Bigr\}\,\diff t\Bigr]\geq0
\end{equation*}
for any $\tilde{v}\in\cU(0,T)$, where
\begin{align*}
	&R_\lambda(t):=R(t)+(D^\top\lint P\rint D)(t)-\lambda I_\ell,\ t\in(0,T),\\
	&\Xi_\lambda(t):=S(t)+(D^\top\lint P\rint C)(t)+(R(t)+(D^\top\lint P\rint D)(t)-\lambda I_\ell)\Xi(t),\ t\in(0,T),\\
	&\Gamma_\lambda(s,t):=(B^\top\lint P)(s,t)+(R(t)+(D^\top\lint P\rint D)(t)-\lambda I_\ell)\Gamma(s,t),\ (s,t)\in\triangle_2(0,T).
\end{align*}
Now we consider $\tilde{v}^N(t)=\sqrt N\1_{[\tau,\tau+1/N]}(t)v$, $t\in(0,T)$, for arbitrary $\tau\in(0,T)$, $N\in\bN$ with $\tau+1/N<T$ and $v\in\bR^\ell$. Let $(\tilde{X}^N_0,\tilde{\Theta}^N_0)$ be the corresponding causal feedback solution. Then we have
\begin{equation}\label{eq_lambda_inequality}
	\Bigl\langle N\int^{\tau+1/N}_\tau R_\lambda(t)\,\diff t\,v,v\Bigr\rangle+2\Bigl\langle \sqrt{N}\int^{\tau+1/N}_\tau\Bigl\{\Xi_\lambda(t)\bE[\tilde{X}^N_0(t)]+\int^T_t\Gamma_\lambda(s,t)\bE[\tilde{\Theta}^N_0(s,t)]\,\diff s\Bigr\}\,\diff t,v\Bigr\rangle\geq0.
\end{equation}
On the one hand, by the Lebesgue differentiation theorem, we have $\lim_{N\to\infty}N\int^{\tau+1/N}_\tau R_\lambda(t)\,\diff t=R_\lambda(\tau)$ for a.e.\ $\tau\in(0,T)$. On the other hand, by \cref{lemm_expectation_explicit}, we have
\begin{align*}
	&\bE[\tilde{X}^N_0(t)]=\sqrt{N}\int^t_0G_2(t,t,\theta)\1_{[\tau,\tau+1/N]}(\theta)v\,\diff\theta,\ t\in(0,T),\\
	&\bE[\tilde{\Theta}^N_0(s,t)]=\sqrt{N}\int^t_0G_2(s,t,\theta)\1_{[\tau,\tau+1/N]}(\theta)v\,\diff\theta,\ (s,t)\in\triangle_2(0,T),
\end{align*}
with $G_2:\{(s,t,\theta)\in(0,T)^3\,|\,T>s\geq t\geq \theta>0\}\to\bR^{d\times \ell}$ satisfying the estimate \eqref{eq_G_i_estimate}. Thus, the second term of the left-hand side of \eqref{eq_lambda_inequality} becomes
\begin{equation*}
	\Bigl\langle \sqrt{N}\int^{\tau+1/N}_\tau\Bigl\{\Xi_\lambda(t)\bE[\tilde{X}^N_0(t)]+\int^T_t\Gamma_\lambda(s,t)\bE[\tilde{\Theta}^N_0(s,t)]\,\diff s\Bigr\}\,\diff t,v\Bigr\rangle=\Bigl\langle N\int^{\tau+1/N}_\tau\!\!\int^t_\tau G_\lambda(t,\theta)\,\diff \theta\!\,\diff t\,v,v\Bigr\rangle
\end{equation*}
with $G_\lambda:\triangle_2(0,T)\to\bR^{\ell\times\ell}$ defined by
\begin{equation*}
	G_\lambda(t,\theta):=\Xi_\lambda(t)G_2(t,t,\theta)+\int^T_t\Gamma_\lambda(s,t)G_2(s,t,\theta)\,\diff s,\ (t,\theta)\in\triangle_2(0,T).
\end{equation*}
Noting the estimate \eqref{eq_G_i_estimate}, we can easily show that $G_\lambda\in L^2(\triangle_2(0,T);\bR^{\ell\times\ell})$, and thus
\begin{equation*}
	N\int^{\tau+1/N}_\tau\!\!\int^t_\tau |G_\lambda(t,\theta)|\,\diff \theta\!\,\diff t\leq\frac{1}{\sqrt{2}}\Bigl(\int^{\tau+1/N}_\tau\!\!\int^t_\tau |G_\lambda(t,\theta)|^2\,\diff \theta\!\,\diff t\Bigr)^{1/2}\to0
\end{equation*}
as $N\to\infty$. This implies that the second term in the left-hand side of the inequality \eqref{eq_lambda_inequality} tends to zero  as $N\to\infty$ for any $\tau\in(0,T)$. Consequently, we have $\langle R_\lambda(\tau)v,v\rangle\geq0$ for a.e.\ $\tau\in(0,T)$ for any $v\in\bR^\ell$, which implies the estimate \eqref{eq_estimate_lambda}.

Lastly, we prove \eqref{eq_estimate_alpha}. Let $(t_0,x)\in\cI$ be fixed. By \eqref{eq_cost_expansion} and the condition (iii), for any $u\in\cU(t_0,T)$,
\begin{align*}
	J^0(t_0,x;u)&=J^0(t_0,x,0)+J^0(t_0,0;u)+2\langle\cD_uJ^0(t_0,x;0),u\rangle_{\cU(t_0,T)}\\
	&\geq J^0(t_0,x;0)+\lambda\|u\|^2_{\cU(t_0,T)}-2|\langle\cD_uJ^0(t_0,x;0),u\rangle_{\cU(t_0,T)}|\\
	&\geq J^0(t_0,x;0)-\frac{1}{\lambda}\|\cD_uJ^0(t_0,x;0)\|^2_{\cU(t_0,T)}.
\end{align*}
By virtue of the estimate \eqref{eq_closed-loop_estimate} of the state process and \cref{rem_expression_cost_functional} for the representation of $\cD_uJ$, it is easy to see that
\begin{equation*}
	|J^0(t_0,x;0)|+\|\cD_uJ^0(t_0,x;0)\|^2_{\cU(t_0,T)}\leq M\int^T_{t_0}|x(t)|^2\,\diff t
\end{equation*}
for some constant $M>0$ which is independent of $(t_0,x)$. Thus, by letting $\alpha:=-\max\{1,1/\lambda\}M$, we have $J^0(t_0,x;u)\geq\alpha\int^T_{t_0}|x(t)|^2\,\diff t$ for any $(t_0,x)\in\cI$ and $u\in\cU(t_0,T)$. In particular, for any $(\Xi,\Gamma)\in L^\infty(0,T;\bR^{\ell\times d})\times L^2(\triangle_2(0,T);\bR^{\ell\times d})$ and any $(t_0,x)\in\cI$,
\begin{equation*}
	J^0(t_0,x;(\Xi,\Gamma,0)^0[t_0,x])\geq\alpha\int^T_{t_0}|x(t)|^2\,\diff t.
\end{equation*}
By \cref{theo_representation_cost}, the left-hand side is equal to
\begin{equation*}
	\int^T_{t_0}\langle P^{(1)}(t)x(t),x(t)\rangle\,\diff t+\int^T_{t_0}\!\!\int^T_{t_0}\langle P^{(2)}(s_1,s_2,t_0)x(s_2),x(s_1)\rangle\,\diff s_1\!\,\diff s_2,
\end{equation*}
where $P=(P^{(1)},P^{(2)})\in\Pi(0,T)$ is the solution to the Lyapunov--Volterra equation \eqref{eq_Lyapunov--Volterra_(Xi,Gamma)}. Thus, the estimate \eqref{eq_estimate_alpha} holds. This completes the proof.
\end{proof}

The following corollary gives a simple sufficient condition for the uniform convexity of the cost functional.


\begin{cor}\label{cor_standard_condition}
Assume that the following \emph{standard condition} holds for some $\lambda>0$:
\begin{equation}\label{eq_standard_condition}
	R(t)\geq\lambda I_\ell\ \text{and}\ Q(t)-S(t)^\top R(t)^{-1}S(t)\geq0\ \text{for a.e.}\ t\in(0,T).
\end{equation}
Then the cost functional $\cU(0,T)\ni u\mapsto J^0(0,0;u)$ is uniformly convex.
\end{cor}


\begin{proof}
For any $v\in\cU(0,T)$, let $(X,\Theta,u)\in L^2_\bF(0,T;\bR^d)\times L^2_{\bF,\mathrm{c}}(\triangle_2(0,T);\bR^d)\times\cU(0,T)$ be the causal feedback solution of the homogeneous controlled SVIE \eqref{eq_state_0} at the zero input condition $(t_0,x)=(0,0)\in\cI$ corresponding to the causal feedback strategy $(-R^{-1}S,0,v)\in\cS(0,T)$. Then we have
\begin{align*}
	&J^0(0,0;(-R^{-1}S,0,v)^0[0,0])\\
	&=\bE\Bigl[\int^T_0\left\langle\begin{pmatrix}Q(t)&S(t)^\top\\S(t)&R(t)\end{pmatrix}\begin{pmatrix}X(t)\\u(t)\end{pmatrix},\begin{pmatrix}X(t)\\u(t)\end{pmatrix}\right\rangle\,\diff t\Bigr]\\
	&=\bE\Bigl[\int^T_0\left\langle\begin{pmatrix}Q(t)&S(t)^\top\\S(t)&R(t)\end{pmatrix}\begin{pmatrix}X(t)\\-R(t)^{-1}S(t)X(t)+v(t)\end{pmatrix},\begin{pmatrix}X(t)\\-R(t)^{-1}S(t)X(t)+v(t)\end{pmatrix}\right\rangle\,\diff t\Bigr]\\
	&=\bE\Bigl[\int^T_0\{\langle(Q(t)-S(t)^\top R(t)^{-1}S(t))X(t),X(t)\rangle+\langle R(t)v(t),v(t)\rangle\}\,\diff t\Bigr]\\
	&\geq\lambda\bE\Bigl[\int^T_0|v(t)|^2\,\diff t\Bigr].
\end{align*}
Thus, the condition (ii)' in \cref{lemm_uniform_convex} holds for $(\Xi,\Gamma)=(-R^{-1}S,0)$, and hence the cost functional $\cU(0,T)\ni u\mapsto J^0(0,0;u)$ is uniformly convex.
\end{proof}


\begin{rem}
In the works of \cite{AbMiPh21,AbMiPh21+,WaYoZh22} (where $S=0$), the standard condition \eqref{eq_standard_condition} is a priori assumed.
\end{rem}

The following is the main theorem of this section.


\begin{theo}\label{theo_strongly_regular_solvability}
The following are equivalent:
\begin{itemize}
\item[(i)]
The cost functional $\cU(0,T)\ni u\mapsto J^0(0,0;u)$ is uniformly convex;
\item[(iv)]
Riccati--Volterra equation \eqref{eq_Riccati--Volterra} admits a strongly regular solution.
\end{itemize}
\end{theo}


\begin{proof}
\underline{(iv)\,$\Rightarrow$\,(i)}: Suppose that the Riccati--Volterra equation \eqref{eq_Riccati--Volterra} has a strongly regular solution $P=(P^{(1)},P^{(2)})\in\Pi(0,T)$. By the definition, there exists a constant $\lambda>0$ such that $R(t)+(D^\top\lint P\rint D)(t)\geq\lambda I_\ell$ for a.e.\ $t\in(0,T)$. Define $(\check{\Xi},\check{\Gamma})\in L^\infty(0,T;\bR^{\ell\times d})\times L^2(\triangle_2(0,T);\bR^{\ell\times d})$ by
\begin{align*}
	&\check{\Xi}(t)=-(R(t)\mathalpha{+}(D^\top\lint P\rint D)(t))^{-1}(S(t)\mathalpha{+}(D^\top\lint P\rint C)(t)),\ t\in(0,T),\\
	&\check{\Gamma}(s,t)=-(R(t)\mathalpha{+}(D^\top\lint P\rint D)(t))^{-1}(B^\top\lint P)(s,t),\ (s,t)\in\triangle_2(0,T).
\end{align*}
Then it is easy to see that $P$ solves the Lyapunov--Volterra equation \eqref{eq_Lyapunov--Volterra_(Xi,Gamma)} with $(\Xi,\Gamma)=(\check{\Xi},\check{\Gamma})$. By \cref{theo_representation_cost}, for any $\tilde{v}\in\cU(0,T)$, we have
\begin{align*}
	&J^0(0,0;(\check{\Xi},\check{\Gamma},\tilde{v})^0[0,0])\\
	&=\bE\Bigl[\int^T_{0}\Bigl\{\langle(R(t)+(D^\top\lint P\rint D)(t))\tilde{v}(t),\tilde{v}(t)\rangle\\
	&\hspace{2cm}+2\Bigl\langle(S(t)+(D^\top\lint P\rint C)(t)+(R(t)+(D^\top\lint P\rint D)(t))\check{\Xi}(t))\tilde{X}^{0,0}_0(t)\\
	&\hspace{3cm}+\int^T_t((B^\top\lint P)(s,t)+(R(t)+(D^\top\lint P\rint D)(t))\check{\Gamma}(s,t))\tilde{\Theta}^{0,0}_0(s,t)\,\diff s,\tilde{v}(t)\Bigr\rangle\Bigr\}\,\diff t\Bigr]\\
	&=\bE\Bigl[\int^T_{0}\langle(R(t)+(D^\top\lint P\rint D)(t))\tilde{v}(t),\tilde{v}(t)\rangle\,\diff t\Bigr]\\
	&\geq\lambda\bE\Bigl[\int^T_{0}|\tilde{v}(t)|^2\,\diff t\Bigr],
\end{align*}
where $(\tilde{X}^{0,0}_0,\tilde{\Theta}^{0,0}_0)$ is the causal feedback solution to the homogeneous controlled SVIE \eqref{eq_state_0} with respect to the zero input condition $(t_0,x)=(0,0)\in\cI$ and the causal feedback strategy $(\check{\Xi},\check{\Gamma},\tilde{v})\in\cS(0,T)$. Thus, the condition (ii)' in \cref{lemm_uniform_convex} holds for $(\Xi,\Gamma)=(\check{\Xi},\check{\Gamma})$, and hence the cost functional $\cU(0,T)\ni u\mapsto J^0(0,0;u)$ is uniformly convex.

\underline{(i)\,$\Rightarrow$\,(iv)}: Suppose that the cost functional $\cU(0,T)\ni u\mapsto J^0(0,0;u)$ is uniformly convex. Let $\lambda>0$ and $\alpha\in\bR$ be the constants appearing in the last assertion of \cref{lemm_uniform_convex}. Recall the definitions \eqref{eq_Lyapunov--Volterra_coefficients} and \eqref{eq_Lyapunov--Volterra_inhomogeneous} of $(F^{(1)}[\Xi;P],F^{(2)}[\Xi,\Gamma;P],F^{(3)}[\Gamma;P])$ and $(Q^{(1)}[\Xi],Q^{(2)}[\Xi,\Gamma],Q^{(3)}[\Gamma])$, respectively. We define $(\Xi_i,\Gamma_i)\in L^\infty(0,T;\bR^{\ell\times d})\times L^2(\triangle_2(0,T);\bR^{\ell\times d})$ and $P_i=(P^{(1)}_i,P^{(2)}_i)\in\Pi(0,T)$ with $i\in\bN$ by the following induction: For $i=1$, let $\Xi_1:=0$ and $\Gamma_1:=0$, and let $P_1=(P^{(1)}_1,P^{(2)}_1)\in\Pi(0,T)$ be the solution to the Lyapunov--Volterra equation \eqref{eq_Lyapunov--Volterra_(Xi,Gamma)} with $(\Xi,\Gamma)=(\Xi_1,\Gamma_1)=(0,0)$:
\begin{equation*}
	\begin{dcases}
	P^{(1)}_1(t)=(C^\top\lint P_1\rint C)(t)+Q(t),\ t\in(0,T),\\
	P^{(2)}_1(s,t,t)=P^{(2)}_1(t,s,t)^\top=(P_1\rint A)(s,t),\ (s,t)\in\triangle_2(0,T),\\
	\dot{P}^{(2)}_1(s_1,s_2,t)=0,\ (s_1,s_2,t)\in\square_3(0,T).
	\end{dcases}
\end{equation*}
For $i\geq2$, define
\begin{align*}
	&\Xi_i(t):=-(R(t)+(D^\top\lint P_{i-1}\rint D)(t))^{-1}(S(t)+(D^\top\lint P_{i-1}\rint C)(t)),\ t\in(0,T),\\
	&\Gamma_i(s,t):=-(R(t)+(D^\top\lint P_{i-1}\rint D)(t))^{-1}(B^\top\lint P_{i-1})(s,t),\ (s,t)\in\triangle_2(0,T),
\end{align*}
and let $P_i=(P^{(1)}_i,P^{(2)}_i)\in\Pi(0,T)$ be the solution to the Lyapunov--Volterra equation \eqref{eq_Lyapunov--Volterra_(Xi,Gamma)} with $(\Xi,\Gamma)=(\Xi_i,\Gamma_i)$:
\begin{equation}\label{eq_Lyapunov--Volterra_i}
	\begin{dcases}
	P^{(1)}_i(t)=F^{(1)}[\Xi_i;P_i](t)+Q^{(1)}[\Xi_i](t),\ t\in(0,T),\\
	P^{(2)}_i(s,t,t)=P^{(2)}_i(t,s,t)^\top=F^{(2)}[\Xi_i,\Gamma_i;P_i](s,t)+Q^{(2)}[\Xi_i,\Gamma_i](s,t),\ (s,t)\in\triangle_2(0,T),\\
	\dot{P}^{(2)}_i(s_1,s_2,t)+F^{(3)}[\Gamma_i;P_i](s_1,s_2,t)+Q^{(3)}[\Gamma_i](s_1,s_2,t)=0,\ (s_1,s_2,t)\in\square_3(0,T).
	\end{dcases}
\end{equation}
We observe that the above induction is well-defined by \cref{theo_Lyapunov--Volterra}, together with the last assertion in \cref{lemm_uniform_convex}. Furthermore, for any $i\in\bN$, we have $R(t)+(D^\top\lint P_i\rint D)(t)\geq\lambda I_\ell$ for a.e.\ $t\in(0,T)$, and
\begin{equation*}
	\int^T_{t_0}\langle P^{(1)}_i(t)x(t),x(t)\rangle\,\diff t+\int^T_{t_0}\!\!\int^T_{t_0}\langle P^{(2)}_i(s_1,s_2,t_0)x(s_2),x(s_1)\rangle\,\diff s_1\!\,\diff s_2\geq\alpha\int^T_{t_0}|x(t)|^2\,\diff t
\end{equation*}
for any $(t_0,x)\in\cI$. We shall show that $\{P_i\}_{i\in\bN}$ converges (in a suitable sense) to the strongly regular solution of the Riccati--Volterra equation \eqref{eq_Riccati--Volterra}.

For each $i\in\bN$, define $\bar{P}_i=(\bar{P}^{(1)}_i,\bar{P}^{(2)}_i)\in\Pi(0,T)$ and $(\bar{\Xi}_i,\bar{\Gamma}_i)\in L^\infty(0,T;\bR^{\ell\times d})\times L^2(\triangle_2(0,T);\bR^{\ell\times d})$ by
\begin{equation*}
	\bar{P}^{(1)}_i:=P^{(1)}_i-P^{(1)}_{i+1},\ \bar{P}^{(2)}_i:=P^{(2)}_i-P^{(2)}_{i+1},\ \bar{\Xi}_i:=\Xi_i-\Xi_{i+1},\ \bar{\Gamma}_i:=\Gamma_i-\Gamma_{i+1}.
\end{equation*}
Noting that $P\mapsto F^{(1)}[\Xi;P]$ is linear, we have
\begin{align*}
	\bar{P}^{(1)}_i&=F^{(1)}[\Xi_i;P_i]+Q^{(1)}[\Xi_i]-F^{(1)}[\Xi_{i+1};P_{i+1}]-Q^{(1)}[\Xi_{i+1}]\\
	&=F^{(1)}[\Xi_{i+1};\bar{P}_i]+F^{(1)}[\Xi_i;P_i]+Q^{(1)}[\Xi_i]-F^{(1)}[\Xi_{i+1};P_i]-Q^{(1)}[\Xi_{i+1}]\\
	&=F^{(1)}[\Xi_{i+1};\bar{P}_i]+\bar{\Xi}_i^\top(S+D^\top\lint P_i\rint C)+(S^\top+C^\top\lint P_i\rint D)\bar{\Xi}_i\\
	&\hspace{0.5cm}+\bar{\Xi}^\top_i(R+D^\top\lint P_i\rint D)\Xi_{i+1}+\Xi^\top_{i+1}(R+D^\top\lint P_i\rint D)\bar{\Xi}_i+\bar{\Xi}^\top_i(R+D^\top\lint P_i\rint D)\bar{\Xi}_i\\
	&=F^{(1)}[\Xi_{i+1};\bar{P}_i]+\bar{\Xi}^\top_i(R+D^\top\lint P_i\rint D)\bar{\Xi}_i,
\end{align*}
where the last equality follows from the definition of $\Xi_{i+1}$. By similar calculations for $\bar{P}^{(2)}_i$ and $\dot{\bar{P}}^{(2)}_i$, we see that $\bar{P}_i=(\bar{P}^{(1)}_i,\bar{P}^{(2)}_i)\in\Pi(0,T)$ solves the following Lyapunov--Volterra equation:
\begin{equation*}
	\begin{dcases}
	\bar{P}^{(1)}_i(t)=F^{(1)}[\Xi_{i+1};\bar{P}_i](t)+\bar{Q}^{(1)}_i(t),\ t\in(0,T),\\
	\bar{P}^{(2)}_i(s,t,t)=\bar{P}^{(2)}_i(t,s,t)^\top=F^{(2)}[\Xi_{i+1},\Gamma_{i+1};\bar{P}_i](s,t)+\bar{Q}^{(2)}_i(s,t),\ (s,t)\in\triangle_2(0,T),\\
	\dot{\bar{P}}^{(2)}_i(s_1,s_2,t)+F^{(3)}[\Gamma_{i+1};\bar{P}_i](s_1,s_2,t)+\bar{Q}^{(3)}_i(s_1,s_2,t)=0,\ (s_1,s_2,t)\in\square_3(0,T),
	\end{dcases}
\end{equation*}
where
\begin{align*}
	&\bar{Q}^{(1)}_i(t):=\bar{\Xi}_i(t)^\top(R(t)+(D^\top\lint P_i\rint D)(t))\bar{\Xi}_i(t),\ t\in(0,T),\\
	&\bar{Q}^{(2)}_i(s,t):=\bar{\Gamma}_i(s,t)^\top(R(t)+(D^\top\lint P_i\rint D)(t))\bar{\Xi}_i(t),\ (s,t)\in\triangle_2(0,T),\\
	&\bar{Q}^{(3)}_i(s_1,s_2,t):=\bar{\Gamma}_i(s_1,t)^\top(R(t)+(D^\top\lint P_i\rint D)(t))\bar{\Gamma}_i(s_2,t),\ (s_1,s_2,t)\in\square_3(0,T).
\end{align*}
Therefore, by the representation formula (\cref{theo_Lyapunov--Volterra}), for any $(t_0,x)\in\cI$, we have
\begin{align*}
	&\int^T_{t_0}\langle \bar{P}^{(1)}_i(t)x(t),x(t)\rangle\,\diff t+\int^T_{t_0}\!\!\int^T_{t_0}\langle \bar{P}^{(2)}_i(s_1,s_2,t_0)x(s_2),x(s_1)\rangle\,\diff s_1\!\,\diff s_2\\
	&=\bE\Bigl[\int^T_{t_0}\Bigl\{\langle \bar{Q}^{(1)}_i(t)X^{t_0,x}_{i+1}(t),X^{t_0,x}_{i+1}(t)\rangle+2\int^T_t\langle \bar{Q}^{(2)}_i(s,t)X^{t_0,x}_{i+1}(t),\Theta^{t_0,x}_{i+1}(s,t)\rangle\,\diff s\\
	&\hspace{4cm}+\int^T_t\!\!\int^T_t\langle \bar{Q}^{(3)}_i(s_1,s_2,t)\Theta^{t_0,x}_{i+1}(s_2,t),\Theta^{t_0,x}_{i+1}(s_1,t)\rangle\,\diff s_1\!\,\diff s_2\Bigr\}\,\diff t\Bigr],
\end{align*}
where $(X^{t_0,x}_{i+1},\Theta^{t_0,x}_{i+1})$ is the causal feedback solution to the homogeneous controlled SVIE \eqref{eq_state_0} at $(t_0,x)\in\cI$ corresponding to the causal feedback strategy $(\Xi_{i+1},\Gamma_{i+1},0)\in\cS(0,T)$. From the definition of $(\bar{Q}^{(1)}_i,\bar{Q}^{(2)}_i,\bar{Q}^{(3)}_i)$ and the fact that $R(t)+(D^\top\lint P_i\rint D)(t)\geq\lambda I_\ell$ for a.e.\ $t\in(0,T)$, we get
\begin{align*}
	&\int^T_{t_0}\langle \bar{P}^{(1)}_i(t)x(t),x(t)\rangle\,\diff t+\int^T_{t_0}\!\!\int^T_{t_0}\langle \bar{P}^{(2)}_i(s_1,s_2,t_0)x(s_2),x(s_1)\rangle\,\diff s_1\!\,\diff s_2\\
	&=\bE\Bigl[\int^T_{t_0}\Bigl\langle(R(t)+(D^\top\lint P_i\rint D)(t))\Bigl(\bar{\Xi}_i(t)X^{t_0,x}_{i+1}(t)+\int^T_t\bar{\Gamma}_i(s,t)\Theta^{t_0,x}_{i+1}(s,t)\,\diff s\Bigr),\\
	&\hspace{6cm}\bar{\Xi}_i(t)X^{t_0,x}_{i+1}(t)+\int^T_t\bar{\Gamma}_i(s,t)\Theta^{t_0,x}_{i+1}(s,t)\,\diff s\Bigr\rangle\,\diff t\Bigr]\\
	&\geq0.
\end{align*}
Therefore, for any $(t_0,x)\in\cI$ and any $i\in\bN$, the following chain of inequalities holds:
\begin{equation}\label{eq_P_i_inequality}
\begin{split}
	&\hspace{0.5cm}\int^T_{t_0}\langle P^{(1)}_1(t)x(t),x(t)\rangle\,\diff t+\int^T_{t_0}\!\!\int^T_{t_0}\langle P^{(2)}_1(s_1,s_2,t_0)x(s_2),x(s_1)\rangle\,\diff s_1\!\,\diff s_2\\
	&\geq\int^T_{t_0}\langle P^{(1)}_i(t)x(t),x(t)\rangle\,\diff t+\int^T_{t_0}\!\!\int^T_{t_0}\langle P^{(2)}_i(s_1,s_2,t_0)x(s_2),x(s_1)\rangle\,\diff s_1\!\,\diff s_2\\
	&\geq\int^T_{t_0}\langle P^{(1)}_{i+1}(t)x(t),x(t)\rangle\,\diff t+\int^T_{t_0}\!\!\int^T_{t_0}\langle P^{(2)}_{i+1}(s_1,s_2,t_0)x(s_2),x(s_1)\rangle\,\diff s_1\!\,\diff s_2\\
	&\geq\alpha\int^T_{t_0}|x(t)|^2\,\diff t.
\end{split}
\end{equation}

Let $t_0\in[0,T)$ be fixed. For each $i\in\bN$, define a bounded linear operator $\cP^{t_0}_i$ on the Hilbert space $L^2(t_0,T;\bR^d)$ by
\begin{equation*}
	(\cP^{t_0}_ix)(t):=P^{(1)}_i(t)x(t)+\int^T_{t_0}P^{(2)}_i(t,s,t_0)x(s)\,\diff s,\ t\in(t_0,T),
\end{equation*}
for $x\in L^2(t_0,T;\bR^d)$. By \cref{lemm_self-adjoint}, $\cP^{t_0}_i$ is self-adjoint. The chain of inequalities \eqref{eq_P_i_inequality} implies that
\begin{equation*}
	\cP^{t_0}_1\geq\cP^{t_0}_i\geq\cP^{t_0}_{i+1}\geq\alpha I_{L^2(t_0,T;\bR^d)}
\end{equation*}
for any $t_0\in[0,T)$ and $i\in\bN$, where $I_{L^2(t_0,T;\bR^d)}$ denotes the identity operator on $L^2(t_0,T;\bR^d)$. Thus, for each $t_0\in[0,T)$, the sequence $\{\cP^{t_0}_i\}_{i\in\bN}$ is a bounded and monotone sequence of self-adjoint operators, and hence it is strongly convergent. In other words, for any $x\in L^2(t_0,T;\bR^d)$, the sequence $\{\cP^{t_0}_ix\}_{i\in\bN}$ is a convergent sequence on $L^2(t_0,T;\bR^d)$. Furthermore, for each $t_0\in[0,T)$ and $i\in\bN$, the operator norm $\|\cP^{t_0}_i\|_\op$ of $\cP^{t_0}_i$ is estimated as $\|\cP^{t_0}_i\|_\op\leq\max\{|\alpha|,\|\cP^{t_0}_1\|_\op\}$. Noting that
\begin{align*}
	\|\cP^{t_0}_1\|_\op&=\sup_{\substack{x\in L^2(t_0,T;\bR^d)\\\|x\|_{L^2(t_0,T)}\leq1}}|\langle\cP^{t_0}_1x,x\rangle_{L^2(t_0,T)}|\\
	&\leq\underset{t\in(0,T)}{\mathrm{ess\,sup}}|P^{(1)}_1(t)|+\Bigl(\int^T_0\!\!\int^T_0\sup_{t\in[0,s_1\wedge s_2]}|P^{(2)}_1(s_1,s_2,t)|^2\,\diff s_1\!\,\diff s_2\Bigr)^{1/2}<\infty,
\end{align*}
we obtain the uniform boundedness:
\begin{equation*}
	\sup_{t_0\in[0,T),i\in\bN}\|\cP^{t_0}_i\|_\op<\infty.
\end{equation*}

The above observations and the dominated convergence theorem yield the following assertions:
\begin{itemize}
\item
For any $M=(M_1,\dots,M_{d_1})\in L^2(\triangle_2(0,T);\bR^{d\times d_1})$ with $d_1\in\bN$, the sequence
\begin{equation*}
	(P_i\rint M)(s,t)=((\cP^t_iM_1(\cdot,t))(s),\dots,(\cP^t_iM_{d_1}(\cdot,t))(s)),\ (s,t)\in\triangle_2(0,T),\ i\in\bN,
\end{equation*}
converges in $L^2(\triangle_2(0,T);\bR^{d\times d_1})$. Similarly, the sequence $M^\top\lint P_i=(P_i\rint M)^\top$, $i\in\bN$, converges in $L^2(\triangle_2(0,T);\bR^{d_1\times d})$.
\item
For any $M=(M_1,\dots,M_{d_1})\in \sL^2(\triangle_2(0,T);\bR^{d\times d_1})$ and $N=(N_1,\dots,N_{d_2})\in \sL^2(\triangle_2(0,T);\bR^{d\times d_2})$ with $d_1,d_2\in\bN$, the sequence
\begin{equation*}
	(M^\top\lint P_i\rint N)(t)=(\langle M_k(\cdot,t),\cP^t_i N_\ell(\cdot,t)\rangle_{L^2(t,T)})_{k,\ell},\ t\in(0,T),\ i\in\bN,
\end{equation*}
converges for a.e.\ $t\in(0,T)$. Furthermore, $\sup_{i\in\bN}\|M^\top\lint P_i\rint N\|_{L^\infty(0,T)}<\infty$.
\end{itemize}
Therefore, noting that $R(t)+(D^\top\lint P_i\rint D)(t)\geq\lambda I_\ell$ for a.e.\ $t\in(0,T)$ and any $i\in\bN$, we have the following assertions:
\begin{itemize}
\item
The sequence
\begin{equation*}
	\Xi_i(t)=-(R(t)+(D^\top\lint P_{i-1}\rint D)(t))^{-1}(S(t)+(D^\top\lint P_{i-1}\rint C)(t)),\ t\in(0,T),\ i\geq2,
\end{equation*}
converges for a.e.\ $t\in(0,T)$. Furthermore, $\sup_{i\in\bN}\|\Xi_i\|_{L^\infty(0,T)}<\infty$.
\item
The sequence
\begin{equation*}
	\Gamma_i(s,t)=-(R(t)+(D^\top\lint P_{i-1}\rint D)(t))^{-1}(B^\top\lint P_{i-1})(s,t),\ (s,t)\in\triangle_2(0,T),\ i\geq2,
\end{equation*}
converges in $L^2(\triangle_2(0,T);\bR^{\ell\times d})$.
\item
The sequences
\begin{equation*}
	Q^{(1)}[\Xi_i](t)=Q(t)+\Xi_i(t)^\top S(t)+S(t)^\top\Xi_i(t)+\Xi_i(t)^\top R(t)\Xi_i(t),\ t\in(0,T),\ i\in\bN,
\end{equation*}
and
\begin{align*}
	&F^{(1)}[\Xi_i;P_i](t)=(C^\top\lint P_i\rint C)(t)+\Xi_i(t)^\top(D^\top\lint P_i\rint C)(t)+(C^\top\lint P_i\rint D)\Xi_i(t)\\
	&\hspace{4cm}+\Xi_i(t)^\top(D^\top\lint P_i\rint D)(t)\Xi_i(t),\ t\in(0,T),\ i\in\bN,
\end{align*}
converge for a.e.\ $t\in(0,T)$. Furthermore, the uniform estimates $\sup_{i\in\bN}\|Q^{(1)}[\Xi_i]\|_{L^\infty(0,T)}<\infty$ and $\sup_{i\in\bN}\|F^{(1)}[\Xi_i;P_i]\|_{L^\infty(0,T)}<\infty$ hold.
\item
The sequences
\begin{equation*}
	Q^{(2)}[\Xi_i,\Gamma_i](s,t)=\Gamma_i(s,t)^\top S(t)+\Gamma_i(s,t)^\top R(t)\Xi_i(t),\ (s,t)\in\triangle_2(0,T),\ i\in\bN,
\end{equation*}
and
\begin{align*}
	&F^{(2)}[\Xi_i,\Gamma_i;P_i](s,t)=(P_i\rint A)(s,t)+(P_i\rint B)(s,t)\Xi_i(t)+\Gamma_i(s,t)^\top(D^\top\lint P_i\rint C)(t)\\
	&\hspace{4cm}+\Gamma_i(s,t)^\top(D^\top\lint P_i\rint D)(t)\Xi_i(t),\ (s,t)\in\triangle_2(0,T),\ i\in\bN,
\end{align*}
converge in $L^2(\triangle_2(0,T);\bR^{d\times d})$.
\item
The sequences
\begin{equation*}
	Q^{(3)}[\Gamma_i](s_1,s_2,t)=\Gamma_i(s_1,t)^\top R(t)\Gamma_i(s_2,t),\ (s_1,s_2,t)\in\square_3(0,T),\ i\in\bN,
\end{equation*}
and
\begin{align*}
	&F^{(3)}[\Gamma_i;P_i](s_1,s_2,t)=\Gamma_i(s_1,t)^\top(B^\top\lint P_i)(s_2,t)+(P_i\rint B)(s_1,t)\Gamma_i(s_2,t)\\
	&\hspace{4cm}+\Gamma_i(s_1,t)^\top(D^\top\lint P_i\rint D)(t)\Gamma_i(s_2,t),\ (s_1,s_2,t)\in\square_3(0,T),\ i\in\bN,
\end{align*}
converge in $L^{2,2,1}_\sym(\square_3(0,T);\bR^{d\times d})$.
\end{itemize}
The above observations and the Lyapunov--Volterra equation \eqref{eq_Lyapunov--Volterra_i} yield that $\{P_i\}_{i\in\bN}$ converges to an element $P=(P^{(1)},P^{(2)})\in\Pi(0,T)$ in the following sense:
\begin{itemize}
\item
$\lim_{i\to\infty}P^{(1)}_i(t)=P^{(1)}(t)$ for a.e.\ $t\in(0,T)$, and $\sup_{i\in\bN}\|P^{(1)}_i\|_{L^\infty(0,T)}<\infty$;
\item
$\lim_{i\to\infty}(P^{(2)}_i(s_1,s_2,s_1\wedge s_2))_{(s_1,s_2)\in(0,T)^2}=(P^{(2)}(s_1,s_2,s_1\wedge s_2))_{(s_1,s_2)\in(0,T)^2}$ in $L^2((0,T)^2;\bR^{d\times d})$, and $\lim_{i\to\infty}\dot{P}^{(2)}_i=\dot{P}^{(2)}$ in $L^{2,2,1}_{\sym}(\square_3(0,T);\bR^{d\times d})$.
\end{itemize}
By the dominated convergence theorem, we see that
\begin{equation*}
	R(t)+(D^\top\lint P\rint D)(t)=\lim_{i\to\infty}(R(t)+(D^\top\lint P_i\rint D)(t))\geq\lambda I_\ell\ \text{for a.e.\ $t\in(0,T)$}.
\end{equation*}
Similarly, we have
\begin{align*}
	&\lim_{i\to\infty}\Xi_i(t)=-(R(t)+(D^\top\lint P\rint D)(t))^{-1}(S(t)+(D^\top\lint P\rint C)(t))=:\check{\Xi}(t)\ \text{for a.e.\ $t\in(0,T)$},\\
	&\lim_{i\to\infty}\Gamma_i(s,t)=-(R(t)+(D^\top\lint P\rint D)(t))^{-1}(B^\top\lint P)(s,t)=:\check{\Gamma}(s,t)\ \text{in $L^2(\triangle_2(0,T);\bR^{d\times d})$},\\
	&\lim_{i\to\infty}Q^{(1)}[\Xi_i](t)=Q^{(1)}[\check{\Xi}](t),\ \lim_{i\to\infty}F^{(1)}[\Xi_i;P_i](t)=F^{(1)}[\check{\Xi};P](t)\ \text{for a.e.\ $t\in(0,T)$},\\
	&\lim_{i\to\infty}Q^{(2)}[\Xi_i,\Gamma_i]=Q^{(2)}[\check{\Xi},\check{\Gamma}],\ \lim_{i\to\infty}F^{(2)}[\Xi_i,\Gamma_i;P_i]=F^{(2)}[\check{\Xi},\check{\Gamma};P]\ \text{in $L^2(\triangle_2(0,T);\bR^{d\times d})$},\\
	&\lim_{i\to\infty}Q^{(3)}[\Gamma_i]=Q^{(3)}[\check{\Gamma}],\ \lim_{i\to\infty}F^{(3)}[\Gamma_i;P_i]=F^{(3)}[\check{\Gamma};P]\ \text{in $L^{2,2,1}_\sym(\square_3(0,T);\bR^{d\times d})$}.
\end{align*}
Consequently, $P=(P^{(1)},P^{(2)})\in\Pi(0,T)$ satisfies the Lyapunov--Volterra equation
\begin{equation*}
	\begin{dcases}
	P^{(1)}(t)=F^{(1)}[\check{\Xi};P](t)+Q^{(1)}[\check{\Xi}](t),\ t\in(0,T),\\
	P^{(2)}(s,t,t)=P^{(2)}(t,s,t)^\top=F^{(2)}[\check{\Xi},\check{\Gamma};P](s,t)+Q^{(2)}[\check{\Xi},\check{\Gamma}](s,t),\ (s,t)\in\triangle_2(0,T),\\
	\dot{P}^{(2)}(s_1,s_2,t)+F^{(3)}[\check{\Gamma};P](s_1,s_2,t)+Q^{(3)}[\check{\Gamma}](s_1,s_2,t)=0,\ (s_1,s_2,t)\in\square_3(0,T).
	\end{dcases}
\end{equation*}
By inserting the formulae
\begin{align*}
	&\check{\Xi}(t)=-(R(t)+(D^\top\lint P\rint D)(t))^{-1}(S(t)+(D^\top\lint P\rint C)(t)),\ t\in(0,T),\\
	&\check{\Gamma}(s,t)=-(R(t)+(D^\top\lint P\rint D)(t))^{-1}(B^\top\lint P)(s,t),\ (s,t)\in\triangle_2(0,T),
\end{align*}
into the above Lyapunov--Volterra equation, we see that $P$ satisfies the Riccati--Volterra equation \eqref{eq_Riccati--Volterra}. Furthermore, since $R(t)+(D^\top\lint P\rint D)(t)\geq\lambda I_\ell$ for a.e.\ $t\in(0,T)$ with $\lambda>0$, the solution is strongly regular. This completes the proof.
\end{proof}


\begin{rem}
The above proof shows that the sequence of the solutions of the Lyapunov--Volterra equations \eqref{eq_Lyapunov--Volterra_i} converges to the strongly regular solution of the Riccati--Volterra equation \eqref{eq_Riccati--Volterra}. This fact is useful in view of the numerical approximations of the (unique) causal feedback optimal strategy and the value functional.
\end{rem}

Combining \cref{cor_unique_optimal_strategy}, \cref{cor_standard_condition} and \cref{theo_strongly_regular_solvability}, we get the following consequence.


\begin{cor}\label{cor_standard_optimal}
Assume that the standard condition \eqref{eq_standard_condition} holds. Then the Riccati--Volterra equation \eqref{eq_Riccati--Volterra} admits a unique strongly regular solution $P=(P^{(1)},P^{(2)})\in\Pi(0,T)$. Consequently, Problem (SVC) has a unique causal feedback optimal strategy $(\hat{\Xi},\hat{\Gamma},\hat{v})\in\cS(0,T)$ given by \eqref{eq_unique_optimal_strategy}.
\end{cor}


\appendix

\section{Appendix: Auxiliary lemmas}\label{appendix}

In this appendix, we prove some auxiliary lemmas used in Sections \ref{section_optimality} and \ref{section_strongly_regular_solvability}.


\begin{lemm}\label{lemm_strategy_bijective}
Let $(\Xi,\Gamma)\in L^\infty(0,T;\bR^{\ell\times d})\times L^2(\triangle_2(0,T);\bR^{\ell\times d})$ be fixed. Then the map $v\mapsto(\Xi,\Gamma,v)^0[0,0]$ is a bijective bounded linear operator on $\cU(0,T)$.
\end{lemm}


\begin{proof}
By the uniqueness of the causal feedback solutions of homogeneous controlled SVIEs \eqref{eq_state_0}, together with the a priori estimate \eqref{eq_closed-loop_estimate}, we see that the map $v\mapsto(\Xi,\Gamma,v)^0[0,0]$ is a bounded linear operator on $\cU(0,T)$. It is easy to see that the map $u\mapsto v:=(u(t)-\Xi(t)X(t)-\int^T_t\Gamma(s,t)\Theta(s,t))_{t\in(0,T)}$ with $X$ and $\Theta$ being defined as the solution to the SVIE
\begin{equation*}
	X(t)=\int^t_{0}\{A(t,s)X(s)+B(t,s)u(s)\}\,\diff s+\int^t_{0}\{C(t,s)X(s)+D(t,s)u(s)\}\,\diff W(s),\ t\in(0,T),
\end{equation*}
and
\begin{equation*}
	\Theta(s,t)=\int^t_{0}\{A(s,r)X(r)+B(s,r)u(r)\}\,\diff r+\int^t_{0}\{C(s,r)X(r)+D(s,r)u(r)\}\,\diff W(r),\ (s,t)\in\triangle_2(0,T),
\end{equation*}
is the inverse map of $v\mapsto(\Xi,\Gamma,v)^0[0,0]$. This completes the proof.
\end{proof}


\begin{lemm}\label{lemm_expectation_explicit}
For each $(\Xi,\Gamma)\in L^\infty(0,T;\bR^{\ell\times d})\times L^2(\triangle_2(0,T);\bR^{\ell\times d})$, there exist two measurable maps $G_1:\{(s,t,\theta)\in(0,T)^3\,|\,T>s\geq t\geq \theta>0\}\to\bR^{d\times d}$ and $G_2:\{(s,t,\theta)\in(0,T)^3\,|\,T>s\geq t\geq \theta>0\}\to\bR^{d\times \ell}$ depending on $A,B,C,D,\Xi,\Gamma$ such that the following conditions hold:
\begin{itemize}
\item
For $i=1,2$, $[\theta,s]\ni t\mapsto G_i(s,t,\theta)$ is continuous for a.e.\ $(s,\theta)\in\triangle_2(0,T)$, and the following estimate holds:
\begin{equation}\label{eq_G_i_estimate}
	\int^T_0\!\!\int^s_0\sup_{t\in[\theta,s]}|G_i(s,t,\theta)|^2\,\diff \theta\,\diff s<\infty.
\end{equation}
\item
For each $v\in\cU(0,T)$ and $(t_0,x)\in\cI$, the causal feedback solution $(X^{t_0,x}_0,\Theta^{t_0,x}_0)\in L^2_\bF(t_0,T;\bR^d)\times L^2_{\bF,\mathrm{c}}(\triangle_2(t_0,T);\bR^d)$ to the homogeneous controlled SVIE \eqref{eq_state_0} at the input condition $(t_0,x)$ corresponding to the causal feedback strategy $(\Xi,\Gamma,v)\in\cS(0,T)$ satisfies
\begin{equation}\label{eq_deterministic_VCF}
\begin{split}
	&\bE[X^{t_0,x}_0(t)]=x(t)+\int^t_{t_0}\{G_1(t,t,\theta)x(\theta)+G_2(t,t,\theta)\bE[v(\theta)]\}\,\diff\theta\\
	&\hspace{2cm}+\int^T_{t_0}\!\!\int^{t\wedge\theta}_{t_0}G_2(t,t,r)\Gamma(\theta,r)\,\diff r\,x(\theta)\,\diff\theta,\ t\in(t_0,T),\\
	&\bE[\Theta^{t_0,x}_0(s,t)]=x(s)+\int^t_{t_0}\{G_1(s,t,\theta)x(\theta)+G_2(s,t,\theta)\bE[v(\theta)]\}\,\diff\theta\\
	&\hspace{2cm}+\int^T_{t_0}\!\!\int^{t\wedge\theta}_{t_0}G_2(s,t,r)\Gamma(\theta,r)\,\diff r\,x(\theta)\,\diff\theta,\ (s,t)\in\triangle_2(t_0,T).
\end{split}
\end{equation}
\end{itemize}
\end{lemm}


\begin{proof}
Let $(t_0,x)\in\cI$ be fixed, and let $(X^{t_0,x}_0,\Theta^{t_0,x}_0,u^{t_0,x}_0)\in L^2_\bF(t_0,T;\bR^d)\times L^2_{\bF,\mathrm{c}}(\triangle_2(t_0,T);\bR^d)\times U(t_0,T)$ be the causal feedback solution to the homogeneous controlled SVIE \eqref{eq_state_0} at $(t_0,x)$ corresponding to $(\Xi,\Gamma,v)\in\cS(0,T)$. Denote $\bar{v}:=\bE[v]$ and $(\bar{X},\bar{\Theta},\bar{u}):=(\bE[X^{t_0,x}_0],\bE[\Theta^{t_0,x}_0],\bE[u^{t_0,x}_0])$. By \cite[Lemma A.4]{HaWa1}, we have
\begin{equation*}
	\begin{dcases}
	\begin{pmatrix}\bar{X}(t)\\\bar{u}(t)\end{pmatrix}=\bar{\bX}(t),\ t\in(t_0,T),\\
	\bar{\Theta}(s,t)=x(s)+\int^t_{t_0}(A(s,r),B(s,r))\bar{\bX}(r)\,\diff r,\ (s,t)\in\triangle(t_0,T),
	\end{dcases}
\end{equation*}
where $\bar{\bX}\in L^2(t_0,T;\bR^{d+\ell})$ is the solution to the following deterministic Volterra equation:
\begin{equation*}
	\bar{\bX}(t)=\begin{pmatrix}x(t)\\x^{\Xi,\Gamma}(t)+\bar{v}(t)\end{pmatrix}+\int^t_{t_0}\bA(t,s)\bar{\bX}(s)\,\diff s,\ t\in(t_0,T).
\end{equation*}
Here, $x^{\Xi,\Gamma}\in L^2(t_0,T;\bR^\ell)$ and $\bA\in L^2(\triangle_2(0,T);\bR^{(d+\ell)\times(d+\ell)})$ are defined by
\begin{equation*}
	x^{\Xi,\Gamma}(t):=\Xi(t)x(t)+\int^T_t\Gamma(s,t)x(s)\,\diff s,\ t\in(t_0,T),
\end{equation*}
and
\begin{equation*}
	\bA(t,s):=\begin{pmatrix}A(t,s)&B(t,s)\\\Xi(t)A(t,s)+\int^T_t\Gamma(r,t)A(r,s)\,\diff r&\Xi(t)B(t,s)+\int^T_t\Gamma(r,t)B(r,s)\,\diff r\end{pmatrix},\ (t,s)\in\triangle_2(0,T),
\end{equation*}
respectively. From the general theory on deterministic Volterra equations (see \cite[Chapter 9]{GrLoSt90} or \cite{Ha21+++}), the kernel $\bA$ has a resolvent $\bF$ in $L^2(\triangle_2(0,T);\bR^{(d+\ell)\times(d+\ell)})$, and $\bar{\bX}$ is given by the variation of constants formula:
\begin{equation*}
	\bar{\bX}(t)=\begin{pmatrix}x(t)\\x^{\Xi,\Gamma}(t)+\bar{v}(t)\end{pmatrix}+\int^t_{t_0}\bF(t,s)\begin{pmatrix}x(s)\\x^{\Xi,\Gamma}(s)+\bar{v}(s)\end{pmatrix}\,\diff s,\ t\in(t_0,T).
\end{equation*}
Thus, we have
\begin{align*}
	\bar{\Theta}(s,t)&=x(s)+\int^t_{t_0}(A(s,r),B(s,r))\bar{\bX}(r)\,\diff r\\
	&=x(s)+\int^t_{t_0}(A(s,r),B(s,r))\Bigl\{\begin{pmatrix}x(r)\\x^{\Xi,\Gamma}(r)+\bar{v}(r)\end{pmatrix}+\int^r_{t_0}\bF(r,\theta)\begin{pmatrix}x(\theta)\\x^{\Xi,\Gamma}(\theta)+\bar{v}(\theta)\end{pmatrix}\,\diff\theta\Bigr\}\,\diff r\\
	&=x(s)+\int^t_{t_0}\Bigl\{(A(s,\theta),B(s,\theta))+\int^t_\theta(A(s,r),B(s,r))\bF(r,\theta)\,\diff r\Bigr\}\begin{pmatrix}x(\theta)\\x^{\Xi,\Gamma}(\theta)+\bar{v}(\theta)\end{pmatrix}\,\diff\theta\\
	&=x(s)+\int^t_{t_0}(\tilde{G}_1(s,t,\theta),\tilde{G}_2(s,t,\theta))\begin{pmatrix}x(\theta)\\x^{\Xi,\Gamma}(\theta)+\bar{v}(\theta)\end{pmatrix}\,\diff\theta
\end{align*}
for $(s,t)\in\triangle_2(t_0,T)$, where $\tilde{G}_i$, $i=1,2$, are defined by
\begin{align*}
	&\tilde{G}_1(s,t,\theta):=A(s,\theta)+\int^t_\theta\{A(s,r)F_{1,1}(r,\theta)+B(s,r)F_{2,1}(r,\theta)\}\,\diff r,\\
	&\tilde{G}_2(s,t,\theta):=B(s,\theta)+\int^t_\theta\{A(s,r)F_{1,2}(r,\theta)+B(s,r)F_{2,2}(r,\theta)\}\,\diff r,
\end{align*}
for $T>s\geq t\geq\theta>0$, with the notation $\bF(t,s)=\begin{pmatrix}F_{1,1}(t,s)&F_{1,2}(t,s)\\F_{2,1}(t,s)&F_{2,2}(t,s)\end{pmatrix}$, $(t,s)\in\triangle_2(0,T)$. Define
\begin{equation*}
	G_1(s,t,\theta):=\tilde{G}_1(s,t,\theta)+\tilde{G}_2(s,t,\theta)\Xi(\theta),\ G_2(s,t,\theta):=\tilde{G}_2(s,t,\theta)
\end{equation*}
For $T>s\geq t\geq\theta>0$. It is easy to see that the estimates \eqref{eq_G_i_estimate} hold for $i=1,2$, and the representation of $\bar{\Theta}=\bE[\Theta^{t_0,x}_0]$ in \eqref{eq_deterministic_VCF} holds. Noting that $\bE[X^{t_0,x}_0(t)]=\bar{X}(t)=\bar{\Theta}(t,t)$, we get the representation for $\bE[X^{t_0,x}_0]$.
\end{proof}


\begin{lemm}\label{lemm_MN=0}
Let $(\Xi,\Gamma)\in L^\infty(0,T;\bR^{\ell\times d})\times L^2(\triangle_2(0,T);\bR^{\ell\times d})$ be fixed. Let $M_1\in L^\infty(0,T;\bR^{\ell'\times d})$ and $M_2\in L^2(\triangle_2(0,T);\bR^{\ell'\times d})$ with $\ell'\in\bN$. Assume that, for any $(t_0,x)\in\cI$,
\begin{equation}\label{eq_MN=0}
	M_1(t)\bE[X^{t_0,x}_0(t)]+\int^T_tM_2(s,t)\bE[\Theta^{t_0,x}_0(s,t)]\,\diff s=0\ \text{for a.e.\ $t\in(t_0,T)$},
\end{equation}
where $(X^{t_0,x}_0,\Theta^{t_0,x}_0)$ is the causal feedback solution to the homogeneous controlled SVIE \eqref{eq_state_0} at the input condition $(t_0,x)$ corresponding to the causal feedback strategy $(\Xi,\Gamma,0)\in\cS(0,T)$. Then it holds that
\begin{equation*}
	M_1(t)=0\ \text{for a.e.\ $t\in(0,T)$ and}\ M_2(s,t)=0\ \text{for a.e.\ $(s,t)\in\triangle_2(0,T)$}.
\end{equation*}
\end{lemm}


\begin{proof}
First, we prove that $M_1(t)=0$ for a.e.\ $t\in(0,T)$. By the assumption, together with \cref{lemm_expectation_explicit}, for any $x\in L^2(0,T;\bR^d)$,
\begin{align*}
	0&=M_1(t)\Bigl\{x(t)+\int^t_0G_1(t,t,\theta)x(\theta)\,\diff\theta+\int^T_0\!\!\int^{t\wedge\theta}_0G_2(t,t,r)\Gamma(\theta,r)\,\diff r\,x(\theta)\,\diff\theta\Bigr\}\\
	&\hspace{0.5cm}+\int^T_tM_2(s,t)\Bigl\{x(s)+\int^t_0G_1(s,t,\theta)x(\theta)\,\diff\theta+\int^T_0\!\!\int^{t\wedge\theta}_0G_2(s,t,r)\Gamma(\theta,r)\,\diff r\,x(\theta)\,\diff\theta\Bigr\}\,\diff s\\
	&=M_1(t)x(t)+\int^T_0\cM(t,\theta)x(\theta)\,\diff \theta
\end{align*}
for a.e.\ $t\in(0,T)$, where
\begin{align*}
	\cM(t,\theta)&:=\Bigl\{M_1(t)G_1(t,t,\theta)+\int^T_tM_2(s,t)G_1(s,t,\theta)\,\diff s\Bigr\}\1_{\triangle_2(0,T)}(t,\theta)+M_2(\theta,t)\1_{\triangle_2(0,T)}(\theta,t)\\
	&\hspace{0.5cm}+\int^{t\wedge \theta}_0\Bigl\{M_1(t)G_2(t,t,r)+\int^T_tM_2(s,t)G_2(s,t,r)\,\diff s\Bigr\}\Gamma(\theta,r)\,\diff r,\ (t,\theta)\in(0,T)^2.
\end{align*}
Noting the estimate \eqref{eq_G_i_estimate}, it is easy to see that $\cM\in L^2((0,T)^2;\bR^{\ell'\times d})$. Considering $x(t)=N\1_{[\tau,\tau+1/N]}(t)x$, $t\in(0,T)$, with $\tau\in(0,T)$, $N\in\bN$ with $\tau+1/N<T$ and $x\in\bR^d$ being arbitrary, and then integrating with respect to $t\in(\tau,\tau+1/N)$, we get
\begin{equation*}
	N\int^{\tau+1/N}_\tau M_1(t)x\,\diff t+N\int^{\tau+1/N}_\tau\!\!\int^{\tau+1/N}_\tau \cM(t,\theta)x\,\diff\theta\,\diff t=0,
\end{equation*}
and hence
\begin{equation*}
	N\int^{\tau+1/N}_\tau M_1(t)\,\diff t+N\int^{\tau+1/N}_\tau\!\!\int^{\tau+1/N}_\tau \cM(t,\theta)\,\diff\theta\,\diff t=0.
\end{equation*}
By the Lebesgue differentiation theorem, we have $\lim_{N\to\infty}N\int^{\tau+1/N}_\tau M_1(t)\,\diff t=M_1(\tau)$ for a.e.\ $\tau\in(0,T)$. On the other hand, for any $\tau\in(0,T)$,
\begin{equation*}
	N\int^{\tau+1/N}_\tau\!\!\int^{\tau+1/N}_\tau |\cM(t,\theta)|\,\diff\theta\,\diff t\leq\Bigl(\int^{\tau+1/N}_\tau\!\!\int^{\tau+1/N}_\tau |\cM(t,\theta)|^2\,\diff\theta\,\diff t\Bigr)^{1/2}\to0
\end{equation*}
as $N\to\infty$. Therefore, we get $M_1(t)=0$ for a.e.\ $t\in(0,T)$.

Next, we show that $M_2(s,t)=0$ for a.e.\ $(s,t)\in\triangle_2(0,T)$. By the assumption \eqref{eq_MN=0}, together with \cref{lemm_expectation_explicit} and $M_1=0$, for any $(t_0,x)\in\cI$,
\begin{equation*}
	\int^T_tM_2(s,t)\Bigl\{x(s)+\int^t_{t_0}G_1(s,t,\theta)x(\theta)\,\diff\theta+\int^T_{t_0}\!\!\int^{t\wedge\theta}_{t_0}G_2(s,t,r)\Gamma(\theta,r)\,\diff r\,x(\theta)\,\diff\theta\Bigr\}\,\diff s=0
\end{equation*}
for a.e.\ $t\in(t_0,T)$. Let $x\in\bR^d$, $(\tau_1,\tau_2)\in\triangle_2(0,T)$, and $N\in\bN$ with $1/N<\min\{T-\tau_1,\tau_1-\tau_2\}$ be fixed. Consider $(t_0,x)\in\cI$ with $t_0=\tau_2$ and $x(t)=\1_{[\tau_1,\tau_1+1/N]}(t)x$, $t\in(\tau_2,T)$. Then for a.e.\ $t\in(\tau_2,\tau_2+1/N)$,
\begin{equation*}
	\Bigl\{\int^{\tau_1+1/N}_{\tau_1}M_2(s,t)\,\diff s+\int^T_tM_2(s,t)\int^{\tau_1+1/N}_{\tau_1}\!\!\int^t_{\tau_2}G_2(s,t,r)\Gamma(\theta,r)\,\diff r\,\diff\theta\,\diff s\Bigr\}x=0,
\end{equation*}
and thus
\begin{equation*}
	\int^{\tau_1+1/N}_{\tau_1}M_2(s,t)\,\diff s+\int^{\tau_1+1/N}_{\tau_1}\!\!\int^t_{\tau_2}\cN(t,r)\Gamma(\theta,r)\,\diff r\,\diff\theta=0,
\end{equation*}
where
\begin{equation*}
	\cN(t,r):=\int^T_tM_2(s,t)G_2(s,t,r)\,\diff s,\ (t,r)\in\triangle_2(0,T).
\end{equation*}
Noting the estimate \eqref{eq_G_i_estimate}, it is easy to see that $\cN\in L^2(\triangle_2(0,T);\bR^{\ell'\times d})$. Integrating both sides with respect to $t\in(\tau_2,\tau_2+1/N)$ and multiplying $N^2$, we get
\begin{equation*}
	N^2\int^{\tau_2+1/N}_{\tau_2}\!\!\int^{\tau_1+1/N}_{\tau_1}M_2(s,t)\,\diff s\!\,\diff t+N^2\int^{\tau_2+1/N}_{\tau_2}\!\!\int^{\tau_1+1/N}_{\tau_1}\int^t_{\tau_2}\cN(t,r)\Gamma(\theta,r)\,\diff r\,\diff\theta\,\diff t=0.
\end{equation*}
On one hand, by the Lebesgue differentiation theorem, we have
\begin{equation*}
	\lim_{N\to\infty}N^2\int^{\tau_2+1/N}_{\tau_2}\!\!\int^{\tau_1+1/N}_{\tau_1}M_2(s,t)\,\diff s\!\,\diff t=M_2(\tau_1,\tau_2)
\end{equation*}
for a.e.\ $(\tau_1,\tau_2)\in\triangle_2(0,T)$. On the other hand, for any $(\tau_1,\tau_2)\in\triangle_2(0,T)$,
\begin{align*}
	&\Bigl|N^2\int^{\tau_2+1/N}_{\tau_2}\!\!\int^{\tau_1+1/N}_{\tau_1}\int^t_{\tau_2}\cN(t,r)\Gamma(\theta,r)\,\diff r\,\diff\theta\,\diff t\Bigr|\\
	&\leq N^2\int^{\tau_2+1/N}_{\tau_2}\!\!\int^{\tau_1+1/N}_{\tau_1}\Bigl(\int^t_{\tau_2}|\cN(t,r)|^2\,\diff r\Bigr)^{1/2}\Bigl(\int^{\tau_2+1/N}_{\tau_2}|\Gamma(\theta,r)|^2\,\diff r\Bigr)^{1/2}\,\diff\theta\,\diff t\\
	&=N^2\int^{\tau_2+1/N}_{\tau_2}\Bigl(\int^t_{\tau_2}|\cN(t,r)|^2\,\diff r\Bigr)^{1/2}\,\diff t\,\int^{\tau_1+1/N}_{\tau_1}\Bigl(\int^{\tau_2+1/N}_{\tau_2}|\Gamma(\theta,r)|^2\,\diff r\Bigr)^{1/2}\,\diff\theta\\
	&\leq\Bigl(\int^{\tau_2+1/N}_{\tau_2}\!\!\int^t_{\tau_2}|\cN(t,r)|^2\,\diff r\!\,\diff t\Bigr)^{1/2}\,\Bigl(N^2\int^{\tau_1+1/N}_{\tau_1}\!\!\int^{\tau_2+1/N}_{\tau_2}|\Gamma(\theta,r)|^2\,\diff r\!\,\diff\theta\Bigr)^{1/2}.
\end{align*}
We have
\begin{equation*}
	\lim_{N\to\infty}\int^{\tau_2+1/N}_{\tau_2}\!\!\int^t_{\tau_2}|\cN(t,r)|^2\,\diff r\!\,\diff t=0
\end{equation*}
for any $\tau_2\in(0,T)$ and, by the Lebesgue differentiation theorem,
\begin{equation*}
	\lim_{N\to\infty}N^2\int^{\tau_1+1/N}_{\tau_1}\!\!\int^{\tau_2+1/N}_{\tau_2}|\Gamma(\theta,r)|^2\,\diff r\!\,\diff\theta=|\Gamma(\tau_1,\tau_2)|^2<\infty
\end{equation*}
for a.e.\ $(\tau_1,\tau_2)\in\triangle_2(0,T)$. Therefore,
\begin{equation*}
	\lim_{N\to\infty}N^2\int^{\tau_2+1/N}_{\tau_2}\!\!\int^{\tau_1+1/N}_{\tau_1}\int^t_{\tau_2}\cN(t,r)\Gamma(\theta,r)\,\diff r\,\diff\theta\,\diff t=0
\end{equation*}
for a.e.\ $(\tau_1,\tau_2)\in\triangle_2(0,T)$. Consequently, we have $M_2(\tau_1,\tau_2)=0$ for a.e.\ $(\tau_1,\tau_2)\in\triangle_2(0,T)$. This completes the proof.
\end{proof}


\end{document}